\newif\ifmarek
\marektrue

 \title{\normalsize
{\textbf     {LIMIT THEORY FOR RANDOM WALKS IN DEGENERATE TIME-DEPENDENT RANDOM ENVIRONMENTS }}}

\date{}

\documentclass[11pt]{article}
\usepackage[english]{babel}
\usepackage{amsmath}
\usepackage{amssymb}
\usepackage{amsfonts}
\usepackage{amsthm}
\usepackage[
left=3.0cm,right=3.0cm,top=2.5cm,bottom=3.5cm]{geometry}
\usepackage{amsfonts, amsmath, amssymb, amsthm, mathrsfs}
\usepackage{graphicx}
\usepackage{color}
\usepackage[sans]{dsfont}

\usepackage{color}

\usepackage{comment}
\usepackage[hang,flushmargin]{footmisc}

\numberwithin{equation}{section}

\usepackage{titlesec}

\titleformat{\subsection}
 {\normalfont\fontsize{12}{15}\bfseries}{\thesubsection}{1em}{}

\usepackage{lipsum}

\let\OLDthebibliography\thebibliography
\renewcommand\thebibliography[1]{
  \OLDthebibliography{#1}
  \setlength{\parskip}{0pt}
  \setlength{\itemsep}{0pt plus 0.3ex}
}

\newtheoremstyle{thm}{1.5ex}{1.5ex}{\itshape\rmfamily}{}
{\bfseries\rmfamily}{}{2ex}{}

\newtheoremstyle{def}{1.5ex}{1.5ex}{\slshape\rmfamily}{}
{\bfseries\rmfamily}{}{2ex}{}

\newtheoremstyle{rem}{1.3ex}{1.3ex}{\rmfamily}{}
{\itshape}
{} {1.5ex}{}

\newenvironment{proofsect}[1]
{\vskip0.1cm\noindent{\rmfamily\itshape#1.}}{\qed\vspace{0.15cm}}

\theoremstyle{thm}
\newtheorem{theorem}{Theorem}[section]
\newtheorem{lemma}[theorem]{Lemma}
\newtheorem{proposition}[theorem]{Proposition}

\newtheorem{assumption}[theorem]{Assumption}
\newtheorem{corollary}[theorem]{Corollary}
\newtheorem{problem}[theorem]{Problem}

\theoremstyle{def}
\newtheorem{definition}[theorem]{Definition}

\theoremstyle{rem}
\newtheorem{remark}[theorem]{\itshape Remark}

\renewcommand{\theequation}{\arabic{section}.\arabic{equation}}

\renewcommand{\subsection}{\secdef\subsct\sbsect}
\newcommand{\subsct}[2][default]{\refstepcounter{subsection}
\nopagebreak\vspace{0.45\baselineskip} {\flushleft\bf
\thesection.\arabic{subsection}~\bf #1.~}
\\*[3mm]\noindent
\nopagebreak}
\newcommand{\sbsect}[1]{\vspace{0.1cm}\noindent
\textbf{#1.~}\vspace{0.1cm}}

\newcommand{\supp}{\operatorname{supp}}

\newcommand{\textd}{\text{\rm d}\mkern0.5mu}

\newcommand{\EE}{\mathcal E}
\newcommand{\FF}{\mathcal F}

\newcommand{\E}{\mathbb E}

\newcommand{\N}{\mathbb N}

\newcommand{\BbbP}{\mathbb P}

\newcommand{\R}{\mathbb R}

\newcommand{\Z}{\mathbb Z}

\newcommand{\twoeqref}[2]{{\rm (\ref{#1}--\ref{#2})}}
\newcommand{\1}{{1\mkern-4.5mu\textrm{l}}}
\renewcommand{\1}{\text{\sf 1}}
\newcommand{\e}{{\text{\rm e}}}

\newcommand{\hate}{e}

\newcommand{\czero}{c_0}
\newcommand{\ci}{c_1}
\newcommand{\cii}{c_2}
\newcommand{\ciii}{c_3}
\newcommand{\civ}{c_4}
\newcommand{\cv}{c_5}
\newcommand{\cvi}{c_6}
\newcommand{\cvii}{c_7}
\newcommand{\cviii}{c_8}
\newcommand{\cix}{c_9}
\newcommand{\cx}{c_{10}}
\newcommand{\cxi}{c_{11}}

\newcommand{\vertiii}[1]{{\left\vert\kern-0.25ex\left\vert\kern-0.25ex\left\vert #1 
    \right\vert\kern-0.25ex\right\vert\kern-0.25ex\right\vert}}

\makeatletter
\newcommand{\@fatnormbar}{\vrule\@width 1.2\p@}
\newcommand{\fatnorm}{\@ifstar\@xfatnorm\@fatnorm}
\newcommand{\@xfatnorm}[1]{%
  \left.\kern-\nulldelimiterspace
  \@fatnormbar
  #1
  \@fatnormbar
  \right.\kern-\nulldelimiterspace
}
\newcommand\@fatnorm[2][]{%
  \mathopen{\vphantom{#1|}\mkern2mu\@fatnormbar\mkern2mu}
  #2
  \mathclose{\vphantom{#1|}\mkern2mu\@fatnormbar\mkern2mu}
}
\makeatother

\newcommand{\bigdot}[1]{\overset{\text{\vbox to 0pt{\hbox{\bf.}}}}{#1}}
\newcommand{\vvert}{\vert\mkern-2.6mu\vert\mkern-2.6mu\vert}
\newcommand{\bvvert}{\big\vert\mkern-3.5mu\big\vert\mkern-3.5mu\big\vert}
\newcommand{\cc}{{\text{\rm c}}}

\begin{document}

\maketitle

\begin{center}
\vspace{-1cm}
Marek Biskup$^{1,2}$ and Pierre-Fran\c cois Rodriguez$^1$ 

\end{center}
\vspace{0.5cm}
\begin{abstract}
\vspace{0.5cm}
\noindent We study continuous-time (variable speed) random walks in random environments on~$\Z^d$, $d\ge2$, where, at time~$t$, the walk at~$x$ jumps across edge~$(x,y)$ at time-dependent rate~$a_t(x,y)$. The rates, which we assume stationary and ergodic with respect to space-time shifts, are symmetric and bounded but possibly degenerate in the sense that the total jump rate from a vertex may vanish over finite intervals of time. We formulate conditions on the environment under which the law of diffusively-scaled random walk path tends to Brownian motion for almost every sample of the rates. The proofs invoke Moser iteration to prove sublinearity of the corrector in pointwise sense; a key additional input is a conversion of certain weighted energy norms to ordinary ones. Our conclusions apply to random walks on dynamical bond percolation and interacting particle systems as well as to random walks arising from the Helffer-Sj\"ostrand representation of gradient models with certain non-strictly convex potentials. 
\end{abstract}

\thispagestyle{empty}

\vspace{3.5cm}

\begin{flushleft}

$^1$Department of Mathematics \hfill December 18, 2017 \\
University of California, Los Angeles \\
Los Angeles, CA, USA \\

\vspace{0.5cm}

$^2$Center for Theoretical Study\\
Charles University \hfill\texttt{biskup@math.ucla.edu} \\
Prague, Czech Republic \hfill\texttt{rodriguez@math.ucla.edu}
\end{flushleft}



\vspace{1cm}

\noindent\fontsize{9.6}{9.6}\selectfont\copyright\,\textrm{2017}\ \textrm{M.~Biskup, P.-F.~Rodriguez.
Reproduction, by any means, of the entire
article for non-commercial purposes is permitted without charge.}

\ifmarek\else
\newpage\fi
\normalsize	
\mbox{}
\thispagestyle{empty}

\newpage

\section{Introduction}

\subsection{Model and assumptions}
The aim of this note is to study the long-time behavior of random walks on~$\Z^d$, $d\ge2$, in a class of dynamical random environments given as a family of non-negative random variables
\begin{equation}
\label{E:1.1a}
\bigl\{a_t(e)\colon e\in E(\Z^d),\, t\in\R\bigr\}\,,
\end{equation}
where~$E(\Z^d)$ denotes the set of (unordered) nearest-neighbor edges of~$\Z^d$. For each sample of these random variables, referred to as conductances, we consider the continuous time Markov chain  $\{X_t\colon t\ge0\}$ on~$\Z^d$ with the instantaneous generator~$L_t$ acting on functions~$f\colon\Z^d\to\R$~as
\begin{equation}
\label{E:1.1}
L_tf(x):=\sum_{y\colon |y-x|=1}a_t(x,y)\bigl[f(y)-f(x)\bigr].
\end{equation}
The variable $a_t(e)=a_t(x,y)$, i.e., the jump rate of the walk across edge~$e=(x,y)$ at time~$t$, is assumed to obey $a_t(e)\in[0,1]$ with $a_t(e)=0$ allowed for non-trivial finite intervals of time.
Our aim is to describe the long-time behavior of such random walks and, in particular, to show that their path distribution, scaled diffusively, tends to a non-degenerate Brownian motion.

A representative example of the above setting is the variable-speed random walk on dynamical bond percolation on~$\Z^d$. In this case $a_t(e)$ is, for each~$e\in E(\Z^d)$, an independent copy of a stationary continuous-time process on~$\{0,1\}$ with joint invariant distribution (product) Bernoulli($p$) for some prescribed~$p\in(0,1)$. We interpret $a_t(e)=1$ as the event that edge~$e$ is occupied at time~$t$ and $a_t(e)=0$ as the event that edge~$e$ is vacant. The random walk then jumps at rate 1 across edges incident with its current position that are occupied at that instant of time. When the site where the walk is located has no incident occupied edges, the walk does not move. 

It is clear that some mixing properties of the conductances \eqref{E:1.1a} in both space and time are required for the desired convergence to Brownian motion to be possible. We will work under the following set of technical assumptions:

\begin{assumption}
\label{ass1}
The family $\{a_t(e)\colon e\in E(\Z^d),\, t\in\R\}$ is realized as coordinate projections on the product space $\Omega:=[0,\infty)^{\R\times E(\Z^d)}$ endowed with the product Borel $\sigma$-algebra~$\FF$ and the probability distribution denoted by~$\BbbP$. In addition, we assume: 
\begin{enumerate}
\item[(1)] $t\mapsto a_t(e)$ obeys
\begin{equation}
a_t(e)\in[0,1]
\end{equation}
for each~$e\in E(\Z^d)$ and each $t\in\R$,
\item[(2)] letting $\tau_{s,x}\colon\Omega\to\Omega$ denote the map
\begin{equation}
(\tau_{s,x}a)_t(y,z):=a_{t+s}(y+x,z+x),\qquad (y,z)\in E(\Z^d),\, t\in\R,
\end{equation}
the law $\BbbP$ is invariant and jointly ergodic under $\{\tau_{t,x}\colon t\in\R,\,x\in\Z^d\}$,
\item[(3)] denoting, for each $e\in E(\Z^d)$,
\begin{equation}
\label{E:1.5}
T_e:=\inf\Bigl\{t\ge0\colon \int_0^t\textd s\,\, a_s(e)\ge1\Bigr\}
\end{equation}
we have $T_e<\infty$, $\BbbP$\text{\rm-a.s.}
\end{enumerate}
We will write~$\E$ to denote expectation with respect to~$\BbbP$.
\end{assumption}

\noindent 
We remark that joint ergodicity in~(2) means that any measurable subset of~$\Omega$ preserved by~$\tau_{t,x}$ for all $t\in\R$ and~$x\in\Z^d$ is a zero-one event under~$\BbbP$.
The restriction to conductances bounded by~$1$ is only a matter of convenience; any uniform constant upper bound will suffice (and ensure that~$X$ is non-explosive). Additional moment conditions on~$T_e$ will need to be assumed in the statement of our main result. However, no assumptions will be made on the dynamics of the conductances and/or the law of its time reversal (which is stationary but possibly unrelated to~$\BbbP$).

Besides dynamical percolation, the setting of Assumption~\ref{ass1} accommodates various other examples of interest. For instance, one can consider the random walk on the symmetric exclusion process $\{\eta_t(x)\colon x\in\Z^d\}$, where $\eta_t(x)$ is the indicator that site~$x$ is occupied by a particle at time~$t$ and the configuration~$t\mapsto\eta_t$ evolves by swaps $\eta_t(x)\leftrightarrow\eta_t(y)$ at endpoints~$x$ and~$y$ of edges in~$E(\Z^d)$ whenever an independent exponential clock rings at that edge. We then set, e.g.,
\begin{equation}
a_t(e):=c\eta_t(x)\eta_t(y)\quad\text{whenever}\quad e=(x,y)
\end{equation}
for some~$c>0$. The walk is thus active only at times when it resides on an occupied site and the transitions are only between occupied vertices. Other particle systems such as the voter model or the contact process can of course be considered~as~well.

Another interesting class of random walks arises in the context of Helffer-Sj\"ostrand representations of gradient models with convex, but not uniformly strictly convex, potentials~$V$. The representative examples covered by our theory include
\begin{equation}
\label{E:1.6a}
V(\eta):=\beta\log\cosh(\eta)
\end{equation}
with any~$\beta>0$, or even
\begin{equation}
\label{E:1.7a}
V(\eta):=\begin{cases}
\frac12|\eta|^2,\qquad&\text{if }|\eta|\le1,
\\
|\eta|-\frac12,\qquad&\text{else}.
\end{cases}
\end{equation}
In this case the random environment is a family of diffusions $\{\phi(x)\colon x\in\Z^d\}$ evolving according to the Langevin dynamics
\begin{equation}
\textd \phi_t(x) = \sum_{y\colon |y-x|=1}V'\bigl(\phi_t(y)-\phi_t(x)\bigr)\,\textd t+\sqrt 2\,\textd B_t(x)\,,
\end{equation}
where $\{B(x)\colon x\in\Z^d\}$ is a family of independent standard Brownian motions. The random walk jump rates are then given by
\begin{equation}
a_t(e):=V''\bigl(\phi_t(y)-\phi_t(x)\bigr)\quad\text{whenever}\quad e=(x,y).
\end{equation}
In both \eqref{E:1.6a} and \eqref{E:1.7a}, $a_t(e)$ is non-negative and bounded yet not bounded away from~zero.

\subsection{Main result}
In order to give a statement of our main result, we need some additional notation. Let $D([0,\infty))$ denote the space of c\`adl\`ag functions $\omega\colon[0,\infty)\to\R$ endowed (disregarding the standard notation for the Skorokhod space) with the norm
\begin{equation}
\Vert\omega\Vert_{D([0,\infty))}:=\sum_{n\ge1}2^{-n}\sup_{t\in[0,n]}|\omega(t)|\wedge1\,.
\end{equation}
The space of continuous functions $C([0,\infty))$, a set that supports the law of the Brownian motion, is naturally embedded in $D([0,\infty))$ and is, in fact, a closed (and thus measurable) subset thereof in the topology induced by the above norm.
Our main conclusion regarding the Markov chain~$\{X_t\colon t\ge0\}$ defined via \eqref{E:1.1} is as follows:

\begin{theorem}
\label{thm-1.2}
Let $d\ge2$ and suppose that Assumption~\ref{ass1} holds and, in addition, the quantity in \eqref{E:1.5} obeys
\begin{equation}
\label{E:q_ass}
\exists\vartheta>4d\colon\qquad\E(T_e^{\vartheta})<\infty,\quad e\in E(\Z^d).
\end{equation}
Then, for $\BbbP$-a.e.\ random environment, the law of $t\mapsto n^{-1/2}X_{tn}$ on $D([0,\infty))$ tends, as $n\to\infty$, to the law of Brownian motion $\{B_t\colon t\ge0\}$ with 
\begin{equation}
E(B_t)=0\quad\text{and}\quad
E((v\cdot B_t)^2)=v\cdot\Sigma v,\qquad v\in\R^d,
\end{equation}
where~$\Sigma=\{\Sigma_{ij}\}_{i,j=1}^d$ is  a non-degenerate (deterministic) covariance matrix.
\end{theorem}

We note that this is a \emph{quenched} statement (i.e., one for~$\BbbP$-a.e.\ environment). The corresponding \emph{annealed} (or \emph{averaged}) statement follows from the fact that the limiting covariance~$\Sigma$ is non-random. The covariance actually admits the usual representation
\begin{equation}
\label{E:1.14}
v\cdot\Sigma v = \E\biggl(\,\sum_{e\colon |e|=1}a_0(e)\bigl|v\cdot\psi(0,e,\cdot)\bigr|^2\biggr)\,,
\quad v\in\R^d,
\end{equation}
where $\psi\colon\R\times\Z^d\times\Omega\to\R^d$ is the harmonic coordinate constructed in Section~\ref{sec2}. However, unlike for the static situations, the harmonic coordinate is not obtained by minimizing Dirichlet energy; instead one has to solve the heat equation \eqref{E:3.1} using a suitable limit procedure.

\subsection{Connections and main ideas}
Theorem~\ref{thm-1.2} is an example of a quenched invariance principle which has been a topic of persistent interest over the past few decades. In the realm of \emph{static} environments, the studied examples include uniformly elliptic random conductance models (Kipnis and Varadhan~\cite{KV86}, Boivin~\cite{B93}, Boivin and Depauw~\cite{BD03}, Sidoravicius and Sznitman~\cite{SS04}), the random walk on the supercritical percolation cluster (De Masi, Ferrari, Goldstein and Wick~\cite{DFGW89a,DFGW89b}, Sidoravicius and Sznitman~\cite{SS04}, Berger and Biskup~\cite{BB07}, Mathieu and Piatnitski~\cite{MP07}), non-elliptic i.i.d.\ random conductance models (Mathieu~\cite{M08}, Biskup and Prescott~\cite{BP07}, Barlow and Deuschel~\cite{BD12}, Andres, Barlow, Deuschel and Hambly~\cite{ABDH13}), balanced models (Lawler~\cite{L82}, Guo and Zeitouni~\cite{GZ12}, Berger and Deuschel~\cite{BD14}), environments admitting finite cycle decompositions (Deuschel and K\"osters~\cite{DK08}). Recently, an elliptic regularity-based theory was developed that covers general random conductance models subject to moment conditions on the conductance tails at zero and infinity (Andres, Slowik and Deuschel~\cite{ASD15}).

Significant advances have occured also for random walks in \emph{dynamical} random environments. Here a line of attack focused on Markovian environments under various mixing conditions (Boldri\-ghini, Minlos and Pellegrinotti~\cite{BMP07}, Bandyopadhyay and Zeitouni~\cite{BZ06}, Dolgopyat, Keller and Liverani~\cite{DKL08}, Redig and V\"ollering~\cite{RV13}) while other approaches worked under other structural assumptions on the environment such as independence and directionality (Rassoul-Agha and Sepp\"alainen~\cite{RS05}) or ergodicity and uniform ellipticity (Andres~\cite{A14}). Random walks on dynamical percolation have been studied by Peres, Stauffer and Steif~\cite{PSS15} but the objective there were mixing properties rather than the scaling limit. An annealed invariance principle for random walks on the symmetric exclusion has been proved by~Avena~\cite{A12}. 

The sharpest conclusions concerning scaling to Brownian motion for dynamical environments of the kind \eqref{E:1.1a} appear at present in the work of Andres, Chiarini, Deuschel and Slowik~\cite{ACDS16}. Indeed, a quenched invariance principle has been shown there to hold whenever
\begin{equation}
\label{E:1.13a}
\E\bigl(a_t(e)^p\bigr)<\infty\quad\text{and}\quad \E\bigl(a_t(e)^{-q}\bigr)<\infty
\end{equation}
are true for some $p,q>1$ with
\begin{equation}
\label{E:1.14a}
\frac1{p-1}+\frac1{(p-1)q}+\frac1q<\frac2d\,.
\end{equation}
(Somewhat weaker, albeit harder-to-state, conditions actually suffice.) Although our rates are bounded (i.e., we can set~$p:=\infty$ above), the principal novelty of our work is that we allow $a_t(e)=0$ with positive probability (which rules out existence of any~$q$ as above). This is quite important in applications; e.g., we can reach previously unattainable examples such as the random walk on dynamical percolation or the Helffer-Sj\"ostrand walks for potentials \eqref{E:1.7}, and even \eqref{E:1.6} for any~$\beta>0$ (note that \twoeqref{E:1.13a}{E:1.14a} apply only for~$\beta$ sufficiently~large). 

Our approach is technically based on combining an enhanced version of the methods of the aforementioned article~\cite{ACDS16} with an observation from Proposition~4.6 in Mourrat and Otto~\cite{MO16}. The latter work proves a heat-kernel estimate (a.k.a.\ return probability) for random walks covered by our Assumption~\ref{ass1}. The former in turn addresses random walks among random conductances satisfying \twoeqref{E:1.13a}{E:1.14a} with the aim of proving that these scale to Brownian motion. The strategy there is fairly standard: prove that the key object of stochastic homogenization, the corrector, scales sublinearly in space and sub-diffusively in time.

The technical approach of~\cite{ACDS16} (drawing on its precursor~\cite{ASD15} for static environments) is to control the corrector in supremum norm by way of Moser iteration starting only from \textit{a priori} estimates in~$L^1$-norm. A key point is that the condition on the negative moment of~$a_t(e)$ from \eqref{E:1.13a} is used only in a handful of places, and that typically for a conversion of a bound on a weighted $L^2$-norm to a bound on an $L^1$-norm, but these seem absolutely irreplaceable in the whole argument. This is where the said observation from \cite{MO16} enters for us as this work shows that, under suitable averaging over time, one can control the heat kernel using energy norms where the ``naked'' $a_t(e)$ is substituted by the weights
\begin{equation}
\label{E:1.16}
w_t(e):=\int_t^\infty \textd s\,\,k_{s-t}\,a_s(e)
\end{equation}
for some positive, polynomially decaying function $t\mapsto k_t$. The crucial input from \cite[Proposition~4.6]{MO16} is that these weighted energy norms can, for solutions of relevant Poisson equations, be again bounded by the ordinary energy norms (i.e., those where~$a_t(e)$ replaces~$w_t(e)$).

Under the condition $T_e<\infty$ a.s.\ we have $w_t(e)>0$ a.s.\ and since we will even require finiteness of some moments of~$T_e$, we can count on having suitable moments of~$w_t(e)^{-1}$. The basic strategy of our proofs is thus to demonstrate that one can substitute $a_t(e)^{-1}$ by $w_t(e)^{-1}$ in those few places in the argument of \cite{ACDS16} where finiteness and moments of these quantities are crucially required. However, this would in itself be an understatement of our contribution. Indeed, we have to carefully adapt the Moser iteration from \cite{ACDS16} which is based on conversion (via an inequality from Kru\v zkov and Kolodi\u{\i}~\cite{KK77}) of certain space-time norms of the corrector into an $L^\infty$-norm in time. This in turn requires generalizing \cite[Proposition~4.6]{MO16} to include arbitrary moments of the solutions. In addition, we also need to devise an alternative construction, and prove the \emph{a priori} $L^1$-estimate, of the corrector. Unlike for \cite{ACDS16}, these will again hinge on the aforementioned conversion of the energy norms.

\subsection{Remarks and open questions}
We proceed with a couple of remarks and open questions. First off, our aim here has been to find a way to prove convergence to Brownian motion under \emph{some} reasonable (moment) conditions on the environment and so we have not tried to tune these conditions to get optimal control. It is thus of interest to solve:

\begin{problem}
Find out whether sharp moment conditions on~$T_e$ exist for an invariance principle to hold for all environments satisfying Assumption~\ref{ass1}. 
\end{problem}

\noindent
We note that this includes both quenched and annealed statements. To see that we should hope to get better than \eqref{E:q_ass}, we note:

\begin{lemma}
\label{lemma-moments}
Suppose Assumption~\ref{ass1} holds and, in addition, assume that~$\BbbP$ is separately ergodic with respect to time shifts $\{\tau_{t,0}\colon t\in\R\}$ alone. Then for each~$q>0$,
\begin{equation}
\label{E:1.18u}
\E\bigl(T_e^{q+1}\bigr)\le\bigl[\E\bigl(a_0(e)^{-q}\bigr)\bigr]^{\frac{q+1}q}.
\end{equation}
\end{lemma}

We relegate the (easy) proof to the Appendix. 

\begin{remark}
\label{rem-conditions}
Under the conditions \twoeqref{E:1.13a}{E:1.14a} with $p:=\infty$, which requires $q>d/2$, we thus get finiteness of moments of~$T_e$ of order larger than $\frac d2+1$. We note that this is less than~$4d$ in all~$d\ge2$ so our condition \eqref{E:q_ass} is generally quite a bit stronger than \twoeqref{E:1.13a}{E:1.14a}.
\end{remark}

As we will see in Lemma~\ref{L:nu_w}, our conditions on~$T_e$ imply conditions on the negative moments of~$w_0(e)$ which then serve as technical input for the rest of the proofs. Noting that integrability of~$t\mapsto k_t$ and Jensen's inequality imply
\begin{equation}
\label{E:1.19u}
\E\bigl(w_0(e)^{-q}\bigr)\le c\, \E\bigl(a_0(e)^{-q}\bigr)
\end{equation}
for some~$c=c(k)\in(0,\infty)$, these moment conditions on $w(e)^{-1}$ are directly implied by the corresponding moment conditions on $a_0(e)^{-1}$. The bounds \twoeqref{E:1.18u}{E:1.19u} indicate that the setting of \cite{ACDS16} is naturally included in ours, except that (as was just noted in Remark~\ref{rem-conditions}) our conditions are more stringent than those in~\cite{ACDS16}. We take this as a suggestion for potential improvements of our techniques.

\smallskip
Another aspect left out in our study are environments where~$a_t(e)$ is unbounded from above. These include some very interesting cases; in fact, our initial motivation was to understand a specific model where~$t\mapsto a_t(e)$ is zero except for some random times when it has a Dirac-delta singularity. 
Our proofs require boundedness of~$a_t$ in a number of places and we do not know how to overcome these restrictions.

Yet another aspect where our study falls short is our choice of time-parametrization of the walk. Indeed, our choice of the generator \eqref{E:1.1} corresponds to the so-called variable-speed random walk but other parametrizations, e.g., the constant-speed random walk, are of interest as well.  In particular, we would like to solve: 

\begin{problem}
Extend our conclusions to discrete-time random walks among (discrete) time dependent random conductances subject to (analogues of) Assumption~\ref{ass1}.
\end{problem}

\noindent
A  somewhat unexpected feature of time-dependent random environments is that different time-parametrizations are not directly related and so our proofs do not shed any light on those either. We consider this to be one of the most challenging open problems of this subject area.

As an attentive reader has surely noticed, our results are stated under the restriction to spatial dimensions~$d\ge2$. This is dictated by the fact that the parameters space-time Sobolev inequalities behave differently in~$d=1$ than in~$d\ge2$. Although we think that these differences can be overcome, we have decided to skip the~$d=1$ case in order to avoid having to deal with annoying provisos and keep the paper to a manageable length. Under the moment conditions \twoeqref{E:1.13a}{E:1.14a}, the one-dimensional case has been addressed in \cite{DS16}.

Finally, although we work with elliptic regularity techniques, we have not touched the subject of heat-kernel estimates; i.e., Gaussian-type upper/lower bounds on the probability $p^{t,s}(x,y)$ that the walk conditional on being at~$x$ at time~$t$ is at~$y$ at a later time~$s$. Unlike for the static environments, such bounds are much less regular and various pathologies may arise (cf.\ Huang and Kumagai~\cite{HK16}). 
As already mentioned, for our class of environments upper bounds on the diagonal term $p^{t,s}(x,x)$ have been derived in Mourrat and Otto~\cite{MO16}. In analogy with the static case, we expect that our proof of the invariance principle with non-degenerate diffusion matrix should imply an on-diagonal lower bound. However, we have not been able to conclude this rigorously.

\subsection{Outline}
The remainder of the paper is organized as follows. In Section~\ref{sec-2a} we develop the functional-theoretical tools underpinning the proofs in later sections. This, in particular, includes the introduction of Sobolev inequalities and conversion of the Dirichlet energies mentioned above. In Section~\ref{sec2}, we then construct the harmonic coordinate, which one can think of as an embedding of~$\Z^d$ on which the random walk is a martingale. The change in the embedding is expressed by the said {corrector}, which is a fundamental quantity in all standard treatments of random conductance models (see, e.g, recent reviews by Biskup~\cite{B11} and Kumagai~\cite{Kumagai}). The above mentioned \emph{a priori} $L^1$-estimates on the corrector are also derived here using methods of independent interest. In Section~\ref{sec3} we give a proof of the main result subject to a pointwise sublinearity estimate on the corrector. This estimate is then substantiated in Sections~\ref{sec6}--\ref{sec6a} by combining the \emph{a priori} $L^1$-bounds with Moser iteration. The Appendix collects some estimates that would be a distraction in the main line of a proof.

Let us make the following convention about the use of constants. We denote by $c,c',\dots$ positive and finite constants which can change from place to place. Numbered constants $c_1, c_2,\dots$ become fixed whenever they first appear. Their dependence on all parameters will always be explicit.


\section{Sobolev inequalities and weighted energies}
\label{sec-2a}\nopagebreak\noindent
Here we introduce the necessary functional-theoretical tools for our later proofs. A reader preferring to avoid technicalities until they are actually used may consider skipping this section and returning to it only while reading the rest of the paper.

\subsection{The $\ell^1$-Sobolev inequality}
The control of the corrector in stochastic homogenization seems to always require a kind of coercive-type estimate for its Dirichlet energy in terms of a suitable norm. Historically this was done (e.g., in Sidoravicius and Sznitman~\cite{SS04}, drawing on Delmotte~\cite{D99} and Barlow~\cite{B04}) via the Poincar\'e inequality. This is easy and elegant in uniformly elliptic cases but becomes less so when one deals with non-elliptic environments and, particularly, wishes to work under moment assumptions on the conductances only. In this line of thought, Andres, Deuschel and Slowik~\cite{ASD15} devised a powerful approach based on Sobolev inequalities which we will follow here as well. The starting point of this approach is:

\begin{lemma}[$\ell^1$-Sobolev inequality]
For each~$d\ge2$ there is $c=c(d)\in(0,\infty)$ such that any $f\colon\Z^d\to\R$ with finite support, 
\begin{equation}
\label{E:2.27a}
\Bigl(\,\sum_{x\in\Z^d}\bigl|f(x)\bigr|^{\frac{d}{d-1}}\Bigr)^{\frac{d-1}d}
\le c(d)\sum_{(x,y)\in E(\Z^d)}\bigl|f(x)-f(y)\bigr|\,.
\end{equation}
\end{lemma}

\begin{proofsect}{Proof}
This is very standard, but we give a proof as it is short and instructive and, also, as we will reuse the argument in the next lemma. 
First off we use H\"older's inequality to get
\begin{equation}
\label{E:2.28}
\begin{aligned}
\sum_{x\in\Z^d}\bigl|f(x)\bigr|^{\frac{d}{d-1}}
&=\int_0^\infty\textd s\,\sum_{x\in\Z^d}\bigl|f(x)\bigr|^{\frac{1}{d-1}}1_{\{|f(x)|>s\}}
\\
&\le 
\Bigl(\,\sum_{x\in\Z^d}\bigl|f(x)\bigr|^{\frac{d}{d-1}}\Bigr)^{1/d}\int_0^\infty\textd s
\,\bigl|\{x\in\Z^d\colon |f(x)|>s\}\bigr|^{\frac{d-1}{d}}\,.
\end{aligned}
\end{equation}
By the isoperimetric inequality in~$\Z^d$, for any finite~$\Lambda\subset\Z^d$, we have $|\Lambda|^{\frac{d-1}d}\le c(d)|\partial\Lambda|$, where $\partial\Lambda$ is the set of edges with exactly one endpoint in~$\Lambda$ and $|A|$ denotes the cardinality of $A$, for any $A \subset \Z^d$. Using this for $\Lambda:=\{x\in\Z^d\colon |f(x)|>s\}$ in \eqref{E:2.28}, shows
\begin{equation}
\label{E:2.29}
\Bigl(\,\sum_{x\in\Z^d}\bigl|f(x)\bigr|^{\frac{d}{d-1}}\Bigr)^{\frac{d-1}d}
\le c(d)\int_0^\infty\textd s\,\sum_{\begin{subarray}{c}
x,y\in\Z^d\\|x-y|=1
\end{subarray}}
1_{\{|f(x)|>s\ge|f(y)|\}}.
\end{equation}
Performing the integral and using that $\bigl||a|-|b|\bigr|\le|a-b|$ now yields \eqref{E:2.27a}.
\end{proofsect}

We note (as our proof above attests) that the $\ell^1$-Sobolev inequality is equivalent to the isoperimetric inequality. The restriction to~$f$ with finite support is sometimes inconvenient and one might wish to work instead in a finite box. The following lemma addressing this setting will be quite useful. No surprise, it is still based on isoperimetry but this time in a finite box:

\begin{lemma}
\label{lemma-Sob}
For~$d\ge2$ there is~$c'=c'(d)\in(0,\infty)$ such that for any $f\colon\Z^d\to\R$, any $n\ge1$ and any translate~$B$ of~$[0,n]^d\cap\Z^d$,
\begin{equation}
\sum_{x\in B}\bigl|f(x)-\bar f_{B}\bigr|
\le c'(d)\,|B|^{1/d}\sum_{\begin{subarray}{c}
(x,y)\in E(\Z^d)\\x,y\in B
\end{subarray}}
\bigl|f(x)-f(y)\bigr|,
\end{equation}
where $\bar f_{B}:=|B|^{-1}\sum_{x\in B}f(x)$.
\end{lemma}

\begin{proofsect}{Proof}
Replacing~$f$ by~$-f$ if needed, we may assume without loss of generality that
\begin{equation}
\bigl|\{x\in B\colon f(x)>\bar f_B\}\bigr|\le \bigl|\{x\in B\colon f(x)<\bar f_B\}\bigr|\,.
\end{equation}
Let~$\Lambda$ denote the set on the left-hand side.
Since $\sum_{x\in B}(f(x)-\bar f_B)=0$, we have
\begin{equation}
\label{E:2.32}
\sum_{x\in B}\bigl|f(x)-\bar f_B\bigr|=2\sum_{x\in \Lambda}\bigl(f(x)-\bar f_B\bigr).
\end{equation}
Jensen's inequality along with the argument in \eqref{E:2.28} then show
\begin{equation}
\sum_{x\in \Lambda}\bigl(f(x)-\bar f_B\bigr)
\le|\Lambda|^{1/d}\int_0^\infty\textd s\,\bigl|\{x\in\Lambda\colon f(x)-\bar f_B>s\}\bigr|^{\frac{d-1}d}\,.
\end{equation}
Since~$|\Lambda|\le\frac12|B|$, the isoperimetric inequality in~$B$ yields $|\Lambda'|^{\frac{d-1}d}\le\tilde c(d)|\partial^B\Lambda'|$ for any $\Lambda'\subset\Lambda$, where $\partial^B\Lambda'$ is the set of edges in~$\partial\Lambda'$ that have both endpoints in~$B$. Using this as in \eqref{E:2.29} and plugging the result into  \eqref{E:2.32} yields the claim with~$c'(d):=2^{1-1/d}\tilde c(d)$.
\end{proofsect}

\subsection{Sobolev inequalities with weighted energies}
Our next goal will be a conversion of the $\ell^1$-Sobolev inequality into a more useful form. Given any Lebesgue measurable $\zeta\colon\R\to[0,\infty)$, for any measurable $f\colon\R\times\Z^d\to\R$, with the value at $(t,x)$ denoted by~$f_t(x)$, any~$B\subset\Z^d$ and any~$p,q\in(0,\infty)$, define the norms
\begin{equation}
\label{eq:norm_pp'}
\Vert f\Vert_{p,q;B,\zeta}:=\Biggl(\int\textd t\,\zeta(t)\Bigl(\sum_{x\in B}\bigl|f_t(x)\bigr|^p\Bigr)^{q/p}\Biggr)^{1/q}\,.
\end{equation}
Recalling our notation $E(\Z^d)$ for the set of unordered edges (with each edge included only once) in~$\Z^d$ and writing~$E(B)$ for the set of edges in~$E(\Z^d)$ with at least one endpoint in~$B$, we will use the notation $\Vert f\Vert_{p,q;E(B),\zeta}$ to denote the corresponding object for functions $f\colon \R\times E(\Z^d)\to\R$; just replace sum over~$x\in B$ by sum over $e\in E(B)$. 

For any~$B\subset\Z^d$, any $t\in\R$, any $f\colon\R\times\Z^d\to\R$ and any collection of non-negative weights $\{w_t(e)\colon e\in\Z^d\}$ we define
\begin{equation}
\label{E:2.9.0}
\EE^w_{t,B}(f_t):=\sum_{(x,y)\in E(B)}
w_t(x,y)\bigl[\,f_t(y)-f_t(x)\bigr]^2.
\end{equation}
The notation $\EE^a_{t,B}(f_t)$ will be reserved for the specific situation when the weights are given by the conductances~$a_t(e)$. Assuming in addition that $t\mapsto w_t(e)$ is Borel measurable for each~$e$, we then define the integrated forms of these via
\begin{equation}
\label{E:2.9}
\EE^{w,\zeta}_B(f):=\int\textd t\,\zeta(t)\EE^w_{t,B}(f_t),
\end{equation}
reserving $\EE^{a,\zeta}_B(f)$ again for the case when the weights are given by the conductances. If~$B=\Z^d$, we denote the above energies simply by $\EE^{w}_t(f_t)$ and $\EE^{w,\zeta}(f)$.
We now claim the validity of the following family of inequalities:

\begin{lemma}[Sobolev inequalities]
\label{thm-Sobolev}
For each $d\ge2$, each $\alpha\in(1,2\frac{d-1}{d-2})$ and each $\beta\in(0,2)$ there is $\czero=\czero(d,\alpha,\beta)\in(0,\infty)$ such that for $r,s$ defined by
\begin{equation}
\label{E:1.4a}
\frac{\alpha-1}\alpha\frac{d-1}d+\frac1r=\frac12\quad\text{and}\quad\frac1{s}+\frac12=\frac1\beta,
\end{equation}
the inequality
\begin{equation}
\label{E:1.5a}
\Vert f\Vert_{\alpha\frac d{d-1},\beta;B,\zeta}
\le \czero\Vert w^{-1/2}\Vert_{r,s;E(B),\zeta}\,\,\EE_B^{w,\zeta}(f)^{1/2}
\end{equation}
holds for any finite~$B\subset\Z^d$ and any measurable $f\colon\R\times\Z^d\to\R$.
\end{lemma}

The quantity $2\frac{d-1}{d-2}$ should henceforth be interpreted as infinity when~$d=2$. We remark that Andres, Chiarini, Deuschel and Slowik~\cite{ACDS16} derive \eqref{E:1.5a} for the particular case when $\zeta(t):=T^{-1}1_{[0,T]}(t)$ and~$w_t(e)$ replaced by~$a_t(e)$. However, their parametrization is different from ours. 

\begin{remark}
The norm \eqref{eq:norm_pp'} is asymmetric in the sense that it puts integration with respect to the spatial variables before that with respect to time and so the reader may wonder whether setting the norms up the opposite way may give us  any advantage. To address this issue, define
\begin{equation}
\Vert f\Vert_{p,q;B,\zeta}^{\sim}:=\Biggl(\sum_{x\in B}\Bigl(\int\textd t\,\zeta(t)\bigl|f_t(x)\bigr|^p\Bigr)^{q/p}\Biggr)^{1/q}\,.
\end{equation}
Then a similar calculation to the one in the proof of Lemma~\ref{thm-Sobolev} below shows that, for each $\alpha\in(1,2\frac{d-1}{d-2})$ and each $\beta\in(1,2)$,
\begin{equation}
\label{E:1.6}
\Vert f\Vert_{\beta,\alpha\frac d{d-1};B,\zeta}^\sim
\le c\Vert w^{-1/2}\Vert_{r,s;E(B),\zeta}^\sim\,\,\EE_B^{w,\zeta}(f)^{1/2}
\end{equation}
holds for any finite~$B\subset\Z^d$ and any measurable $f\colon\R\times\Z^d\to\R$.
In particular, both ways to define space-time norms seem more or less equally powerful. 
\end{remark}

\begin{proofsect}{Proof of Lemma~\ref{thm-Sobolev}}
Let~$\tilde E(B)$ denote the set of ordered edges with at least one endpoint in~$B$. Pick~$\alpha$ and~$\beta$ from the allowed ranges. The $\ell^1$-Sobolev inequality along with the fact that
\begin{equation}
\label{E:1.7}
|x^\alpha-y^\alpha|\le\alpha (x^{\alpha-1}+y^{\alpha-1})|x-y|,\qquad x,y>0,\,\alpha>0,
\end{equation}
and simple symmetrization show
\begin{equation}
\label{E:1.8}
\Bigl(\,\sum_{x\in B}\bigl|f_t(x)\bigr|^{\alpha\frac d{d-1}}\Bigr)^{\frac{d-1}{\alpha d}\beta}
\le c\Bigl(\,\sum_{(x,y)\in\tilde E(B)}\bigl|f_t(x)\bigr|^{\alpha-1}\bigl|f_t(x)-f_t(y)\bigr|\Bigr)^{\beta/\alpha}\,.
\end{equation}
Let $p$ be defined by
\begin{equation}
\label{E:1.9}
p(\alpha-1)=\alpha\frac d{d-1}
\end{equation}
and notice that then $\frac1p+\frac12+\frac1r=1$ by the first equality in \eqref{E:1.4a}. H\"older's inequality with indices $(p,2,r)$ then bounds the right-hand side of \eqref{E:1.8} by
\begin{equation}
c\Bigl(\,\sum_{x\in B}\bigl|f_t(x)\bigr|^{\alpha\frac d{d-1}}\Bigr)^{\frac\beta{p\alpha}}
\Bigl(\,\sum_{(x,y)\in E(B)} w_t(x,y)\bigl|f_t(x)-f_t(y)\bigr|^{2}\Bigr)^{\frac\beta{2\alpha}}
\Bigl(\,\sum_{(x,y)\in E(B)}\!\!\!\!w_t(x,y)^{-r/2}\Bigr)^{\frac\beta{r\alpha}},
\end{equation}
where the constant~$c$ arises from rewriting the first sum from that over edges to that over sites and where the sums are now over unordered edges again. Now multiply the resulting inequality by~$\zeta(t)$ and integrate over~$t$. Since the second equality in~\eqref{E:1.4a} ensures that $(p\frac{d-1}d,2\alpha/\beta, s\alpha/\beta)$ are H\"older conjugate indices, another use of H\"older's inequality yields
\begin{multline}
\int\textd t\,\zeta(t)\Bigl(\,\sum_{x\in B}\bigl|f_t(x)\bigr|^{\alpha\frac d{d-1}}\Bigr)^{\frac{d-1}{\alpha d}\beta}
\\
\le
c\Biggl(\int\textd t\,\zeta(t)\Bigl(\,\sum_{x\in B}\bigl|f_t(x)\bigr|^{\alpha\frac d{d-1}}\Bigr)^{\frac{d-1}{\alpha d}\beta}\Biggr)^{\frac1p\frac{d}{d-1}}\, \Bigl[\EE_B^{w,\zeta}(f)^{1/2}\Vert w^{-1/2}\Vert_{r,s;E(B),\zeta}\Bigr]^{\beta/\alpha}\,.
\end{multline}
This now readily implies \eqref{E:1.5a}.
\end{proofsect}

Our later applications make it convenient to introduce normalized versions of the above norms. Assuming~$\zeta$ to be integrable and denoting by $\Vert\zeta\Vert_{L^1}$ its $L^1$-norm is with respect to Lebesgue measure, we thus set
\begin{equation}
\label{eq:norm_2}
\vvert f\vvert_{p,q;B, \zeta}:=|B|^{-1/p} \Vert \zeta \Vert_{L^1}^{-1/q} \Vert f\Vert_{p,q;B, \zeta}\,. 
\end{equation}
For the case~$q:=\infty$ we get
\begin{equation}
\label{eq:norm_3}
\vvert f\vvert_{p,\infty;B,\zeta}:= \text{esssup}\biggl( t\mapsto \Bigl(\, \frac{1}{|B|}  \sum_{x\in B}\bigl|f_t(x)\bigr|^p\Bigr)^{1/p}\biggr),
\end{equation}
where the essential supremum is with respect to the Lebesgue measure on~$\supp\zeta$.
We will write $\Vert f\Vert_{p,q;E(B), \zeta}$ and $\vvert f\vvert_{p,q;E(B), \zeta}$ to denote the corresponding norms for functions indexed by edges of~$\Z^d$.
For later reference, we note that, by Jensen's inequality,
\begin{equation}
\label{E:pq_monot}
\vvert f\vvert_{p,q;B, \zeta} \text{ is increasing in $p$ and $q$ for all $p,q >0$}.
\end{equation}
The norms $\vvert f\vvert_{p,q;B, \zeta}$ will be used heavily in Sections~\ref{sec6}-\ref{sec6a}.
The following form of \eqref{E:1.5a} is tailored to the purposes of that section.

\begin{corollary}
\label{C:Sobolev-tailored'}
For each $d\ge2$, each $\alpha\in(1,2\frac{d-1}{d-2})$ and each $\beta\in(0,2)$ and for~$r,s$ and $\czero$ as in Lemma \ref{thm-Sobolev}, defining
\begin{equation}
\label{eq:p_s}
\hat p =\hat p(\alpha):= \frac{\alpha}{2}\frac d{d-1} \quad\text{and}\quad
\hat q := \hat q(\beta)= \frac{\beta}{2},
\end{equation}
the bound
\begin{equation}
\label{E:1.5abcd'}
\vvert f^2\vvert_{\hat p,\hat q; B,\zeta}
\le 2d \czero^2\, |B|^{\frac{2}{d}} \, \,  \frac{\EE^{w,\zeta}_B({f})}{|B|}  \, \,  \vvert w^{-1}\vvert_{\frac r2, \frac s2; E(B),\zeta} \,.
\end{equation}
holds for all finite $B\subset\Z^d$ and all measurable $f\colon\R\times\Z^d\to\R$.
\end{corollary}

\begin{proof}
An application of \eqref{E:1.5a} yields
\begin{equation}
\label{eq:sob107}
\begin{split}
\vvert f^2\vvert_{\frac{\alpha}{2}\frac d{d-1}, \frac{ \beta}2;B,\zeta} &= 
\Vert\zeta\Vert_{L^1}^{-2/\beta}\,|B|^{- \frac{2}{\alpha}\frac{d-1}{d}} \Vert f\Vert_{{\alpha}\frac d{d-1}, { \beta};B,\zeta}^2\\
&\leq \czero^2\, \Vert\zeta\Vert_{L^1}^{-2/\beta} |B|^{- \frac{2}{\alpha}\frac{d-1}{d} } \, \,  \EE_B^{w,\zeta}({f}) \, \, \Vert w^{-1/2}\Vert_{r,s;E(B),\zeta}^{2}\\
&\leq (2d)^{2/r} \czero^2\, |B|^{2[ -\frac{1}{\alpha}\frac{d-1}{d} +  \frac{1}{2} + \frac{1}{r}]} \,  \frac{\EE_B^{w,\zeta}({f})}{|B|}  \, \,  \vvert w^{-1}\vvert_{\frac r2, \frac s2;E(B),\zeta}.
\end{split}
\end{equation}
Now \eqref{E:1.4a} implies $\frac{1}{2} + \frac{1}{r} = 1- \frac{\alpha-1}\alpha\frac{d-1}d=1-\frac{d-1}d+\frac1\alpha\frac{d-1}d$ and so 
\begin{equation}
-\frac{1}{\alpha}\frac{d-1}{d} + \frac{1}{2} + \frac{1}{r} = 1-\frac{d-1}d = \frac1d.
\end{equation}
Using this in \eqref{eq:sob107},  and noting that $2/r \leq 1$, the claim follows.
\end{proof}

Our application of the above norm in Moser iteration requires a comparison between various instances of the norm~\eqref{eq:norm_pp'}. This is the content of the following lemma.

\begin{lemma}[Interpolation]
\label{lemma-bd}
Suppose $p,q,p_1,p_2,q_1,q_2\in(0,\infty)$ and $\theta\in(0,1)$ are such that
\begin{equation}
\label{eq:intpol1}
\frac1p=\frac\theta{p_1}+\frac{1-\theta}{p_2}\quad\text{and}\quad\frac1q=\frac\theta{q_1}+\frac{1-\theta}{q_2}.
\end{equation}
Then, for all measurable~$f\colon [0,\infty) \times \Z^d\to \mathbb{R}$ and all finite $B \subset \Z^d$,
\begin{equation}
\label{eq:intpol2}
\Vert f\Vert_{p,q;B,\zeta}\le\Vert f\Vert_{p_1,q_1;B,\zeta}^\theta\, \Vert f\Vert_{p_2,q_2;B,\zeta}^{1-\theta}.
\end{equation}
In particular, for all $q, q_1 \in (0,\infty)$ with $q_1<q$ and all $p, p_1, p_2\in (0,\infty)$ satisfying the first condition of \eqref{eq:intpol1} with  $\theta := \frac{q_1}{q}$, we have 
\begin{equation}
\label{eq:intpol3}
\vvert f\vvert_{p,q; B , \zeta}
\le
\vvert f\vvert_{p_1,q_1; B , \zeta}^{\frac{q_1}{q}}\, \vvert f\vvert_{p_2,\infty; B,\zeta }^{1-\frac{q_1}{q}}.
\end{equation}
\end{lemma}

\begin{proofsect}{Proof}
Writing $|f_t(x)|^p=|f_t(x)|^{\theta p}|f_t(x)|^{(1-\theta)p}$ in \eqref{eq:norm_pp'} and invoking H\"older's inequality with conjugate exponents $(\frac{p_1}{\theta p},\frac{p_2}{(1-\theta)p})$ yields
\begin{equation}
\Vert f\Vert_{p,q;B,\zeta} \le \Biggl(\int\textd t\,\zeta(t)
\Bigl(\sum_{x\in B}\bigl|f_t(x)\bigr|^{p_1}\Bigr)^{\theta q/p_1} 
\Bigl(\sum_{x\in B} \bigl|f_t(x)\bigr|^{p_2}\Bigr)^{(1-\theta)q/p_2}
\Biggr)^{1/q}\,.
\end{equation}
H\"older's inequality with conjugate exponents $(\frac{q_1}{\theta q},\frac{q_2}{(1-\theta)q})$ then readily gives \eqref{eq:intpol2}. The inequality \eqref{eq:intpol3} follows from \eqref{eq:norm_2} and \eqref{eq:intpol2} by noting that $ \vvert f\vvert_{p_2,q_2;B, \zeta} \to  \vvert f\vvert_{p_2,\infty; B,\zeta}$ as $q_2 \to \infty$.
\end{proofsect}

\subsection{Edge weights and their growth}
Throughout the rest of this paper, the edge weights $w_t(e)$ we will work with always take the form~\eqref{E:1.16}. The choice of the function~$k\colon[0,\infty)\to(0,\infty)$ underlying \eqref{E:1.16} is tied to the choice of the function $\zeta\colon\R\to(0,\infty)$ governing the above norms by certain conditions we will now spell out. Other than having to obey the restricted set of constraints listed in \twoeqref{eq:kernel_conds}{E:2.2a}, the functions~$k$ and $\zeta$ can be chosen arbitrarily for the purposes we have~in~mind.

We assume that
the function $\zeta\colon\R\to[0,\infty)$ is supported in $[0,\infty)$ and is bounded, non-increasing, continuously differentiable on $(0,\infty)$ with both $t\mapsto\zeta(t)$ and $t\mapsto t\zeta(t)$ Lebesgue integrable. We also assume that $\Vert \zeta\Vert_{L^1}>0$ and
\begin{equation}
\label{eq:kernel_conds}
\inf_{t \in [0,2]} \zeta(t) > 0\quad\text{and}\quad
\Vert\dot{\zeta}/\zeta\Vert_{L^\infty(\mathbb{R}_+)}<\infty.
 \end{equation}
The function $k\colon[0,\infty)\to(0,\infty)$ is Borel measurable with both $t\mapsto k_t$ and $t\mapsto tk_t$ Lebesgue integrable on $[0,\infty)$. Moreover, setting
\begin{equation}
\label{E:Kt}
K_t:=k_t+\int_t^\infty \textd s\,(s-t) k_s,\qquad t\ge0,
\end{equation}
there exists a constant $\ci\in(0,\infty)$ such that for each $s\ge0$, 
\begin{equation}
\label{E:2.2a}
\int_0^s \textd t\,\zeta(t)\,K_{s-t}\le \ci \zeta(s).
\end{equation}

\begin{remark} 
\label{R:ass2}
The condition \eqref{E:2.2a} is needed for the conversions of Dirichlet forms mentioned in the Introduction, see in particular Lemmas \ref{L:MO_basic} and \ref{L:DIRICH_CONVERT} below. 
\end{remark}

That a pair of functions $\zeta,k$ satisfying the above requirements exists is ensured by:

\begin{lemma}
\label{lemma-2.2}
Let $\mu>4$ and $\nu\in(2,\mu-2)$. Then $k_t:=(1+t)^{-\mu}$ and $\zeta(t):=2^\nu(1+t)^{-\nu}\1_{[0,\infty)}(t)$ obey the above conditions and, in particular, \twoeqref{eq:kernel_conds}{E:2.2a}. In fact, for all $r \geq 1$, we have
\begin{equation}
\label{E:2.2aa}
\int_0^s \textd t\,\zeta(t/r)\,K_{s-t}\le \ci \zeta(s/r), \quad \text{for }s \geq 0.
\end{equation} 
\end{lemma}

\begin{proofsect}{Proof}
The integrability conditions are immediate from the fact that $\mu>2$ and $\nu>2$; \eqref{eq:kernel_conds} is checked directly. For \eqref{E:2.2aa} we note that $K_t\le c(1+t)^{-\tilde\mu}$ where $\tilde\mu:=\mu-2$ and then observe
\begin{equation}
\begin{aligned}
\int_0^s\textd t\,\bigl(1+(s-t)\bigr)^{-\tilde\mu}&(1+t/r)^{-\nu}
\\
&\le cs^{-\tilde\mu}\int_0^{s/2}\textd t\,(1+t/r)^{-\nu}
+c (s/r)^{-\nu}\int_0^{s/2}\textd t\,(1+t)^{-\tilde\mu}
\\
&\le\tilde c r s^{-\tilde\mu}+\tilde c (s/r)^{-\nu}.
\end{aligned}
\end{equation}
Since $\tilde\mu>\nu$ and $\nu>1$ (and $r\ge1$), both terms on the right are now less than a constant times~$(s/r)^{-\nu}$. This proves \eqref{E:2.2aa} for~$s\ge r$; in the complementary range of~$s$ values the claim is checked directly. 
\end{proofsect}

Unless specified otherwise, we will henceforth always tacitly assume that $\zeta$ and $k$ are a pair of functions satisfying \twoeqref{eq:kernel_conds}{E:2.2a}.
Some (but not all) calculations will require adapting our setting to diffusive scaling of space and time, i.e., choosing $\Vert f\Vert_{p,q;B, \zeta}$ in \eqref{eq:norm_pp'} with~$B$ replaced by
\begin{equation}
\label{E:2.27}
B_r:=[-r,r]^d\cap\Z^d\qquad \text{for } r \geq 1
\end{equation}
 and $\zeta$ replaced by
\begin{equation}
\label{eq:zeta_n}
\zeta_r(t):=\frac1{r^2}\zeta(t/r^2),\, \qquad \text{for } r \geq 1,
\end{equation}
with $\zeta$ as above. It is then natural to require \eqref{E:2.2aa}, instead of just \eqref{E:2.2a}, to hold. (Note that \eqref{E:2.2aa} is tantamount to saying that \eqref{E:2.2a} holds for all pairs $(\zeta_r, k)$,~$r \geq 1$.) When needed, the condition \eqref{E:2.2aa} will always be mentioned explicitly.

\smallskip
The diffusive scaling of time naturally underlies the following property that will be repeatedly used in the sequel:

\begin{lemma}
\label{lemma-sup}
For each~$p>\tilde p\ge1$ there is $c=c(p,\tilde p,\zeta)\in(0,\infty)$ such that for all $f\in L^p(\BbbP)$,
\begin{equation}
\label{E:sup-bd}
\Bigl\Vert\,\sup_{n\ge1}\,\int_0^\infty\textd t\,\,\zeta_n(t) f\circ\tau_{t,0}\Bigr\Vert_{L^{\tilde p}(\BbbP)}\le c\Vert f\Vert_{L^p(\BbbP)}
\end{equation}
In particular, the integrals converge absolutely for all~$n\ge1$.
\end{lemma}

\begin{proofsect}{Proof}
Dominating $f$ by~$|f|$, we may assume without loss of generality that~$f\ge0$. The assumed properties of~$\zeta$ ensure that $\zeta_n(t)=-\frac1{n^2}\int_{t/n^2}^\infty\textd s\,\dot{\zeta}(s)$. Using that $-\dot{\zeta}$ is greater or equal to zero and Tonelli's Theorem yields
\begin{equation}
\int_0^\infty\textd t\,\,\zeta_n(t) f\circ\tau_{t,0} = \int_0^\infty\textd s\,\,(-\dot{\zeta}(s))\frac1{n^2}\int_0^{n^2 s}\textd t f\circ\tau_{t,0}\,.
\end{equation}
Denoting $h:=\sup_{n\ge1}\frac1n\int_0^n\textd t\, f\circ\tau_{t,0}$, straightforward monotonicity considerations show that the supremum over $n$ of the quantity on the right is at most
\begin{equation}
h\Bigl[\,\int_0^1\textd s\,(-\dot{\zeta}(s))+2\int_1^\infty\textd s\,(-\dot{\zeta}(s))s\Bigr]\,.
\end{equation}
The boundedness and integrability of~$\zeta$ imply that both integrals are finite. Jensen's inequality and the Maximal Ergodic Theorem in turn ensure $\Vert h\Vert_{L^{\tilde p}(\BbbP)}\le c\Vert f\Vert_{L^p(\BbbP)}$ for some~$c=c(p,\tilde p)\in(0,\infty)$ independent of~$f$. The claim follows.
\end{proofsect}

In order to use the Sobolev inequalities \eqref{E:1.5abcd'}, we will need a uniform bound on the norms of the weights~$w$ appearing on the right-hand side. This is the content of:
 
\begin{lemma}
\label{L:nu_w}
Under Assumption~\ref{ass1}, given $\E[T_e^{\vartheta}]< \infty$ for some~$\vartheta>0$ (cf. \eqref{E:1.5}), and for $k,\zeta$ satisfying \eqref{E:2.2aa} in addition to \twoeqref{eq:kernel_conds}{E:2.2a}, the following holds: For each $e\in E(\Z^d)$, the family $\{w_t(e)\colon t\in\R\}$ defined in \eqref{E:1.16} is stationary with respect to time-shifts. 
Moreover, if $k_t\ge 1\vee t^{-\mu}$ is true for all~$t\ge0$ and some~$\mu>0$, then
\begin{equation}
\label{E:2.12aa}
\E\bigl(w_0(e)^{-\vartheta/\mu}\bigr)<\infty
\end{equation}
and, in addition,
\begin{equation}
\label{E:2.13a}
\sup_{n\ge1}  \max_{m \in [n,2n]} \,\vvert w^{-1}\vvert_{\frac r2,\frac s2;E(B_m),\zeta_n}<\infty\quad\BbbP\text{-a.s.}
\end{equation}
is satisfied for all $1\le r\le s< 2\vartheta /\mu$.
\end{lemma}

\begin{proofsect}{Proof} 
The stationarity of $t\mapsto w_t(e)$ is clear from \eqref{E:1.16} and the assumed stationarity of $t\mapsto a_t(e)$. The definition \eqref{E:1.5} and monotonicity of~$t\mapsto k_t$ ensure that, for any $e\in E(\Z^d)$,
\begin{equation}
\label{E:2.31}
w_0(e)\ge k_{T_e}
\end{equation}
and so \eqref{E:2.12aa} directly follows from \eqref{E:q_ass} and the assumed bound on~$k_t$. For \eqref{E:2.13a}, we first note that, if~$r\le s$, then \eqref{E:pq_monot} implies
\begin{equation}
\label{E:2.33}
\bvvert w^{-1/2}\bvvert_{r,s;E(B_n),\zeta_n}^s\le
\bvvert w^{-1/2}\bvvert_{s,s;E(B_n),\zeta_n}^s
\le\Vert\zeta\Vert_{L^1}^{-1}\frac1{|E(B_{n})|}\sum_{x\in B_n}h_n\circ\tau_{0,x}\,,
\end{equation}
where
\begin{equation}
\label{E:2.34}
h_n:=\int_0^\infty \textd t\,\zeta_n(t)\,\sum_{z\colon|z|=1} w_{t}\bigl((0,z)\bigr)^{-s/2}\,.
\end{equation}
Under $s<2\vartheta/\mu$, \eqref{E:2.12aa} implies that $w_0(e)^{-s/2}\in L^p(\BbbP)$ for some~$p>1$. Lemma~\ref{lemma-sup} and stationarity of~$t\mapsto w_t$ then show $\sup_{n\ge1}|h_n|\in L^1(\BbbP)$. Bounding~$h_n$ in \eqref{E:2.33} by the supremum, the claim follows from the Spatial Ergodic Theorem.
\end{proofsect}

\subsection{Conversion of Dirichlet energies}
The usual way a regularity argument starts with the use of Sobolev inequality to bound the desired norm of a function by its Dirichlet energy. For a solution to Poisson or heat equation, the Dirichlet energy is in turn bounded by a lower-order norm, thus gaining regularity.  Unfortunately, our Sobolev inequality outputs a weighted Dirichlet energy and so we need an additional step in which we bound this Dirichlet energy by the ordinary one to which the rest of the argument can be applied. 

Recall the definition of the (finite volume) Dirichlet energy in \eqref{E:2.9.0}. The bound that achieves the stated goal is then as follows:

\begin{lemma}
\label{L:MO_basic}
Suppose~$t\mapsto a_t(e)$ are measurable and take values in~$[0,1]$. Let~$B\subset\Z^d$ be finite and set $\overline B:=B\cup\partial B$. If $u\colon\R\times\Z^d\to\R$ solves (weakly) the heat equation
\begin{equation}
\label{E:2.22}
\frac{\partial}{\partial t}u(t,x) =L_tu(t,x)+f(t,x),\qquad t\in\R,\,x\in\overline B,
\end{equation}
for some bounded measurable $f\colon\R\times\Z^d\to\R$, then for each~$t\in\R$,
\begin{equation}
\label{E:2.23}
\EE^w_{t,B}(u_t)\le 48d^2\int_t^\infty \textd s\,\, K_{s-t}\,\EE^a_{s,\overline B}(u_s)+24d\int_t^\infty\textd s \,K_{s-t}\sum_{x\in B}\bigl|f_s(x)\bigr|^2\,,
\end{equation}
where $K_t$ is as in \eqref{E:Kt}.
\end{lemma}

\begin{proofsect}{Proof}
We follow the calculation in the proof of Proposition~4.6 in Mourrat and Otto~\cite{MO16}. The definition of the weights~$w_t(e)$ in \eqref{E:1.16} gives
\begin{equation}
\EE^w_{t,B}(u_t) = \int_t^\infty\textd s\,\, k_{s-t}\sum_{x\in B}\,\,\sum_{\begin{subarray}{c}
y\in\Z^d\\ (x,y)\in E(\Z^d)
\end{subarray}}
a_s(x,y)\bigl[u(t,x)-u(t,y)\bigr]^2.
\end{equation}
Writing $u(t,x) = u(s,x)+[u(t,x)-u(s,x)]$ and using that $(a+b+c)^2\le 3a^2+3b^2+3c^2$ and that $a_t(x,y)\le1$ then shows
\begin{equation}
\label{E:2.25}
\EE^w_{t,B}(u_t) \le \int_t^\infty\textd s\, k_{s-t} \Biggl(3\EE^a_{s,\overline B}(u_s)+
12d\sum_{x\in \overline  B}
\bigl[u(t,x)-u(s,x)\bigr]^2\Biggr).
\end{equation}
Concerning the second term in the parentheses, here \eqref{E:2.22} and $(a+b)^2\le 2a^2+2b^2$ yield
\begin{equation}
\label{E:2.52}
\begin{aligned}
\bigl[u(t,x)-u&(s,x)\bigr]^2 
=\Bigl[\int_s^t \textd r\,\bigl[f_r(x)+L_r u(r,x)\bigr]\Bigr]^2
\\
&=\Biggl[\int_s^t \textd r\,\biggl(\,f_r(x)+\sum_{\begin{subarray}{c}
y\in\Z^d\\ (x,y)\in E(\Z^d)
\end{subarray}} a_r(x,y)\bigl[u(r,y)-u(r,x)\bigr]\biggr)\Biggr]^2
\\
&\le 2(s-t)\int_t^s\textd r\,\biggl(\bigl|f_r(x)\bigr|^2+2d\sum_{\begin{subarray}{c}
y\in\Z^d\\ (x,y)\in E(\Z^d)
\end{subarray}} a_r(x,y)\bigl[u(r,y)-u(r,x)\bigr]^2\biggr)\,,
\end{aligned}
\end{equation}
where the last inequality follows by Cauchy-Schwarz and the bound $a_t(x,y)\le1$. Plugging this in \eqref{E:2.25} and invoking the definition of~$K_t$, we get \eqref{E:2.23}.
\end{proofsect}

\begin{remark}
The argument \eqref{E:2.52} uses crucially that the lattice gradient is a bounded operator. This is what makes the above proof fail in the continuum setting.
\end{remark}

Recall the definitions of $B_n$ and $\zeta_n$ from \eqref{E:2.27} and  \eqref{eq:zeta_n}.
Then we have:

\begin{corollary}
For~$a_t$ and~$u$ as in Lemma~\ref{L:MO_basic}, if \eqref{E:2.2aa} holds (in addition to \eqref{eq:kernel_conds}-\eqref{E:2.2a}), then for each~$n\ge1$,
\begin{equation}
\label{E:2.46}
\EE^{w,\zeta_n}_{B_n}(u)\le 48d^2\ci\,\EE^{a,\zeta_n}_{\overline B_n}(u)+24d \ci\Vert f\Vert_{2,2,B_n,\zeta_n}^2\,.
\end{equation} 
Moreover, under Assumption~\ref{ass1}, if $u$ and $f$ are such that $u(t,x,\cdot) = u(0,0,\cdot)\circ\tau_{t,x}$ and $f(t,x,\cdot)=f(0,0,\cdot)\circ\tau_{t,x}$ for each~$x\in\Z^d$, each~$t\in\R$, and \eqref{E:2.22} holds, then
\begin{equation}
\label{E:2.47}
\begin{aligned}
\E\biggl(\,\,&\sum_{e=e_1,\dots,e_d}w_0(e)\bigl|u(0,e,\cdot)-u(0,0,\cdot)\bigr|^2\biggr)
\\
&\quad\le 48d^2\ci\E\biggl(\,\,\sum_{e=e_1,\dots,e_d}a_0(e)\bigl|u(0,e,\cdot)-u(0,0,\cdot)\bigr|^2\biggr)
+24d \ci\E\bigl(|f(0,0,\cdot)|^2\bigr)\,.
\end{aligned}
\end{equation}
\end{corollary}

\begin{proofsect}{Proof}
In light of \eqref{E:2.23}, the first conclusion follows directly from \eqref{E:2.2aa}. For \eqref{E:2.47} take expectation of \eqref{E:2.46} (this eliminates the integrals over time), divide by~$|B_n|$ and take $n\to\infty$.
\end{proofsect}

We remark that, in the derivations underlying the Moser iteration, we will need to rederive variants of these estimates for powers of the solutions multiplied by suitable mollifiers. Besides illustrating the main ideas of our proofs, the above simpler versions will be used to define, and derive \emph{a priori} $L^1$-estimates, of the corrector in the next section.


\section{Construction of the corrector}
\label{sec2}\nopagebreak\noindent
The next task is the construction and derivation of the needed properties of the harmonic coordinate and the associated corrector. The natural setting for our proof is to require a certain moment condition for the weights~$w_t$ defined in \eqref{E:1.16}, see \eqref{E:2.12a} below. We will verify immediately that this condition is met under the assumptions of Theorem \ref{thm-1.2}.

Note that, whenever Assumption \ref{ass1} holds, the family $\{w_t(e)\colon t\in\R\}$ is stationary with respect to time-shifts for each $e\in E(\Z^d)$, as can be seen from \eqref{E:1.16} and the assumed stationarity of $t\mapsto a_t(e)$. This will be used frequently below.  Recall also that the functions $k,\zeta$ are assumed to satisfy \twoeqref{eq:kernel_conds}{E:2.2a}; $k$ enters through the definition of the weights $w$ and, although $\zeta$ does not appear explicitly in the following theorem, it will be used in its proof. Let $L^{p,\text{loc}}(\R)$ denote the space of measurable~$f\colon\R\to\R$ whose~$p$-th power is locally integrable with respect to the Lebesgue measure and~$\BbbP$. The main conclusion of this section is now as follows:

\begin{theorem}
\label{prop-2.1}
Suppose the law of the conductances~$\BbbP$ obeys Assumptions~\ref{ass1}, \eqref{E:2.2aa} holds and, with $w_t$ as defined in \eqref{E:1.16}, there exists $q > 1$ such that
\begin{equation}
\label{E:2.12a}
\E\bigl(w_0(e)^{-q}\bigr)<\infty, \quad \text{ for all } e \in E(\mathbb{Z}^d).
\end{equation}
Then there exists a measurable function $\psi\colon \R\times\Z^d\times\Omega\to\R^d$ such that the following~holds:
\settowidth{\leftmargini}{(11)}
\begin{enumerate}
\item[(1)] $\psi$ is a weak solution to the family of the ODEs
\begin{equation}
\label{E:3.1}
\frac{\partial}{\partial t}\psi(t,x,\cdot) +L_t\psi (t,x,\cdot) =0,\qquad t\in\R,\, x\in\Z^d,
\end{equation}
where~$L_t$ is the generator defined in \eqref{E:1.1}, and $L_t\psi(t,x,a):= (L_t\psi(t,\cdot,a))(x)$,
\item[(2)]
$\psi$ satisfies the cocycle conditions in space-time: for each $t,s\in\R$ and each~$x,y\in\Z^d$,
\begin{equation}
\label{E:cocycle}
\psi(t,x,\cdot)\circ \tau_{s,y} = \psi(t+s,x+y,\cdot)-\psi(s,y,\cdot)
\end{equation}
with $\psi(0,0,\cdot)=0$,
\item[(3)] $\psi$ is of finite specific energy in the sense that
\begin{equation}
\label{E:3.3}
\E\Bigl(\,\sum_{x\colon|x|=1} a_0(0,x)\bigl|\psi(0,x,a)\bigr|^2\Bigr)<\infty\,,
\end{equation}
\item[(4)] defining the corrector by $\chi(t,x,\cdot):=\psi(t,x,\cdot)-x$ and letting~$p:=2/(1+1/q)>1$,
\begin{equation}
\label{E:3.4}
\chi(t,x,\cdot)\in L^p(\BbbP),\quad \chi(\cdot,x,\cdot)\in \,L^{p,\text{\rm loc}}(\R)\otimes L^p(\BbbP)\quad\text{and}\quad \E\chi(t,x,\cdot)=0
\end{equation}
holds for each~$x\in\Z^d$ and each~$t\in\R$. 
\end{enumerate}
\end{theorem}

\begin{remark}
Theorem~\ref{prop-2.1} fits the setting of Theorem~\ref{thm-1.2} for the choice $k_t:=(1+t)^{-\mu}$ with any $\mu\in (4, \vartheta/2)$ because \eqref{E:2.12aa} implies \eqref{E:2.12a} with $q:=\vartheta/\mu>1$. Such a choice of~$\mu$ can be made since $\frac\vartheta2 > 4$ when $\vartheta> 4d$ (and $d \geq 2$). 
\end{remark}

From Theorem~\ref{prop-2.1} and Lemma~\ref{lemma-2.2} we thus immediately obtain:
 
\begin{corollary}
Under the assumptions of Theorem \ref{thm-1.2}, there exists a measurable function $\psi\colon \R\times\Z^d\times\Omega\to\R^d$ satisfying {\rm (1--4)} in Theorem~\ref{prop-2.1} above.
\end{corollary}

The strategy of our proof of Theorem~\ref{prop-2.1} is as follows: similarly to all existing constructions of the harmonic coordinate, we will solve a suitably perturbed version of \eqref{E:3.1} and then control the removal of the perturbation. As usual, the latter step will be done using functional analytic methods. In~\cite{ACDS16}, which is closest to our setting, even the former step was based on functional analytic tools (namely, the Lax-Milgram lemma) but here we will proceed by more probabilistic arguments.

Let $p^{t,s}(x,y)$, for $t\le s$ and $x,y\in\Z^d$, denote the transition probability of the random walk~$X$ between times~$t$ and~$s$; i.e.,
\begin{equation}
p^{t,s}(x,y):=P(X_s=y|X_t=x).
\end{equation}
We begin by noting the following fact about uniformly elliptic situations:

\begin{lemma}
\label{lemma-3.3}
Let~$\epsilon\in(0,1)$ and suppose, for the moment, that the conductances $t\mapsto a_t(e)$ are Lebesgue measurable and taking values in~$[\epsilon,1/\epsilon]$. Let~$g\colon \R\times\Z^d\to\R$ be bounded and measurable. Then
\begin{equation}
\label{E:3.7}
h(t,x):=-\int_t^\infty\textd s\,\, \e^{-\epsilon(s-t)}\sum_{y\in\Z^d}p^{t,s}(x,y)g(s,y)
\end{equation}
is well defined with $t\mapsto h(t,x)$ continuously differentiable for each~$x\in\Z^d$. Moreover,~$h$ obeys
\begin{equation}
\frac{\partial}{\partial t}h(t,x)-(\epsilon-L_t)h(t,x)=g(t,x)
\end{equation}
at each~$t\in\R$ and~$x\in\Z^d$.
\end{lemma}

\begin{proofsect}{Proof}
Since~$g$ is bounded, the sum in \eqref{E:3.7} converges absolutely and is bounded uniformly in~$s$, hence the integral over~$s$ converges absolutely as well and $h$ is well-defined. The transition probability admits the representation
\begin{equation}
p^{t,s}(x,y)= \delta(x,y) \e^{-\int_t^s\textd u\,\pi_u(x)}+\int_t^s\textd u\,\pi_u(x)\e^{-\int_t^u\textd r\,\pi_r(x)}\sum_{z\colon z\sim x}\frac{a_u(x,z)}{\pi_u(x)}p^{u,s}(z,y),
\end{equation}
where $\delta(x,y)=1$ if $x=y$ and vanishes otherwise, $z\sim x$ means that $(x,z)\in E(\Z^d)$ and $\pi_u(x):=\sum_{z:z\sim x}a_u(x,z)$. Thus, the function
$t,x\mapsto p^{t,s}(x,y)$ obeys the differential equation
\begin{equation}
\frac{\partial}{\partial t}p^{t,s}(x,y) +L_t p^{t,s}(\cdot,y)(x)=0.
\end{equation}
Since the conductances are nearest-neighbor and uniformly bounded, the sum of the derivatives (with respect to~$t$) of the terms in \eqref{E:3.7}, as well as the resulting integral, converge absolutely. Standard criteria permit us to exchange the time derivative with the integral over~$s$ and the sum over~$y$. The result then boils down to a simple calculation which we leave to the reader.
\end{proofsect}

Given a sample of the conductances~$\{a_t(e)\colon e\in E(\Z^d),\,t\in\R\}$ satisfying Assumption \ref{ass1}, we will apply Lemma~\ref{lemma-3.3} to the function~$g$ given by $(t,x)\mapsto -V(t,x,\cdot)$ where
\begin{equation}
\label{E:3.10}
V(t,x,a):=\sum_{z\colon |z|=1}a_t(x,x+z)z.
\end{equation}
However, in order to have the required ellipticity, the random walk will be driven by the collection of perturbed conductances $\{a_t^\epsilon(e)\colon t\in\R,\, e\in E(\Z^d)\}$, where
\begin{equation}
a_t^\epsilon(e):=\epsilon+a_t(e),\qquad e\in E(\Z^d).
\end{equation}
Writing $p_\epsilon^{t,s}(x,y,a)$ for the transition probability of the random walk among conductances~$a_t^\epsilon(e)$, Lemma~\ref{lemma-3.3} then ensures that
\begin{equation}
\label{E:3.12}
\varphi_\epsilon(t,x,\cdot):=\int_t^\infty\textd s\,\, \e^{-\epsilon(s-t)}\sum_{y\in\Z^d}p_\epsilon^{t,s}(x,y,\cdot)V(s,y,\cdot)
\end{equation}
obeys
\begin{equation}
\label{E:3.13}
\frac{\partial}{\partial t}\varphi_\epsilon(t,x,\cdot)+(\epsilon-\epsilon\Delta-L_t)\varphi_\epsilon(t,x,\cdot) = -V(t,x,\cdot),\qquad t\in\R,\,x\in\Z^d.
\end{equation}
Here~$\Delta$ is the discrete Laplacian on~$\Z^d$ acting as $\Delta f(x):=\sum_{y:y\sim x}[f(y)-f(x)]$ and~$L_t$ is the generator derived from the ``bare'' conductances~$a_t(e)$ as in \eqref{E:1.1}. The effect of the term $\epsilon\Delta$ is to make the generator uniformly elliptic; the term~$\epsilon$ (times identity) then represents a killing of the walk at uniform rate~$\epsilon$. 

Our aim is to show that $\varphi_\epsilon(t,x,\cdot)-\varphi_\epsilon(0,0,\cdot)$ converges, as~$\epsilon\downarrow0$, to the desired corrector~$\chi(t,x,\cdot)$ in a suitable sense. This will be done via a sequence of lemmas. First we note that~$\varphi_\epsilon$ satisfies the cocycle conditions in space-time:

\begin{lemma}
For each $\epsilon>0$, each $t\in\R$ and each~$x\in\Z^d$,
\begin{equation}
\label{E:3.14}
\varphi_\epsilon(t,x,\cdot) = \varphi_\epsilon(0,0,\cdot)\circ\tau_{t,x}.
\end{equation}
In particular, for each~$t,s\in\R$ and each~$x,y\in\Z^d$,
\begin{equation}
\label{E:3.15}
\varphi_\epsilon(t+s,x+y,\cdot)-\varphi_\epsilon(s,y,\cdot) = \varphi_\epsilon(t,x,\cdot)\circ\tau_{s,y}-\varphi_\epsilon(0,0,\cdot)\circ\tau_{s,y}.
\end{equation} 
and so $t,x\mapsto\varphi_\epsilon(t,x,\cdot)-\varphi_\epsilon(0,0,\cdot)$ satisfies \eqref{E:cocycle} for every $\varepsilon > 0$.
\end{lemma}

\begin{proofsect}{Proof}
\eqref{E:3.14} follows from \eqref{E:3.12} and the identities $V(t,x,\cdot) = V(0,0,\cdot)\circ\tau_{t,x}$ and $p^{t,s}(x,y,\cdot) = p^{0,s-t}(0,y-x,\cdot)\circ\tau_{t,x}$. The second line follows from \eqref{E:3.14} and $\tau_{t+s,x+y}=\tau_{s,y}\circ\tau_{t,x}$.
\end{proofsect}

Next we observe the validity of some \emph{a priori} estimates:

\begin{lemma}
Under Assumption~\ref{ass1}, for each~$\epsilon>0$,
\begin{equation}
\label{E:3.16}
\epsilon\E\bigl|\varphi_\epsilon(0,0,\cdot)\bigr|^2
\le d
\end{equation}
and
\begin{equation}
\label{E:3.17}
\sum_{i=1,\dots, d}\E\Bigl(a_0(\hate_i)\bigl|\varphi_\epsilon(0,\hate_i,\cdot)-\varphi_\epsilon(0,0,\cdot)\bigr|^2\Bigr)
\le d\,.
\end{equation}
\end{lemma}

\begin{proofsect}{Proof}
Recall that $\varphi_\epsilon$ is bounded; by \eqref{E:3.13} and the definition of~$V$ the same applies to its time derivative as well. This justifies exchanges of limits and expectations in
\begin{equation}
\begin{aligned}
\E\Bigl(\varphi_\epsilon(0,0,\cdot)\cdot\frac{\partial}{\partial t}\varphi_\epsilon(0,0,\cdot)\Bigr)
&=\lim_{t\downarrow0}\,\frac1t\E\Bigl(\varphi_\epsilon(0,0,\cdot)\cdot\bigl(\varphi_\epsilon(t,0,\cdot)-\varphi_\epsilon(0,0,\cdot)\bigr)\Bigr)
\\
&=\lim_{t\downarrow0}\,\frac1t\E\Bigl(\varphi_\epsilon(0,0,\cdot)\cdot\bigl(\varphi_\epsilon(-t,0,\cdot)-\varphi_\epsilon(0,0,\cdot)\bigr)\Bigr)
\\
&=-\E\Bigl(\varphi_\epsilon(0,0,\cdot)\cdot\frac{\partial}{\partial t}\varphi_\epsilon(0,0,\cdot)\Bigr)\,,
\end{aligned}
\end{equation}
where the middle equality follows from \eqref{E:3.14} and invariance of $\mathbb{P}$ under $\tau_{t,0}$. We thus have
\begin{equation}
\label{E:3.170}
\E\Bigl(\varphi_\epsilon(0,0,\cdot)\cdot\frac{\partial}{\partial t}\varphi_\epsilon(0,0,\cdot)\Bigr)=0.
\end{equation}
Taking the inner product of \eqref{E:3.13} at~$x=0$ and~$t=0$ with~$\varphi_\epsilon(0,0,\cdot)$ and then taking expectation yields, on account of~\eqref{E:3.170},
\begin{equation}
\label{E:3.171}
\begin{aligned}
\epsilon\E|\varphi_{\epsilon}(0,0,\cdot)|^2&
+\sum_{i=1,\dots,d}\E\Bigl(a_0^\epsilon(\hate_i)\bigl|\varphi_\epsilon(0,\hate_i,\cdot)-\varphi_\epsilon(0,0,\cdot)\bigr|^2\Bigr)
\\
&= - \E\bigl(V(0,0,\cdot)\cdot\varphi_\epsilon(0,0,\cdot)\bigr)
\\
&\qquad=\sum_{i=1,\dots,d}\E\Bigl(a_0(\hate_i) \,  \hate_i \cdot\bigl(\varphi_\epsilon(0,\hate_i,\cdot)-\varphi_\epsilon(0,0,\cdot)\bigr)\Bigr)
\\
&\qquad\qquad\le\biggl[ d\sum_{i=1,\dots,d}\E\Bigl(a_0(\hate_i)\bigl|\varphi_\epsilon(0,\hate_i,\cdot)-\varphi_\epsilon(0,0,\cdot)\bigr|^2\Bigr)\biggr]^{1/2},
\end{aligned}
\end{equation}
where we used \eqref{E:3.14} and simple symmetrization for the second term on the left hand side and also to obtain the middle equality, and then invoked the Cauchy-Schwarz inequality along with $a_0(e)\le1$ to get the last inequality. Since~$a_0^\epsilon(e)\ge a_0(e)$, foregoing the term $\epsilon\E|\varphi_{\epsilon}(0,0,\cdot)|^2 \geq 0$ yields \eqref{E:3.17} and, by plugging that in on the right-hand side of \eqref{E:3.171}, also \eqref{E:3.16}.
\end{proofsect}

These bounds have the following consequences:

\begin{lemma}
\label{lemma-3.5}
Under the assumptions of Theorem~\ref{prop-2.1}, for $p:= 2/(1+1/q)$ with~$q$ as in \eqref{E:2.12a},  the following holds uniformly on compact sets of $(t,x)\in\R\times\Z^d$\,:
\begin{equation}
\label{E:3.21}
\epsilon\varphi_\epsilon(t,x,\cdot)\,\underset{\epsilon\downarrow0}\longrightarrow\,0\quad\text{\rm in }L^p(\BbbP)
\end{equation}
and
\begin{equation}
\label{E:3.22}
\sup_{0<\epsilon<1}\E\bigl|\varphi_\epsilon(t,x,\cdot)-\varphi_\epsilon(0,0,\cdot)\bigr|^p<\infty.
\end{equation}
\end{lemma}

\begin{proofsect}{Proof}
As $p\in(1,2)$, the first part of the claim follows immediately from \eqref{E:3.16}, H\"older's inequality and \eqref{E:3.14}. For the second part we write $\varphi_\epsilon(t,x,\cdot)-\varphi_\epsilon(0,0,\cdot)= \int_0^t \partial_s \varphi_\epsilon(s,x,\cdot) \text{d}s+ \varphi_\epsilon(0,x,\cdot)- \varphi_\epsilon(0,0,\cdot)$, write the spatial difference as a telescoping sum, and then use \eqref{E:3.15} and \eqref{E:3.13} to obtain
\begin{multline}
\label{E:3.23}
\quad\Bigl(\E\bigl|\varphi_\epsilon(t,x,\cdot)-\varphi_\epsilon(0,0,\cdot)\bigr|^p\Bigr)^{1/p}
\le \epsilon t\Bigl(\E\bigl|\varphi_\epsilon(0,0,\cdot)\bigr|^p\Bigr)^{1/p}
\\+\bigl(|x|_1+2d(1+\epsilon)t\bigr)
\max_{i=1,\dots,d}\Bigl(\E\bigl|\varphi_\epsilon(0,\hate_i,\cdot)-\varphi_\epsilon(0,0,\cdot)\bigr|^p\Bigr)^{1/p}\,.
\quad
\end{multline}
The first term on the right is bounded thanks to \eqref{E:3.16}. For the expectations in the second term, we invoke the weights~$w_t(e)$ from \eqref{E:1.16} and Cauchy-Schwarz to get
\begin{multline}
\quad
\E\bigl|\varphi_\epsilon(0,\hate_i,\cdot)-\varphi_\epsilon(0,0,\cdot)\bigr|^p
\\
\le \Bigl(\E\bigl(w_0(\hate_i)^{-\frac{p}{2-p}}\bigr)\Bigr)^{\frac{2-p}2}\,\biggl(\,\,\E\sum_{i=1,\dots,d}w_0(\hate_i)\bigl|\varphi_\epsilon(0,\hate_i,\cdot)-\varphi_\epsilon(0,0,\cdot)\bigr|^2\biggr)^{p/2}.
\quad
\end{multline}
Since $\frac{p}{2-p} = q$, the first term on the right-hand side is bounded thanks to \eqref{E:2.12a}. Using \eqref{E:2.47}, \eqref{E:3.13}, \eqref{E:3.14} and the identity $(a+b+c)^2\le3a^2+3b^2+3c^2$, the second expectation on the right is no larger than
\begin{multline}
\label{E:3.25}
\quad
48d^2\ci\sum_{i=1,\dots,d}\E\Bigl(a_0(\hate_i)\bigl|\varphi_\epsilon(0,\hate_i,\cdot)-\varphi_\epsilon(0,0,\cdot)\bigr|^2\Bigr)
\\
+72d \ci\Bigl[\epsilon^2\E\bigl(|\varphi_\epsilon(0,0,\cdot)|^2\bigr)+\epsilon^2\E\bigl(|\Delta\varphi_\epsilon(0,0,\cdot)|^2\bigr)+\E\bigl(|V(0,0,\cdot)|^2\bigr)\Bigr].
\quad
\end{multline}
By \twoeqref{E:3.16}{E:3.17}, \eqref{E:3.14}, the fact that~$|V|\le 2d$, and bounding $|\Delta\varphi_\epsilon|$  in terms of $|\varphi_\epsilon|$ (noting that the lattice gradient is a bounded operator), this is bounded uniformly in~$\epsilon\in(0,1)$.
\end{proofsect}

We now set
\begin{equation}
\label{E:3.26chi}
\chi_\epsilon(t,x,\cdot):=\varphi_\epsilon(t,x,\cdot)-\varphi_\epsilon(0,0,\cdot)
\end{equation}
and note: 

\begin{proposition}
\label{lemma-3.6}
Under the assumptions of Theorem~\ref{prop-2.1}, and with~$p:=2/(1+1/q)>1$ for~$q$ as in \eqref{E:2.12a}, there is a sequence $\epsilon_n\downarrow0$ and a measurable function $\chi\colon\R\times\Z^d\times\Omega\to\R^d$ such that for each~$x\in\Z^d$,
\begin{equation}
\label{E:3.27z}
\chi_{\epsilon_n}(\cdot,x,\cdot)\,\underset{n\to\infty}\longrightarrow\,\chi(\cdot,x,\cdot)\quad\text{\rm weakly in }L^{p,\text{\rm loc}}(\R)\otimes L^p(\BbbP)
\end{equation}
and, for each~$t\in\R$,
\begin{equation}
\label{E:3.27a}
\chi_{\epsilon_n}(t,x,\cdot)\,\underset{n\to\infty}\longrightarrow\,\chi(t,x,\cdot)\quad\text{\rm weakly in }L^p(\BbbP).
\end{equation}
Moreover, on a set of full~$\BbbP$-measure, $\chi$ is normalized so that $\chi(0,0,\cdot)=0$, obeys the cocycle conditions
\begin{equation}
\label{E:3.28}
\chi(t+s,x+y,\cdot)-\chi(t,x,\cdot) = \chi(s,y,\cdot)\circ\tau_{t,x},\qquad t,s\in\R,\,x,y\in\Z^d,\,
\end{equation}
and $t\mapsto\chi(t,x,\cdot)$ is continuous and weakly differentiable with
\begin{equation}
\label{E:3.29}
\frac{\partial}{\partial t}\chi(t,x,\cdot) + L_t\chi(t,x,\cdot) = -V(t,x,\cdot)
\end{equation}
for all~$x\in\Z^d$ and all~$t\in\R$.
\end{proposition}

The bounds of Lemma \ref{lemma-3.5} will readily allow us to take weak limits as $\varepsilon \downarrow 0$. A slightly subtle point, see Lemma \ref{lem2.cont} below, is to choose a version of the resulting limiting process which has continuous trajectories. Once this is achieved, the proof of Proposition~\ref{lemma-3.6} will quickly follow.

We start with a few observations. We are henceforth tacitly working under the assumptions of Proposition~\ref{lemma-3.6}. Let~$r$ be the H\"older conjugate of~$p$; the fact that~$p>1$ (and the fact that~$\Omega$ is a standard Borel space, hence separable) ensures that the dual space $L^r(\R)\otimes L^r(\BbbP)$ is separable. In light of the uniform bound \eqref{E:3.22}, Cantor's diagonal argument ensures the existence of a sequence~$\epsilon_n\downarrow0$ and functions~$ \phi\colon\R\times\Omega\to\R^d$ and $\rho\colon\Z^d\times\Omega\to\R^d$ such that for any~$\xi\in L^r(\R)\otimes L^r(\BbbP)$ with compact support in the first coordinate,
\begin{equation}
\label{E:3.30}
\int\textd t\,\E\bigl(\xi(t,\cdot)\cdot\chi_{\epsilon_n}(t,0,\cdot)\bigr)
\,\underset{n\to\infty}\longrightarrow\,\int\textd t\,\E\bigl(\xi(t,\cdot)\cdot \phi(t,\cdot)\bigr)
\end{equation}
and, for any~$\xi\in L^r(\BbbP)$ and any~$x\in\R^d$,
\begin{equation}
\label{E:3.31}
\E\bigl(\xi\cdot\chi_{\epsilon_n}(0,x,\cdot)\bigr)
\,\underset{n\to\infty}\longrightarrow\, \E\bigl(\xi\cdot\rho(x,\cdot)\bigr).
\end{equation}
Standard arguments give
\begin{equation}
\label{E:3.32}
\phi\in L^{p,\text{loc}}(\R)\otimes L^p(\BbbP)
\quad\text{and}\quad
\rho(x,\cdot)\in L^p(\BbbP)
\end{equation}
for every~$x\in\Z^d$. A key point in what follows is:

\begin{lemma} 
\label{lem2.cont}
The process $\{ \phi(t,\cdot)\colon \, t \in \mathbb{R} \}$ admits a version $\{ \tilde\phi(t,\cdot) \colon \, t \in \mathbb{R} \}$ which has $\mathbb{P}$-a.s. continuous sample paths. Moreover, on a set of full~$\BbbP$-measure, this version obeys
\begin{equation}
\label{eq:cont6a}
\tilde{\phi}(t,\cdot) = -  \int_0^t\textd s\Bigl( V(s,0, \cdot) + (L_0\rho)(0,\cdot)\circ\tau_{s,0}+\sum_{z: \, |z|=1}a_s(0,z)\tilde\phi(s,\cdot)\circ\tau_{0,z} \Bigr),
\end{equation}
for all~$t\in\R$.
\end{lemma}

\begin{proof}
Consider the auxiliary process $\Xi_{\varepsilon}(t,\cdot)$ defined as
\begin{equation}
\label{E:cont1}
\Xi_{\varepsilon}(t,\cdot) := \chi_\epsilon(t,0,\cdot)-\int_0^t \textd s\, (L_s\chi_\epsilon)(s,0,\cdot). 
\end{equation}
First note that, since $\chi_\epsilon(0,0,\cdot)$ vanishes, \eqref{E:3.13} and \eqref{E:3.14} yield that
\begin{equation}
\Xi_{\varepsilon}(t,\cdot) +  \int_0^t \textd s \, [(\varepsilon -\varepsilon \Delta) \varphi_{\epsilon}](s,0, \cdot) =  -\int_0^t \textd s\,V(s,0,\cdot),
\end{equation}
Hence
\begin{equation}
\label{E:cont2}
\E\biggl|\,\Xi_\epsilon(t,\cdot)+\int_0^t \textd s\,V(s,0,\cdot)\biggr|
\le (2+4d)|t|\epsilon\E|\varphi_\epsilon(0,0, \cdot)|,
\end{equation}
for all $t$. On account of \eqref{E:3.21} and with $\epsilon_n$ as defined above \eqref{E:3.30}, we thus get, for any bounded interval $I \subset \mathbb{R}$ and with $\lambda_I$ denoting the Lebesgue measure on~$I$,
\begin{equation}
\label{E:cont3}
\bigg\Vert \, \Xi_{\epsilon_n}(\cdot ,\cdot) + \int_0^{\cdot} \textd s\,V(s,0,\cdot) \bigg\Vert_{L^1(\lambda_I \otimes \mathbb{P})} \,\underset{n \to \infty}\longrightarrow\, 0\,.
\end{equation}
In particular, $- \int_0^{\cdot} \textd s\,V(s,0,\cdot)$ is a weak limit in $L^p(\lambda_I \otimes \mathbb{P})$ of the sequence $\Xi_{\epsilon_n}(\cdot ,\cdot)$. 

Now pick any $\xi \in L^r(\R) \otimes L^r(\mathbb{P})$ with compact support in the first variable. We then claim the validity of
\begin{multline}
\label{eq:cont5}
\lim _{n \to \infty}\int\textd t\,\E\bigl(\xi(t,\cdot)\cdot \Xi_{\epsilon_n}(t,\cdot)\bigr)
\\= \int\textd t\,\E \Biggl( \xi(t,\cdot)\cdot \biggl[ \phi(t,\cdot) + \int_0^t\textd s\Bigl( (L_0\rho)(0,\cdot)\circ\tau_{s,0}+\sum_{z: |z|=1}a_s(0,z)\phi(s,\cdot)\circ\tau_{0,z} \Bigr) \biggr]\Biggr).
\end{multline}
Indeed, we first note the rewrite
\begin{equation}
\label{E:3.37}
(L_s\chi_\epsilon)(s,0,\cdot)=(L_0\chi_\epsilon)(0,0,\cdot)\circ\tau_{s,0}
+\sum_{z:|z|=1}a_s(0,z)\bigl(\chi_\epsilon(s,0,\cdot)\circ\tau_{0,z}\bigr).
\end{equation}
Abbreviating
\begin{equation}
\tilde{\xi}(\cdot) :=\int\textd t\int_0^t \textd s \sum_{z\colon|z|=1}a_0(0,z)\bigl[\xi(t,\cdot) \circ \tau_{-s,z}-\xi(t,\cdot) \circ \tau_{-s,0}\bigr],
\end{equation}
 the convergence statement \eqref{E:3.31} along with Fubini and the invariance of $\mathbb{P}$ under space-time shifts show
\begin{equation}
\label{E:cont5a}
\begin{aligned}
&\int\textd t\,\E\biggl(\, \xi(t,\cdot) \cdot\Big[ \int_0^t\textd s\bigl( (L_0\chi_{\epsilon_n})(0,0,\cdot)\circ\tau_{s,0} \bigr) \Big]\biggr)
=\E \big[ \, \tilde{\xi}(\cdot) \cdot\chi_{\epsilon_n}(0,0,\cdot) \big] \\
&\qquad\underset{n\to\infty}\longrightarrow  \,\,\E \big[ \, \tilde{\xi}(\cdot) \cdot \rho(0,0,\cdot) \big] = 
\int\textd t\,\E \biggl(\, \xi(t,\cdot) \cdot\Big[ \int_0^t\textd s\bigl( (L_0\rho)(0,\cdot)\circ\tau_{s,0} \bigr) \Big]\biggr),
\end{aligned}
\end{equation}
where to get the second line we also noted that $\tilde{\xi}\in L^q(\lambda_{\R}\otimes \mathbb{P})$, by invariance of $\mathbb{P}$ under time-shifts, Jensen's inequality, boundedness of $a_s(e)$ and the fact that~$\xi$ has compact support in the~$t$-variable. A similar computation applies to the term involving $\chi_\epsilon(s,0,\cdot)$ on the right of \eqref{E:3.37}. Indeed, setting
\begin{equation}
\hat\xi(s,\cdot):=\int \textd t\bigl(\1_{[0,t]}(s)-\1_{[-t,0]}(s)\bigr) \sum_{z\colon|z|=1}a_0(0,z)\bigl[\xi(t,\cdot) \circ \tau_{0,z}-\xi(t,\cdot)\bigr]
\end{equation}
we get
\begin{equation}
\label{E:cont5b}
\begin{aligned}
&\int\textd t\,\E \biggl(\, \xi(t,\cdot)\cdot \Big[ \int_0^t\textd s\Bigl(\,\sum_{z\colon|z|=1} a_s(0,z)\bigl(\chi_\epsilon(s,0,\cdot)\circ\tau_{0,z}\bigr) \Bigr) \Big]\biggr)
\\&\qquad= \int\textd s\,   \E \Big[ \, \hat{\xi}(s,\cdot) \cdot \chi_{\epsilon_n}(s,0,\cdot) \Big]
\,\underset{n\to\infty}\longrightarrow\,  \int\textd s\,   \E \Big[ \, \hat{\xi}(s,\cdot) \cdot \phi(s,\cdot) \Big] 
\\&\qquad= 
\int\textd t\,\E \biggl(\, \xi(t,\cdot)\cdot \Big[ \int_0^t\textd s\Bigl(\,\sum_{z\colon|z|=1} a_s(0,z)\bigl(\phi(s,0,\cdot)\circ\tau_{0,z}\bigr) \Bigr) \Big]\biggr),
\end{aligned}
\end{equation}
using \eqref{E:3.30} instead. In light of \eqref{E:cont1} and \eqref{E:3.37}, \twoeqref{E:cont5a}{E:cont5b} yield \eqref{eq:cont5}.

The weak limit in \eqref{eq:cont5} being unique (as implied by the Hahn-Banach theorem), \eqref{E:cont3} implies that, on a set of full $\lambda_{\R} \otimes \mathbb{P}$-measure, $- \int_0^{t} \textd s\,V(s,0,\cdot)$ agrees with the term in square brackets on the right-hand side of \eqref{eq:cont5}. It follows that~$\tilde\phi$ defined as
\begin{equation}
\label{eq:cont6}
\tilde{\phi}(t,\cdot) := -  \int_0^t\textd s\Bigl( V(s,0, \cdot) + (L_0\rho)(0,\cdot)\circ\tau_{s,0}+\sum_{z: \, |z|=1}a_s(0,z)\phi(s,\cdot)\circ\tau_{0,z} \Bigr) ,
\end{equation}
equals~$\phi$ on a set of full $\lambda_{\R} \otimes \mathbb{P}$-measure. But this also implies that we can substitute~$\tilde\phi$ for~$\phi$ in \eqref{E:cont5b} which shows that~$\tilde\phi$ obeys \eqref{eq:cont6a} $\lambda_{\R} \otimes \mathbb{P}$-almost everywhere. As $\tilde\phi$ has $\mathbb{P}$-a.s.\ continuous sample paths, a routine use of Fubini's Theorem shows that \eqref{eq:cont6a} extends to all~$t\in\R$ on a set of full~$\BbbP$-measure.
\end{proof}

We are now ready to complete:

\begin{proofsect}{Proof of Proposition~\ref{lemma-3.6}}
Letting~$\rho$ be as defined in \eqref{E:3.31} and writing $\tilde\phi$ for the continuous version of~$\phi$ as constructed in the proof of Lemma~\ref{lem2.cont}, we set
\begin{equation}
\label{eq:chi_DEF}
\chi(t,x,\cdot):=\rho(x,\cdot)+\tilde{\phi}(t,\cdot)\circ\tau_{0,x}
\end{equation}
and proceed to check the desired properties. The convergence statements \twoeqref{E:3.27z}{E:3.27a} follow directly from \twoeqref{E:3.30}{E:3.31} while \eqref{E:3.28} is a consequence of \eqref{E:3.15}. With the help of an analogue of \eqref{eq:cont5} (formulated for~$\tilde\phi$) and \eqref{E:3.28}, the equality \eqref{eq:cont6a}  translates into
\begin{equation}
\chi(t,x)=\rho(x)-\int_0^t\textd s\bigl( V(s,0,\cdot)+(L_s\chi)(s,x,\cdot)\bigr).
\end{equation}
Hereby \eqref{E:3.29} readily follows (with the derivative even in Lebesgue sense). 
\end{proofsect}

\begin{proofsect}{Proof of Theorem~\ref{prop-2.1}}
Let~$\chi$ be as constructed in Proposition~\ref{lemma-3.6} and set
\begin{equation}
\psi(t,x,\cdot):=x+\chi(t,x,\cdot).
\end{equation}
Then \eqref{E:3.1} follows from \eqref{E:3.31} while \eqref{E:cocycle} from \eqref{E:3.30}. The identity \eqref{E:3.3} is a consequence of \eqref{E:3.17} and the fact that weak convergence in~$L^{p}$ contracts $L^{p'}$-norms. The integrability conditions in \eqref{E:3.4} follow readily from \twoeqref{E:3.27z}{E:3.27a}.
Since $\E\chi_\epsilon(t,x,\cdot)=0$ for each~$\epsilon>0$, this implies also the last condition in \eqref{E:3.4}.
\end{proofsect}

We finish by a lemma that will be useful in some definitions below:

\begin{lemma}
\label{lemma-growth}
For each~$x\in\Z^d$, there is a random variable $C(x,\cdot)>0$ with $\BbbP(C(x,\cdot)<\infty)=1$ such that
\begin{equation}
\bigl|\chi(t,x,\cdot)\bigr|\le C(x,\cdot)\bigl(\,1+|t|\bigr),\qquad t\in\R.
\end{equation}
\end{lemma}

\begin{proofsect}{Proof}
Pick~$r\in(0,(2d)^{-1})$. Using fact that $a_s(e)\le1$ in Lemma~\ref{lem2.cont} then shows, for each~$t\in\R$,
\begin{equation}
\E\Bigl(\,\sup_{0<s<r}\bigl|\tilde\phi(t+s,\cdot)-\tilde\phi(t,\cdot)\bigr|\Bigr)\le\frac{cr}{1-2dr}
\end{equation}
where $c\in(0,\infty)$ is a constant related to the $L^1$-norm of~$\rho(x, \cdot)$ for~$|x|=1$. Since the increments of~$t\mapsto\tilde\phi(t,\cdot)$ are also stationary due to \eqref{E:3.28}, the ergodic theorem implies
\begin{equation}
C_1(\cdot):=\sup_{t\in\R}\frac{|\tilde\phi(t,\cdot)|}{1+|t|}<\infty,\qquad\BbbP\text{-a.s.}
\end{equation}
As $\chi(t,x,\cdot)=\rho(x,\cdot)+\tilde\phi(t,\cdot)\circ\tau_{0,x}$, cf. \eqref{eq:chi_DEF}, the claim follows with the choice $C(x,\cdot):=|\rho(x,\cdot)|+C_1(\cdot)\circ\tau_{0,x}$.
\end{proofsect}


\section{Proof of invariance principle}
\label{sec3}\nopagebreak\noindent
The goal of this section is to give a proof of the main result, which involves showing that the corrector constructed in Section \ref{sec2} is sublinear in a strong ($L^{\infty}$) sense. We proceed by first showing a corresponding statement on average (i.e., in $L^1$-sense, with respect to the norms introduced in Section~\ref{sec-2a}), see Proposition~\ref{prop-3.7} below. This result is then boosted to a sublinearity result in $L^{\infty}$-sense in Theorem~\ref{P:sublin}, which is proved by obtaining a maximal inequality using a Moser iteration approach, see Proposition~\ref{P:chi_Linfty}, whose proof is deferred to Section~\ref{sec6}. Conditionally on Proposition~\ref{P:chi_Linfty}, the proof of Theorem \ref{thm-1.2} is completed at the end of the present section. 

We will occasionally invoke the Maximal Ergodic Theorem for commuting measure preserving transformations throughout the section. We refer to Krengel \cite[Section 6.2]{K85} for further details.
 
\subsection{Sublinearity on average}
\label{subl1}\noindent
We begin with an \emph{a priori} estimate on the $L^1$-norm of the corrector which constitutes a version of ``sublinearity on average.'' This will serve as a starting point for the Moser iteration developed in the next section. Recall the definitions of~$B_n$ and~$\zeta_n$ from \eqref{E:2.27}, \eqref{eq:zeta_n}, with $\zeta$ satisfying \twoeqref{eq:kernel_conds}{E:2.2a}, and the norms $\Vert\cdot\Vert_{p,q;B,\zeta}$ from \eqref{eq:norm_pp'}. The desired statement is as follows:

\begin{proposition}
\label{prop-3.7}
Let~$\chi$ be the corrector constructed in Theorem~\ref{prop-2.1}. Then
\begin{equation}
\lim_{n\to\infty}\,\frac1{n^{d+1}}\,\Vert\chi\Vert_{1,1;B_n,\zeta_n} = 0,\qquad\BbbP\text{\rm-a.s.}
\end{equation}
\end{proposition}

Although we could in principle follow the proof of Proposition~3.3 in~\cite{ACDS16}, we found a different argument. We begin with two lemmas, both of which are formulated for rectangles of the form
\begin{equation}
\label{E:4.2ua}
R_n:=\bigl([a_1n,b_1n)\times\dots\times[a_dn,b_dn)\bigr)\cap\Z^d\,,
\end{equation}
where $a_1,\dots,a_d,b_1,\dots,b_d\in\R$ are numbers that obey~$a_i<b_i$, $i =1,\dots,d$.
Without further mention, we assume in the remainder of Section \ref{subl1} that $\chi$ is the object constructed in Theorem~\ref{prop-2.1}, and we implicitly work under the assumptions of that theorem.

\smallskip
The starting point of the proof is the following observation:

\begin{lemma}
\label{lemma-3.8}
For any sequence~$\{R_n\}$ as above,
\begin{equation}
\label{E:4.3ua}
\frac1{n^{d+1}}\int_0^\infty\textd t\,\frac1n\int_0^n\textd s\,\,\zeta_n(t)\sum_{x\in R_n}\bigl[\chi(t,x,\cdot)-\chi(t+s,x,\cdot)\bigr]
\,\underset{n\to\infty}\longrightarrow\,0,\quad\BbbP\text{\rm-a.s.}
\end{equation}
\end{lemma}

\begin{proofsect}{Proof}
Let~$f_n$ denote the quantity on the left-hand side. The cocycle property gives
\begin{equation}
\label{E:4.4new}
-f_n=\int_0^\infty\textd t\,\zeta_n(t)\,g_n\circ\tau_{t,0}
\end{equation}
where
\begin{equation}
\label{E:4.5new}
g_n:=\frac1{n^{d+2}}\sum_{x\in R_n}\int_0^n\textd s\,\,\chi(s,0,\cdot)\circ\tau_{0,x}\,.
\end{equation}
We first claim
\begin{equation}
\label{E:4.6new}
g_n\,\underset{n\to\infty}\longrightarrow\,0,\quad\BbbP\text{-a.s.}
\end{equation}
For this we invoke \eqref{E:3.29} and some elementary integration to write
\begin{equation}
\label{E:4.7new}
\int_0^n\textd s\,\,\chi(s,0,\cdot) = -\int_0^n\textd s\,\,(n-s)\,\Bigl[V(0,0,\cdot)+L_0\chi(0,0,\cdot)\Bigr]\circ\tau_{s,0}\,.
\end{equation}
Plugging this into \eqref{E:4.5new}, noting that~$V(0,0,\cdot)$ is bounded (and thus in~$L^1(\BbbP)$) and while the boundedness of~$a_s(e)$ and \eqref{E:3.4} ensure $L_0\chi(0,0,\cdot)\in L^1(\BbbP)$, the Spatial Ergodic Theorem shows
\begin{equation}
g_n\,\underset{n\to\infty}\longrightarrow\,\,c\,\E\Bigl[V(0,0,\cdot)+L_0\chi(0,0,\cdot)\Bigr]\quad\BbbP\text{-a.s.}
\end{equation}
where~$c:=\prod_{i=1}^d (b_i-a_i)$. The stationarity of~$\BbbP$ with respect to spatial shifts ensures that the expectation vanishes and so we get \eqref{E:4.6new} as claimed.

Let now~$p>1$ be such that \eqref{E:3.4} holds. Next we claim that, for each~$\tilde p\in(1,p)$,
\begin{equation}
\label{E:4.10new}
\sup_{m\ge n}|g_m|\,\underset{n\to\infty}\longrightarrow\,0,\quad\text{in }L^{\tilde p}(\BbbP).
\end{equation}
In light of \eqref{E:4.6new} and the Dominated Convergence Theorem, for this it suffices to show
\begin{equation}
\label{E:4.11new}
\sup_{n\ge1}|g_n|\in L^{\tilde p}(\BbbP).
\end{equation}
Here one more use of \eqref{E:4.7new} shows
\begin{equation}
|g_n|\le\frac1{n^{d+1}}\sum_{x\in R_n}
\sum_{k=0}^{n-1}\biggl(\int_0^1\textd s\,\Bigl|V(0,0,\cdot)+L_0\chi(0,0,\cdot)\Bigr|\circ\tau_{s,0}\biggr)
\circ\tau_{k,x}
\end{equation}
The quantity in the large parentheses is in~$L^p(\BbbP)$ thanks to \eqref{E:3.4} and so \eqref{E:4.11new} follows from the Maximal Ergodic Theorem for space-time shifts.

In order to prove the desired result, fix $1<\hat p<\tilde p<p$ for~$p$ as in \eqref{E:3.4}. For each~$N\in\N$ \eqref{E:4.4new} then gives
\begin{equation}
|f_n|\le\int_0^\infty \textd t\,\zeta_n(t)\,\bigl(\sup_{m\ge N}|g_m|\bigr)\circ\tau_{t,0},\quad n\ge N.
\end{equation}
Lemma~\ref{lemma-sup} ensures
\begin{equation}
\bigl\Vert \sup_{n\ge N}|f_n|\bigr\Vert_{L^{\hat p}(\BbbP)}
\le c \bigl\Vert \sup_{m\ge N}|g_m|\bigr\Vert_{L^{\tilde p}(\BbbP)}
\end{equation}
for some $c=c(\hat p,\tilde p)$. The norm on the right tends to zero as~$N\to\infty$ by \eqref{E:4.10new}. The Dominated Convergence Theorem then implies $\limsup_{n\to\infty}|f_n|=0$ a.s.\ as desired.
\end{proofsect}

For the rest of the proof, we will work with the quantity
\begin{equation}
\tilde\chi_n(t,x,\cdot):=\int_0^\infty\textd u\,\,\zeta_n(u)\chi(t+u,x,\cdot),
\end{equation}
where the integral converges absolutely thanks to Lemma~\ref{lemma-growth} and our assumption of integrability of~$t\mapsto (1+|t|)\zeta(t)$.
We then have:

\begin{lemma}
\label{lemma-3.9}
For any~$\{R_n\}$ as above,
\begin{equation}
\label{E:4.6ua}
\frac1{n^{d+2}}\sum_{x\in R_n}\int_0^n\textd s\,\tilde\chi_n(s,x,\cdot)
\,\underset{n\to\infty}\longrightarrow\,0,\quad\BbbP\text{\rm-a.s.}
\end{equation}
\end{lemma}

For the proof we will need the following fact:

\begin{lemma}
\label{lemma-approx}
Suppose~$p\ge1$ is such that \eqref{E:3.4} holds.
For each~$\epsilon>0$ there is a measurable $h_\epsilon\colon\Omega\to\R^d$ with $h_\epsilon\in L^{ p}(\BbbP)$ such that for all~$x\in\Z^d$,
\begin{equation}
\label{E:4.7ui}
\E \Bigl| \int_0^1\textd t\,\big(h_\epsilon-h_\epsilon\circ\tau_{t,x}-\chi(t,x,\cdot)\big)\Bigr|^{ p}\,\underset{\epsilon\downarrow0}\longrightarrow\,0.
\end{equation}
\end{lemma}

We note that in earlier constructions of the corrector (including the one in~\cite{ACDS16}) the property in Lemma~\ref{lemma-approx} follows more or less directly. Although we also obtain~$\chi(t,x,\cdot)$ as a limit of the quantities $\varphi_\epsilon(0,0,\cdot)\circ\tau_{t,x}-\varphi_\epsilon(0,0,\cdot)$, this limit is only in the weak sense and we do not presently see a way to boost it to a strong convergence as required above. 

An alternative approach would be to regard~$x\mapsto\chi(0,x,\cdot)$ as an element of the $L^2$-space of cocycle vector fields with inner product $(u,v):=\E\sum_{e\colon|e|=1}a_0(e)u(e,\cdot)\cdot v(e,\cdot)$ and show that it can be approximated by a potential field; i.e., one of the form $h_\epsilon-h_\epsilon\circ\tau_{t,x}$. Even if the existence of these approximations could be checked, we would still not know how to proceed as we no longer have a direct way to convert weighted $L^2$-norms into $L^1$-norms. (Indeed, the energy conversion applies only to solutions of the inhomogenous heat equation.) Our proof of Lemma ~\ref{lemma-approx}, which we defer to the Appendix, proceeds by a direct argument inspired (with some necessary corrections) by derivations in Biskup and Spohn~\cite{BS11}.

\smallskip

\begin{proofsect}{Proof of Lemma~\ref{lemma-3.9}}
Fix~$p>1$ as appearing above \eqref{E:3.4}. The conclusion of Lemma~\ref{lemma-approx} holds and, given~$\epsilon>0$, let~$h_\epsilon$ be as in \eqref{E:4.7ui}. Define
\begin{equation}
\tilde\chi_{n,\epsilon}(t,x,\cdot):=\int_0^\infty\textd u\,\,\zeta_n(u)\int_0^1\textd s\,\Bigl[\chi(t+s+u,x,\cdot) - \bigl(h_\epsilon\circ\tau_{t+s+u,x}-h_\epsilon\bigr)\Bigr]
\end{equation}
where the integrals again converge absolutely by Lemma~\ref{lemma-growth} and the assumed integrability conditions on~$\zeta$.
Abbreviating also
\begin{equation}
\tilde h_{n,\epsilon}:=\int_0^\infty\textd u\,\,\zeta_n(u)\int_0^1\textd s\,\, h_\epsilon\circ\tau_{s+u,0},
\end{equation}
which converges absolutely by the last clause of Lemma~\ref{lemma-sup}, it is now easy to check
\begin{multline}
\qquad\frac1{n^{d+2}}\sum_{x\in R_n}\int_0^n\textd s\,\tilde\chi_n(s,x,\cdot)
=\frac1{n^{d+2}}\sum_{x\in R_n}\sum_{t=0}^{n-1}\,\tilde\chi_{n,\epsilon}(t,x,\cdot)
\\
+\frac1{n^{d+2}}\sum_{x\in R_n}\sum_{t=0}^{n-1}\,\tilde h_{n,\epsilon}\circ\tau_{t,x}
-\frac{|R_n|}{n^{d+1}}\Vert\zeta\Vert_{L^1}\,h_\epsilon(\cdot).
\qquad
\end{multline}
Since~$h_\epsilon\in L^p(\BbbP)$ for~$p>1$, the same holds true for $\int_0^1\textd s\,\, h_\epsilon\circ\tau_{s,0}$, and Lemma~\ref{lemma-sup} gives us that $\sup_{n\ge1}|\tilde h_{n,\epsilon}|\in L^1(\BbbP)$. The Spatial Ergodic Theorem then shows that the second term on the right tends to zero as~$n\to\infty$. The same also applies trivially to the last term, and so we just need to control the first term on the right.

Let $\FF:=\sigma(a_t(e)\colon t\in\R,\,e\in E(\Z^d))$ be the $\sigma$-algebra generated by the conductances and enlarge the probability space to include independent random variables $T,X_1,\dots,X_d$, independent of~$\FF$, with~$T$ uniform on~$[0,1)$ and $X_i$ uniform on $[a_i,b_i)$ for each $i=1,\dots,d$. Writing $\lfloor X n\rfloor$ to abbreviate the vector $(\lfloor X_1n\rfloor,\dots,\lfloor X_d n\rfloor)$ and denoting $|\bar R|:=\prod_{i=1}^d(b_i-a_i)$, we get
\begin{multline}
\label{E:4.13uiu}
\quad
\biggl|
\frac1{n^{d+2}}\sum_{x\in R_n}\sum_{t=0}^{n-1}\,\tilde\chi_{n,\epsilon}(t,x,\cdot)
-\frac{|\bar R|}n\E\Bigl(\tilde\chi_{n,\epsilon}\bigl(\lfloor Tn\rfloor,\lfloor Xn\rfloor,\cdot)\,\Big|\,\FF\Bigr)
\biggr|
\\
\le\frac1{n^{d+2}}\sum_{(x,y)\in E(R_n)}\,\sum_{t=0}^{n-1}\,\bigl|\tilde\chi_{n,\epsilon}(t,y,\cdot)-\tilde\chi_{n,\epsilon}(t,x,\cdot)\bigr|
\le \frac1n\int_0^\infty\zeta_n(u) f\circ\tau_{u,0}\,,
\quad
\end{multline}
where
\begin{equation}
f:=\sup_{n\ge1}\,\frac1{n^{d+1}}\sum_{(x,y)\in E(R_n)}\,\sum_{t=0}^{n-1}
 \int_0^1\textd s\, \Bigl| \chi(0,y-x,\cdot) + (h_{\epsilon} - h_{\epsilon} \circ \tau_{0,y-x})  \Bigr|\circ\tau_{s+t,x}.
\end{equation}
and where the last step follows by invoking the definition of~$\tilde\chi_{n,\epsilon}$ along with the cocycle property.
The Maximal Ergodic Theorem for space-time shifts gives $f\in L^{\tilde p}(\BbbP)$ for some~$\tilde p\in(1,p)$ and Lemma~\ref{lemma-sup} then ensures that \eqref{E:4.13uiu} converges, as~$n\to\infty$, to zero~$\BbbP$-a.s. Thus, if we can show
\begin{equation}
\label{E:4.8ue}
\lim_{n\to\infty}\,\,\frac1n\,\E\Bigl(\chi_{n,\epsilon}\bigl(Tn,\lfloor Xn\rfloor,\cdot)\,\Big|\,\FF\Bigr) =0,\quad\BbbP\text{\rm-a.s.}
\end{equation}
the claim will follow. 

The advantage of working in this ``continuum'' representation is that it makes telescoping arguments more manageable. Indeed, by the cocycle property we can write
\begin{equation}
\tilde\chi_{n,\epsilon}\bigl(\lfloor Tn\rfloor,\lfloor Xn\rfloor,\cdot)
=\sum_{k=0}^{n-1}\tilde\chi_{n,\epsilon}\Bigl(\lfloor T(k+1)\rfloor-\lfloor Tk\rfloor,\lfloor X(k+1)\rfloor-\lfloor Xk\rfloor,\cdot\Bigr)\circ\tau_{\lfloor Tk\rfloor,\lfloor Xk\rfloor}
\end{equation}
Now note that $\lfloor T(k+1)\rfloor-\lfloor Tk\rfloor\in\{0,1\}$ while $\lfloor X(k+1)\rfloor-\lfloor Xk\rfloor$ has $\ell^\infty$-norm bounded by some $r\in\N$ independent of~$n$. Introducing
\begin{equation}
g_\epsilon:=\sup_{n\ge1}\,\sum_{t=0,1}\sum_{z\colon|z|_\infty\le r}\bigl|\tilde\chi_{n,\epsilon}(t,z,\cdot)\bigr|,
\end{equation}
and denoting $\Lambda_k:=\{0,\dots,k\}\times R_k$, we thus have
\begin{equation}
\frac1n\biggl|\,\E\Bigl(\tilde\chi_{n,\epsilon}\bigl(Tn,\lfloor Xn\rfloor,\cdot)\,\Big|\,\FF\Bigr)\biggr|
\le\frac cn\sum_{k=1}^{n-1}\frac1{k^d}\sum_{(t,x)\in\Lambda_k} g_\epsilon\circ\tau_{t,x}.
\end{equation}
Lemma~\ref{lemma-sup} and \eqref{E:3.4} ensure that $g_\epsilon\in L^{\tilde p}(\BbbP)$ for~$\tilde p\in(1,p)$.
By the Spatial Ergodic Theorem, the normalized second sum on the right converges to $\E g_\epsilon$ and so does the Cezaro average over $k=0,\dots,n-1$. But \eqref{E:sup-bd}, Jensen's inequality along with the cocycle property and the triangle inequality for the $L^{\tilde p}$-norm show
\begin{equation}
\E g_\epsilon\le \Vert g_\epsilon\Vert_{L^{\tilde p}(\BbbP)}\le
c\sum_{z\colon|z|_\infty\le r}\,\biggl[\E \Bigl| \int_0^1\textd s \bigl( h_\epsilon-h_\epsilon\circ\tau_{s,x}-\chi(s,x,\cdot)\bigl)\Bigr|^p\biggr]^{1/p}
\end{equation}
for some~$c\in(0,\infty)$ depending only on~$p$, $\tilde p$, $d$ and $\zeta$. Lemma~\ref{lemma-approx} then gives $\E g_\epsilon\to0$ as~$\epsilon\downarrow0$ thus proving \eqref{E:4.8ue} as desired.
\end{proofsect}

As an immediate consequence we get:

\begin{corollary}
\label{cor-4.5}
For any~$\{R_n\}$ as above,
\begin{equation}
\frac1{n^{d+1}}\int_0^\infty\textd t\,\,\zeta_n(t)\sum_{x\in R_n}\chi(t,x,\cdot)
\,\underset{n\to\infty}\longrightarrow\,0,\quad\BbbP\text{\rm-a.s.}
\end{equation}
\end{corollary}

\begin{proofsect}{Proof}
This follows by combining \eqref{E:4.3ua} with \eqref{E:4.6ua}.
\end{proofsect}

We are now ready to give:

\begin{proofsect}{Proof of Proposition~\ref{prop-3.7}}
We adapt part of the argument from page~227 in Sidoravicius and Sznitman~\cite{SS04}. (The argument cannot be used directly as it relies on square integrability of the corrector as well as separate ergodicity.) Denote $\bar\chi_B(t,\cdot):=|B|^{-1}\sum_{x\in B}\chi(t,x,\cdot)$ and, given $L\ge1$, let~$\{R_{n,i}\colon i=1,\dots,m(n)\}$ be the enumeration of sets of the form $(\lfloor n (z/L+[0, 1/L)^d) \rfloor\cap\Z^d$ with $z\in\Z^d$ that have a non-empty intersection with~$B_n$. Denote $\tilde B_n:=\bigcup_{i=1}^{m(n)}R_{n,i}$. Then Lemma~\ref{lemma-Sob} and a routine (by now) use of Cauchy-Schwarz inequality show
\begin{equation}
\Vert\chi\Vert_{1,1;R_{n,i},\zeta_n}\le \Vert\bar\chi_{R_{n,i}}\Vert_{1,1;R_{n,i},\zeta_n}+c\frac nL\Bigl(\Vert w^{-1}\Vert_{1,1;E(R_{n,i}),\zeta_n}
\,\EE^{w,\zeta_n}_{R_{n,i}}(\chi)\Bigr)^{1/2}.
\end{equation}
Now sum this over~$i=1,\dots,m(n)$ and apply Cauchy-Schwarz inequality one more time to get
\begin{equation}
\label{E:3.42}
\Vert\chi\Vert_{1,1;B_n,\zeta_n}
\le \sum_{i=1}^{m(n)}\Vert\bar\chi_{R_{n,i}}\Vert_{1,1;R_{n,i},\zeta_n}
+c\frac nL\,\Vert w^{-1}\Vert_{1,1;E(\tilde B_n),\zeta_n}^{1/2}\,\EE^{w,\zeta_n}_{\tilde B_n}(\chi)^{1/2}\,.
\end{equation}
Corollary~\ref{cor-4.5}  and the fact that~$m(n)$ is at most order~$L^d$ ensures 
\begin{equation}
\frac1{n^{d+1}}\sum_{i=1}^{m(n)}\Vert\bar\chi_{R_{n,i}}\Vert_{1,1;R_{n,i},\zeta_n}\,\underset{n\to\infty}\longrightarrow\,0,\quad\BbbP\text{-a.s.}
\end{equation}
Lemma~\ref{L:nu_w} in turn gives
\begin{equation}
\sup_{n\ge1}\,\,\frac1{n^d}\Vert w^{-1}\Vert_{1,1;E(\tilde B_n),\zeta_n}<\infty,\quad\BbbP\text{-a.s.}
\end{equation}
Since~$\chi$ solves \eqref{E:3.29} with~$V$ bounded, \eqref{E:2.46} and \eqref{E:3.3} also show
\begin{equation}
\sup_{n\ge1}\,\,\frac1{n^d}\EE^{w,\zeta_n}_{\tilde B_n}(\chi)<\infty,\quad\BbbP\text{-a.s.}
\end{equation}
The claim now follows from \eqref{E:3.42} by taking~$n\to\infty$ followed by~$L\to\infty$.
\end{proofsect}

\subsection{Sublinearity everywhere and proof of main result}
Recall the definition of the corrector from the previous section. Our next goal is to boost the $L^1$-sublinearity to an $L^\infty$-version. Define the diffusive space-time cylinder
\begin{equation}
\label{E:Q(n)}
Q(n):=\bigl\{(x,t)\colon x\in B_n,\,0\le t\le n^2\bigr\}.
\end{equation}
We now claim that the corrector is sublinear on diffusive scale of space and time:

\begin{theorem}\label{P:sublin}
Suppose Assumption~\ref{ass1} holds and assume, in addition, \eqref{E:q_ass}. Then 
\begin{equation}\label{eq:sublin}
\lim_{n\to \infty} \max_{(t,x) \in Q(n)}\frac{|\chi(t,x, \cdot)|}{n}=0, \quad \BbbP\text{-a.s.}.
\end{equation}
\end{theorem}

Recalling the notation $\vvert\cdot\vvert_{p,q;B,\zeta}$ for the normalized norms from \eqref{eq:norm_2}, the key point of the proof of this claim is the following proposition valid for general solutions to the heat equation. We state it in a form which will be sufficient to deduce Theorem \ref{thm-1.2}. A more general version of the following result can be found in Corollary \ref{C:itera100}.

\begin{proposition}[$L^1$ to $L^\infty$ bootstrap]
\label{P:chi_Linfty}
Suppose Assumption \ref{ass1} as well as the moment bound \eqref{E:q_ass} hold. There exist functions $k,\zeta$ satisfying \eqref{eq:kernel_conds}-\eqref{E:2.2aa} and constants $\gamma_1 \in (0,\infty)$ and $c,c'\in(1,\infty)$ (all depending on $d$ and the choice of $k,\zeta$) such that, if $u \colon \R\times\Z^d\times\Omega\to\R^d$ is a (measurable) weak solution to 
\begin{equation}
\label{eq:chi_Linfty1}
\frac\partial{\partial t} u(t,x,\cdot) + L_t u(t,x,\cdot) = L_t f,
\end{equation}
for some bounded $f\colon \R\times\Z^d\times\Omega\to\R$ satisfying $|f(y)-f(x)| \leq \frac1n$ for all $(x,y)\in E(B_n)$ and all~$n\ge1$, then for all $r \in (2d, \frac{\vartheta}{2})$,
\begin{equation}
\label{eq:maxbd}
 \max_{(t,x) \in Q(n)}|u(t,x)| \leq c\, W(r)^{\gamma_1} \vvert  u \vvert_{1,1; B_{2n},\zeta_n}^{\gamma_2(n,u)}
 \end{equation}
 where $1\le\gamma_2(n,u)\le c'$ and $\zeta_n$ is defined in \eqref{eq:zeta_n} and
 \begin{equation}
 W(r) := \sup_{n \geq1} \max_{m\in [n,2n]}  \bvvert w^{-1} \bvvert_{r,r;E(B_{m}),\zeta_n}
 \end{equation}
satisfies
\begin{equation}
\label{eq:Wfinite}
W(r) < \infty, \quad  \text{for all $r \in (2d, \vartheta/2)$.}
\end{equation}
\end{proposition}

Deferring the proof of Proposition~\ref{P:chi_Linfty} to Section~\ref{sec6}, let us show how it implies our main result. We begin with:

\begin{proofsect}{Proof of Theorem~\ref{P:sublin}}
Since the corrector obeys the equation \eqref{E:3.29}, this is immediate from Lemma~\ref{L:nu_w}, Proposition~\ref{prop-3.7},  Proposition~\ref{P:chi_Linfty} and the boundedness of~$V$.
\end{proofsect}

Next we note the standard fact:

\begin{lemma}
\label{lemma-ergodic}
Suppose Assumption~\ref{ass1} and, given a sample of $a=\{a_t(e)\colon e\in E(\Z^d),\, t\in\R\}$, let~$\{X_t\colon t\ge0\}$ be a sample of the random walk. The process $t\mapsto \tau_{t,X_t}a$ on~$\Omega$ is then Markov with a unique stationary measure~$\BbbP$. Moreover, the process is ergodic in the sense that, for any function~$f\in L^1(\BbbP)$, we have
\begin{equation}
\label{E:4.37ui}
\frac1t\int_0^t \textd t\,f(\tau_{t,X_t}a)\,\underset{t\to\infty}\longrightarrow\, \E f
\end{equation}
for $\BbbP$-a.e.\ $a\in\Omega$ and $P^0$-a.e.\ sample of~$\{X_t\colon t\ge0\}$.
\end{lemma}

\begin{proofsect}{Proof}
The stationarity and reversibility of~$\BbbP$ is verified easily by a standard generator calculation and the limit in \eqref{E:4.37ui} exists by the Ergodic Theorem. The only item where caution is needed is ergodicity which ensures that the limit value in \eqref{E:4.37ui} is constant $\BbbP$-a.s., and thus equal to~$\E f$. This boils down to showing that any event~$A\subset\Omega$ which is invariant under the Markov shift $t\mapsto \tau_{t,X_t}a$ is a zero-one event.

We build on an argument in Andres~\cite[Proposition~2.1]{A14}. Let~$A$ be as above. For each~$t\ge0$, we then have $1_A = \sum_{x\in\Z^d}p^{0,t}(0,x)\1_A\circ\tau_{t,x}$ $\BbbP$-a.s.\ and so
\begin{equation}
0=\1_{A^\cc}\1_A=\sum_{x\in\Z^d}\1_{A^\cc}p^{0,t}(0,x)\1_A\circ\tau_{t,x}\,.
\end{equation}
Taking expectation and dropping all but one term from the sum yields
\begin{equation}
\label{E:4.39ui}
\E\bigl(1_{A^\cc}p^{0,t}(0,x)\1_A\circ\tau_{t,x}\bigr)=0,\qquad x\in\Z^d,\,t\ge0.
\end{equation}
Since $a_t(e)\le1$ implies $p^{0,t}(0,0)\ge\e^{-2dt}$, choosing~$x=0$ gives $\E(1_{A^\cc}\1_A\circ\tau_{t,0})=0$ for all~$t\ge0$. Applying $\tau_{-t,0}$ under expectation and swapping the roles of~$A$ and~$A^\cc$ then shows $\1_A\circ\tau_{t,0}=\1_A$ $\BbbP$-a.s.\ for each~$t\in\R$, i.e.,~$A$ is time-shift invariant $\BbbP$-a.s.

Next pick a neighbor~$e$ of the origin and apply \eqref{E:4.39ui} to~$x:=e$. Injecting the restriction $T_e<t$ into the expectation, we thus get
\begin{equation}
\E\bigl(\1_{\{T_e<t\}}1_{A^\cc}p^{0,t}(0,e)\1_A\circ\tau_{t,e}\bigr)=0.
\end{equation}
But time-shift invariance of~$A$ shows $\1_A\circ\tau_{t,e} = \1_A\circ\tau_{0,e}$ $\BbbP$-a.s.\ and, on $\{T_e<t\}$, we have
\begin{equation}
p^{0,t}(0,e)\ge \int_0^t \textd s\, \e^{-2ds} a_s(e)  \e^{-2d(t-s)} \ge \e^{-2d t}.
\end{equation}
It follows that $\E(\1_{\{T_e<t\}}1_{A^\cc}\1_A\circ\tau_{0,e})=0$. Taking~$t\to\infty$ and using that $T_e<\infty$ $\BbbP$-a.s.\ (by Assumption~\ref{ass1}(3)) we now again get $\1_A\circ\tau_{0,e}=\1_A$ $\BbbP$-a.s. Hence $A$ is invariant with respect to all space-time shifts a.s.; ergodicity of~$\BbbP$ then implies that $\BbbP(A)\in\{0,1\}$ as desired.
\end{proofsect}


We are now ready to give the:

\begin{proofsect}{Proof of Theorem~\ref{thm-1.2}}
Let~$\psi$ be the harmonic coordinate constructed in Theorem~\ref{prop-2.1} and let~$\{X_t\colon t\ge0\}$ be a sample of the random walk. Let $\FF_t:=\sigma(X_s\colon 0\le s\le t)$. The equation~\eqref{E:3.1} then implies that $\{\psi(t,X_t,\cdot),\FF_t\}_{t\ge0}$ is a martingale. Letting~$v\in\R^d$, the quadratic variation process of~$t\mapsto v\cdot\psi(t,X_t,\cdot)$ is given by
\begin{equation}
\bigl\langle v\cdot\psi(t,X_t,\cdot)\bigr\rangle_t=\int_0^t\textd s\, f\circ\tau_{s,X_s}(a)\,,
\end{equation}
where
\begin{equation}
f(a):=E^0\biggl(\,\sum_{z\colon|z|=1}a_0(0,z)\bigl[v\cdot\psi(0,z,\cdot)\bigr]^2\biggr).
\end{equation}
In light of \eqref{E:3.3} and Lemma~\ref{lemma-ergodic}, the conditions of the Lindeberg-Feller Martingale Functional Central Limit Theorem are satisfied. Hence $t\mapsto n^{-1}\psi(tn^2,X_{tn^2},\cdot)$ scales as~$n\to\infty$ to a Brownian motion with variance as in~\eqref{E:1.14}. 

In order to complete the proof of convergence of $t\mapsto n^{-1}X_{tn^2}$ to Brownian motion, it suffices to show that, for $\BbbP$-a.e.\ realization of the environment,
\begin{equation}
\label{E:4.24}
\sup_{0\le s\le t}\frac1{\sqrt t}\bigl|\psi(s,X_s,\cdot)-X_s\bigr|\,\underset{t\to\infty}\longrightarrow\,0,\quad\text{in $P^0$-probability}\,.
\end{equation}
This is shown by noting that, for any $M\ge1$ and any~$\epsilon>0$,
\begin{multline}
\qquad
P^0\Bigl(\,\sup_{0\le s\le t}\bigl|\psi(s,X_s,\cdot)-X_s\bigr|>\epsilon\sqrt t\Bigr)
\\
\le P^0\Bigl(\,\sup_{0\le s\le t}\bigl|\psi(s,X_s,\cdot)\bigr|>M\sqrt t\Bigr)+1_{\{\max_{(s,x)\in Q(M\sqrt t)}|\chi(s,x,\cdot)|>\epsilon\sqrt t\}}.
\qquad
\end{multline}
For any fixed $M\ge1$ and $\epsilon>0$, Theorem~\ref{P:sublin} ensures that the indicator is zero for~$t$ sufficiently large~$\BbbP$-a.s. On the other hand, in the limit as~$t\to\infty$ followed by~$M\to\infty$, the probability on the right tends to zero by the above convergence of $t\mapsto n^{-1}\psi(X_{tn^2})$ to Brownian motion. This implies \eqref{E:4.24}. 

In order to show that the limiting covariance~$\Sigma$ is non-degenerate suppose that $v\cdot\Sigma v=0$ for some~$v\in\R^d$. Then \eqref{E:1.14} and the cocycle conditions imply $L_t(v\cdot\psi)(t,x,\cdot)=0$ for all $t$ and $x$ and thus by the differential equation, see \eqref{E:3.1}, the function $t\mapsto v\cdot\psi(t,x,\cdot)$ is constant for each $x\in \Z^d$. However, Assumption \ref{ass1}(3) ensures that $t\mapsto a_t(e)$ is positive eventually and so this means that $v\cdot\psi(0,x,\cdot)=0$ $\BbbP$-a.s. If~$v\ne0$, this violates the sublinearity of~$\chi$ from Theorem~\ref{P:sublin} and so we must have~$v=0$ after all.
\end{proofsect}


\section{Maximal inequality via Moser iteration}
\label{sec6}\noindent
The aim of this section is give a proof of the maximal inequality for the corrector stated in Proposition~\ref{P:chi_Linfty}. The proof is based on Moser-iteration technique whose main input is the ``one-step estimate'' stated in  Proposition~\ref{P:one-step} below. 
In this section we provide the proof of Proposition~\ref{P:chi_Linfty} conditional on the one-step estimate; this estimate is then proved in Section~\ref{sec6a}.

\subsection{Cut-offs and the one-step estimate}
Let us start with the statement of the one-step estimate. This will require working under spatial and temporal mollifiers (or smooth cut-offs), denoted by $\eta$ and~$\xi$ respectively, that will be assumed to obey the following conditions:

\begin{definition}
Given finite sets $B_1 \subset B_2\subset\mathbb{Z}^d$, and parameters $\delta \in (0,1)$, $\rho \geq 1$ and $M \geq 1$, we say that the (cut-off) functions $\kappa_1,\kappa_2 \colon [0,\infty)\times \mathbb{Z}^d \to [0,1]$ are \textit{$(B_1,B_2)$-adapted with parameters $(\delta, \rho, M)$} if, for $i=1,2$, these functions take the form
\begin{equation}
\label{E:onestep1}
\kappa_i (t,x)= \xi_i(t) \eta_i(x), \qquad t \geq 0 , \ x \in \Z^d,
\end{equation}
with  $\xi_i\colon [0, \infty) \to [0,1]$ and $\eta_i\colon \Z^d \to [0,1]$ satisfying
\begin{equation}
\textnormal{supp}(\eta_i) \subset B_i  \text{ for } i=1,2 \quad\text{and}\quad  \eta_2(x) = 1 \text{ for } x\in B_1,  \label{E:onestep2}
\end{equation}
and
\begin{equation}
\xi_i \in C^1  \text{ for } i=1,2 \quad\text{and}\quad  
\xi_1(t ) \leq M \xi_2 (t)^{\rho} , \ |\bigdot{\xi_1}(t)| \leq \delta M \xi_2 (t)^{\rho}, \quad t \geq 0. \label{E:onestep3}
\end{equation}
\end{definition}

The spatial mollifier $\eta_i$ should be thought of as a ``smooth'' version of the indicator of~$1_{B_i}$. Note that the conditions in~\eqref{E:onestep2} imply $\eta_1 \leq \eta_2$. An explicit construction of functions $\eta_i$ and~$\xi_i$ is provided below in Lemma~\ref{L:cutoff}.

In order to state the one-step estimate, we need some more notation. Given $e \in {E}(\Z^d)$ (and recalling that edges are unoriented), specify one of its endpoints as its initial vertex~$x_e$ and the other as $y_e$ (the choice will not matter in the sequel). Then abbreviate
\begin{equation}
\label{E:grad}
\nabla f(e) := f(y_e) - f(x_e).
\end{equation}
In what follows, we will write $\Vert \cdot \Vert_{\ell^p(K)}$ to denote the $\ell^p$-norm with respect to the \textit{counting} measure, for any $K \subseteq \mathbb{Z}^d$ (or, if functions on edges are considered, $K \subseteq E(\mathbb{Z}^d)$) and any $p >0$. We denote by $L^\infty(\R_+)$ the set of Lebesgue-a.e.~bounded functions supported in~$[0,\infty)$.
Recall also the notation $\hat p(\alpha):= \frac{\alpha}{2}\frac d{d-1}$ and $\hat q(\beta):= \frac{\beta}{2}$ for  the ``Sobolev exponents'' from \eqref{eq:p_s}, and the normalized norms $\vvert \cdot \vvert_{p,q; B, \zeta}$ from \eqref{eq:norm_2}. We assume throughout that $\zeta$ satisfies \twoeqref{eq:kernel_conds}{E:2.2a} and we are interested in weak solutions to the inhomogenous equation
\begin{equation}
\label{E:5.4}
\frac\partial{\partial t} u(t,x) + L_t u(t,x) = L_tf(t,x), \quad t \geq 0, \, x\in \Z^d.
\end{equation}
We are only interested in the specific case $f(t,x)=x$, see \eqref{E:3.29} and \eqref{E:3.10}, for which a (weak) solution to \eqref{E:5.4} has been constructed in Proposition \ref{lemma-3.6}, but the following results only require that $ \Vert\nabla f\Vert_{\ell^\infty(E)}$ be finite. The ``one-step estimate'' is now the content of:

\begin{proposition}[One-step Moser iteration]
\label{P:one-step}
Let $d\ge2$ and suppose Assumption~\ref{ass1} holds. For all $\alpha \in (2, 2\frac{d-1}{d-2})$, all $\beta \in (0,2)$ and all $q > 1$ and~$p$ defined by
 \begin{equation}
  \label{E:onestep4}
 \frac1{p} := \frac{\theta}{\hat p(\alpha)} + 1-\theta,\quad\text{ where } \theta := \frac{\hat q(\beta)}{q} \in (0,1)\,,
 \end{equation} 
there is $\cii=\cii(d, \alpha, \beta, q)\in (0,\infty)$ such that the following holds for any weak solution~$u$ of inhomogenous heat equation \eqref{E:5.4}: For all finite $B_1 \subset B_2 \subset \Z^d$, all $\delta > 0$, all $\rho \in [1, p\wedge q)$, all $M \geq 1$, all $\lambda_1 \geq 2 $ and all $(B_1,B_2)$-adapted functions $\kappa_1, \kappa_2$ with parameters $(\delta, \rho,M)$ we have
\begin{equation}
\label{E:onestep5}
\vvert \,  \kappa_1^{2/\lambda_1} u \, \vvert_{ \lambda_1 p , \, \lambda_1 q ; \, \, B_1, \zeta
} \le (A_{1,2})^{1/\lambda_1} \, \vvert \, \kappa_2^{2/\lambda_2} u \,\vvert_{  \lambda_2 p , \, \lambda_2 q ; \, \, B_2, \zeta}^\gamma\,,
\end{equation}
where $\lambda_2 := \lambda_1 / \rho$,
\begin{equation}
\label{E:onestep5.1}
\gamma := 
\begin{cases}
 1-\frac{2}{\lambda_1}, &\text{if }\ \vvert \, \xi_2^{2\rho} |u|^{\lambda_1} \, \vvert_{1,1;B_1, \zeta} < 1, \\ 1 &\text{otherwise}
 \end{cases}
\end{equation}
(in particular, $\gamma  \in [0,1]$) and the prefactor $A_{1,2}$ takes the explicit form
\begin{equation}
\label{E:onestep5.2}
A_{1,2} :=  \cii(\lambda_1^2M)^2  \, \Vert \zeta \Vert_{L^1}\, \bigl( 1\vee \vvert \,w^{-1} \,\vvert_{ \frac r2 ,  \, \frac s2 ; \, \, B_1, \zeta}\bigr) \frac{|B_2|}{|B_1|} \Biggl[(\Gamma + \delta)\biggl( \frac1{\inf_{t \in \Sigma_1} \zeta(t)} + |B_1|^{\frac2d}\biggr) \Biggr]  
\end{equation}
with $r,s$ related to~$\alpha,\beta$ as in \eqref{E:1.4a}, $ \Sigma_1 :=  \text{\rm supp}(\xi_1)$  and
\begin{equation}
\label{E:onestep6}
\Gamma :=  \Vert\nabla f\Vert_{\ell^\infty(E)}^2 + \bigl\Vert(\nabla f) (\nabla \eta_1)\bigr\Vert_{\ell^\infty(E)} + \Vert \nabla \eta_1 \Vert_{\ell^\infty(E)}^2 + \Vert \bigdot{\zeta}/ \zeta\Vert_{L^{\infty}(\mathbb{R}_+)}\,. 
\end{equation}
Here~$f\colon\Z^d\to\R$ is the function on the right of \eqref{E:5.4}.
\end{proposition}

The proof also exhibits the following estimate, which we record for later purposes:

\begin{corollary}
\label{R:one-step}
For the setting, notations and under the conditions of Proposition~\ref{P:one-step},
\begin{equation}
\label{E:onestep14}
\bvvert \,  \kappa_1^2  |u|^{\lambda_1} \, \bvvert_{1 , \,  \infty ; \, \, B_1,\zeta}^{1/\lambda_1} \leq (A_{1,2})^{1/\lambda_1} \  \vvert \, \kappa_2^{2/\lambda_2} u \, \vvert_{\lambda_2p,\,\lambda_2q \,;\,B_2, \zeta}^\gamma.
\end{equation}
\end{corollary}

\begin{remark}
\label{R:onestep001}
The allowed range of $\alpha$ implies that $\hat p(\alpha)\in (\frac{d}{d-1}, \frac{d}{d-2})$ and, in particular, that $\hat p(\alpha)>1$. It follows that $p$ defined in \eqref{E:onestep4} satisfies $p > 1$. As a consequence, $\rho \in (1, p\wedge q)$ can always be found so that $\lambda_2<\lambda_1$ (as will be desired).

The prefactor $A_{1,2}$ collects all dependencies on the cut-off functions as well as the norm $\vvert \,w^{-1} \,\vvert_{ \frac r2 ,  \, \frac s2 ; \, \, E(B_1), \zeta}$, which we will control via Lemma~\ref{L:nu_w}. Our choices of parameters will eventually ensure that the term in square brackets on the right-hand side of \eqref{E:onestep5.2} is of order unity, and so~$A_{1,2}$ is basically order-$(\lambda_1^2 M)^2$. Both~$\lambda_1$ and~$M$ will change through iterations, but in such a way that the overall product of prefactors of the type $(A_{1,2})^{1/\lambda_1}$ arising from subsequent iterations remains bounded. 
\end{remark}

Proposition~\ref{P:one-step} is where the principal novel ingredients of the present work enter the proof of Moser iteration; the rest is more or less just an adaptation of the arguments in~\cite{ACDS16}. Deferring the proof of Proposition~\ref{P:one-step} to Section~\ref{sec6a}, we now proceed to discuss these adaptations and give the proof of Proposition~\ref{P:chi_Linfty}.

\subsection{Iteration}
The fact that $\lambda_1$ in Proposition~\ref{P:one-step} can be rather arbitrary, and~$\rho$ can be set to a quantity in excess of one (see Remark~\ref{R:onestep001}), offers the possibility to apply the inequality in \eqref{E:onestep5} iteratively to bound high-$(p,q)$-norms of the solution to the Poisson equation by low-$(p,q)$-norms thereof. As we also need to keep the quantity in \eqref{E:onestep5.2} bounded, this means that the underlying domains, and thus also the mollifiers, will have to vary  throughout the iteration. The discrete nature of the underlying lattice only allows us to run the iteration a limited number of times, albeit increasing with the size of the initial domain. Another iterative argument will thus have to be invoked afterwards to convert the high-$(p,q)$-norm to the maximum over the space-time box~$Q(n)$. This will then readily yield Proposition~\ref{P:chi_Linfty}.

Let us begin by introducing the needed notation. We will consider underlying domains that depend on two adjustable real-valued parameters $\sigma$ and~$\sigma'$ which satisfy
\begin{equation}
\label{E:itera0}
1 \leq \sigma' < \sigma \leq 2. 
\end{equation}
These parameters are introduced only for the sake of the second iteration and they will remain unchanged throughout the first iteration.
Given $n \geq 1$, consider a decreasing sequence of boxes $(B_{n,k})_{k\geq 0}$ such that
\begin{equation}
\label{E:itera1}
B_{n,k}:= B(0,\sigma_k n) ,\quad \text{where}\quad \sigma_k := \sigma'+2^{-k}(\sigma-\sigma').
\end{equation}
We then have
\begin{equation}
\label{eq:itera2}
B_{n}:= B(0,n) \subseteq B_{\sigma' n} \subseteq B_{n,k}\subseteq B_{n,k-1} \subseteq B_{\sigma n} \subseteq B_{2n},\qquad k\ge0.
\end{equation} 
Next we introduce the cut-off functions (depending implicitly on the choice of $\sigma$ and $\sigma'$)
\begin{equation}
\label{E:itera2.1}
\kappa_{n,k}(t,x) := \xi_{n,k}(t) \eta_{n,k}(x)
\end{equation}
as follows: For all $k \geq 0$, the function $\eta_{n,k}\colon\Z^d\to [0,1]$ satisfies
\begin{equation}
\begin{aligned}
\label{eq:itera3}
\text{supp}(\eta_{n,k}) \subset B_{n,k}, \quad &\eta_{n,k}=1 \text{ on $B_{n,k+1}$}\quad\text{and}\quad
\Vert \nabla \eta_{n,k} \Vert_{\ell^{\infty}(E)} \leq \frac{1}{(\sigma_k - \sigma_{k+1})n}\,.
\end{aligned}
\end{equation}
(This can be achieved by interpolating linearly between $B_{n,k+1}$ and $B_{n,k}^\cc$.) Denoting
\begin{equation}
\label{eq:itera5}
\mathfrak b(t):=
\begin{cases}
1, &\text{ if } t\leq 0\\
\exp\Big( 1 - \frac{1}{1-t^2} \Big) , &\text{ if } t \in ( 0,1)\\
0, &\text{ if } t\geq 1,
\end{cases}
\end{equation}
the function $\xi_{n,k}\colon [0,\infty) \to [0,1]$ is defined as 
\begin{equation}\label{eq:itera4}
\xi_{n,k}(t):= \mathfrak b\bigg( \, \frac{(t/n^2) -  \tau_k }{\Delta_{\sigma,\sigma'} }  \,\bigg) = \mathfrak b\bigg( \, \frac{t -  \tau_k n^2}{(\tau_k+ \Delta_{\sigma,\sigma'} )n^2 -  \tau_k n^2}  \,\bigg), \ 
\end{equation}
where
\begin{equation}\label{eq:itera4.1}
\Delta_{\sigma,\sigma'}:= \frac{\sigma-\sigma'}{2} \quad\text{and}\quad\tau_k := \sigma' +  \Delta_{\sigma,\sigma'} \sum_{\ell=k+1}^{\infty}\delta_{\ell} \quad  \text{with}\quad \delta_{\ell} := \frac{6}{\pi^2} \ell^{-2}.
\end{equation}
As seen from the rewrite in \eqref{eq:itera4}, $\xi_{n,k}$ equals~$1$ on $[0,\tau_k n^2]$ and then drops smoothly to~$0$ over the interval $[\tau_k n^2, (\tau_k+ \Delta_{\sigma,\sigma'} )n^2]$. Observe in addition that $\delta_\ell\in[0,1)$ are such that $\sum_{\ell\ge1}\delta_\ell=1$ and that $k\mapsto\tau_k$ is decreasing with $\tau_0= \frac{\sigma + \sigma'}2$ and $\lim_{k\to\infty} \tau_k = \sigma'$. For later purposes we also record that for \text{all} $n,k \geq 0$,
\begin{align}
& \xi_{n,k}(t) = 1, \quad\text{ for } t\leq \sigma' n^2 \label{eq:itera4.2}\\
&  \xi_{n,k}(t) = 0, \quad\text{ for } t\geq \sigma n^2.\label{eq:itera4.3}
\end{align}
Note that $\eta_{n,k}$, $\xi_{n,k}$, $\tau_k$ all depend  implicitly on the choice of $\sigma'$ and $\sigma$ satisfying \eqref{E:itera0}.

\smallskip
To see that the above choices are reasonable, we note:

\begin{lemma}
\label{L:cutoff} For all $\sigma',\sigma$ satisfying \eqref{E:itera0}, all $n,k \geq 1$ and all $\rho \geq 1$, the functions $\kappa_{n,k}$, $\kappa_{n,k-1}$ defined by \eqref{E:itera2.1}, \eqref{eq:itera3} and \eqref{eq:itera4} are $(B_{n,k}, B_{n,k-1})$-adapted with parameters $( \frac{1}{n^2}, \rho , M_k)$, where
\begin{equation}
M_k := (1 \vee   \Delta_{\sigma,\sigma'} ^{-1} \Vert\bigdot{\mathfrak b} \Vert_{L^\infty}) \text{\rm e}^{\rho/ \delta_k}.
\end{equation}
\end{lemma}

\begin{proof}
The conditions \eqref{E:onestep2} hold on account of \eqref{eq:itera3} (in particular, note that $\eta_{n,k-1}=1$ on $B_{n,k}$). As for \eqref{E:onestep3}, first note that $\xi_{n,k} \in C^{\infty}$. It thus remains to show that 
\begin{equation}
\label{eq:itera6}
 \xi_{n,k}(t ) \leq M_k \xi_{n,k-1}(t)^{\rho} \quad\text{and}\quad \bigl|\bigdot{\xi}_{n,k}(t)\bigr| \leq \frac1{n^2} M_k \xi_{n,k-1}(t)^{\rho},\qquad t \geq 0.
\end{equation}
For $t \geq  (\tau_k +  \Delta_{\sigma,\sigma'} )n^2$ we have $\xi_{n,k}(t) = \bigdot{\xi}_{n,k}(t)=0$ and so these bounds hold trivially. 
In the range $t \leq \tau_{k-1}n^2$, we have $\xi_{n,k-1}^\rho(t)=1$ and so the first bound is immediate, while the second follows from
\begin{equation}
\label{E:5.22uiu}
\bigl|\bigdot{\xi}_{n,k}(t)\bigr| \leq \Delta_{\sigma,\sigma'}^{-1}  n^{-2}\Vert \bigdot{\mathfrak b} \Vert_{L^\infty}\,.
\end{equation}
It remains to deal with the case $\frac t{n^2} \in ( \tau_{k-1},  \tau_k +  \Delta_{\sigma,\sigma'})$. For $\frac{t}{n^2}$ in this interval, we observe 
\begin{equation}
\frac{1}{ \xi_{n,k-1} (t)^{\rho}} \stackrel{\eqref{eq:itera4}}{=} \mathfrak b\bigg( \, \frac{(t/n^2) -  \tau_{k-1} }{\Delta_{\sigma,\sigma'} }  \,\bigg)^{-\rho}
\leq \sup_{s\in (0,1-\delta_k)} \big[\mathfrak b(s)^{-\rho} \big]= \e^{ \frac{\rho}{1-(1-\delta_k)^2}-\rho} \leq \e^{\rho/\delta_k},
\end{equation}
and so $M_k\xi_{n,k-1}(t)\ge 1\vee\Delta_{\sigma,\sigma'}^{-1}\Vert \bigdot{\mathfrak b} \Vert_{\infty}$.
The first bound in \eqref{eq:itera6} then follows immediately since $\xi_{n,k}\leq 1$ while the second is obtained by invoking \eqref{E:5.22uiu} one more time. 
\end{proof}

Lastly, we recall the definition \eqref{eq:zeta_n} of $\zeta_n$, $n \geq1$, obtained from $\zeta$, cf. \twoeqref{eq:kernel_conds}{E:2.2a}, by a (diffusive) rescaling. Let
\begin{equation}
\label{eq:c31}
\ciii = \ciii(\zeta) := 1 \vee \Vert \zeta \Vert_{L^1} \vee \Vert \bigdot{\zeta}/ \zeta\Vert_{L^{\infty}(\mathbb{R}_+)} \vee \Bigl(\,\,\inf_{t \in [0,2]} \zeta(t)\Bigr)^{-1}. 
\end{equation}
A recursive application of Proposition~\ref{P:one-step} then yields:

\begin{proposition}[Moser iteration]
\label{P:iter1}
Suppose Assumption~\ref{ass1} and \eqref{E:2.2aa} hold. For all $d\ge2$, all $\alpha \in (2, 2\frac{d-1}{d-2})$, all $\beta \in (0,2)$, all $q > 1$ and $p$ as defined in \eqref{E:onestep4}, there is $\civ=\civ(\alpha,\beta, q, d) \in [1,\infty)$ such that, for all $\rho \in [1, p\wedge q)$, all integers $n \geq 1$, $k > N$, where
\begin{equation}
\label{eq:iteraN}
N:= N(\rho)= \inf \{  k \geq 1; \, \rho^k > 2 \}-1,
\end{equation}
and all  weak solutions $u$ of \eqref{E:5.4} with~$f$ on the right-hand side satisfying $\Vert\nabla f \Vert_{\ell^{\infty}(E)}\leq \frac1n$, we have
\begin{equation}
\label{eq:itera8}
\bvvert \,  \kappa_{n,k}^{2/ \rho^k} u \, \bvvert_{ \rho^k p , \, \rho^k q ; \, \, B_{n,k},\zeta_n} 
\le  \biggl[\frac{\cv W}{(\sigma-\sigma')^4} \biggr]^{\sum_{\ell=1}^k \rho^{-\ell}}\, 
\e^{3 \rho \sum_{\ell=1}^k \ell^2\rho^{-\ell}} \, 
\bvvert \, \kappa_{n,0}^{2} \, u \,\bvvert_{ p' , \,  q' ; \, \, B_{\sigma n} , \zeta_n}^{\bar\gamma (n,k)},
\end{equation}
where $p' = p'(\rho) = \rho^{N} p$, $q' = \rho^{N}q$, $\cv :=\civ(\ciii)^2\cii$, with $\ciii$ given by \eqref{eq:c31} and $\cii$ as in Proposition~\ref{P:one-step},
 \begin{equation}
  \label{eq:itera9}
 W:= 1 \vee \sup_{n\geq 1} \sup_{m \in [n,2n]} \vvert \,w^{-1} \,\vvert_{ \frac r2 ,  \, \frac s2 ; \, \, E(B_{m}), \zeta_n}
 \end{equation}
and where $\bar\gamma(n,k) \in (0,1]$ is defined as
 \begin{equation}
  \label{eq:itera10}
 \bar\gamma(n, k) := \prod_{\ell=N+1}^{k} \gamma_{n,\ell}, \  \text{ with } \gamma_{n,\ell} := 
 \begin{cases}
 1-2\rho^{-\ell}, &\text{if }\ \vvert \, \xi_{n,\ell-1}^{2\rho} |u|^{\rho^\ell} \, \vvert_{1,1;B_{n,\ell}, \zeta_n} < 1, \\ 1, &\text{else}.
 \end{cases}
 \end{equation}
\end{proposition}
 \begin{proof}
Let $n \geq 1$, $k > N$ be integers. In view of \eqref{E:2.2aa}, $\zeta_n$ satisfies conditions \twoeqref{eq:kernel_conds}{E:2.2a}, hence we may apply Proposition~\ref{P:one-step} for the choices $\zeta:= \zeta_n$, $B_1 := B_{n,k}$, $B_2 := B_{n,k-1}$, so that $B_1 \subset B_2$ by~\eqref{E:itera1}, the mollifiers $\kappa_1:= \kappa_{n,k}$ and $\kappa_2 :=\kappa_{n,k-1}$, which are $(B_{n,k}, B_{n,k-1})$-adapted with parameters $( \frac{1}{n^2}, \rho , M_k)$ by Lemma~\ref{L:cutoff}, and $\lambda_1 := \rho^k$, which satisfies $\lambda_1 > 2$ by \eqref{eq:iteraN} and since $k > N$. Noting that $\gamma_{n,k}$ as defined in \eqref{eq:itera10} corresponds precisely to $\gamma$ in \eqref{E:onestep5.1}, the one-step estimate \eqref{E:onestep5} reads
 \begin{equation}
  \label{eq:itera11}
  \bvvert \,  \kappa_{n,k}^{2/ \rho^k} u \, \bvvert_{  \rho^k p , \, \rho^k q ; \, \, B_{n,k}, \zeta_n
} \le  (A_{n,k})^{1/ \rho^k} \, \bvvert \, \kappa_{n,k-1}^{2/ \rho^{k-1}} u \,\bvvert_{\rho^{k-1} p , \, \rho^{k-1} q ; \, \, B_{n,k-1}, \zeta_n}^{\gamma_{n,k}}
 \end{equation}
where
\begin{multline}
  \label{eq:itera12}
\qquad 
A_{n,k} := \cii \rho^{4k} M_k^2  \, \Vert \zeta_n \Vert_{L^1} \bigl( 1 \vee \vvert \,w^{-1} \,\vvert_{ \frac r2 ,  \, \frac s2 ; \, \, E(B_{n,k}), \zeta_n}\bigr) 
\\
\times\frac{|B_{n,k-1}|}{|B_{n,k}|} \Biggl[\biggl(\Gamma_{n,k} + \frac1{n^2}\biggr)\biggl( \frac1{\inf_{t \in \Sigma_{n,k}} \zeta_n(t)} + |B_{n,k}|^{\frac2d}\biggr) \Biggr],
\qquad
\end{multline}
with
\begin{equation}
\Gamma_{n,k} :=  
\Vert\nabla f\Vert_{\ell^\infty(E)}^2 + \Vert(\nabla f) (\nabla \eta_{n,k})\Vert_{\ell^\infty(E)} + \Vert \nabla \eta_{n,k}\Vert_{\ell^\infty(E)}^2 + \Vert \bigdot{\zeta}_n/ \zeta_n\Vert_{L^{\infty}(\mathbb{R}_+)}
\end{equation}
and~$ \Sigma_{n,k} := \text{supp}(\xi_{n,k})$. As we will now demonstrate, $A_{n,k}$ is bounded uniformly in~$n$ by a quantity whose growth in~$k$ can be controlled. 

Clearly, $ \Vert \zeta_n \Vert_{L^1} =  \Vert \zeta \Vert_{L^1}\le \ciii$, while $1\vee\vvert \,w^{-1} \,\vvert_{ \frac r2 ,  \, \frac s2 ; \, \, E(B_{n,k}), \zeta_n}\leq W$ on account of \eqref{eq:itera9} and \eqref{eq:itera2}. Similarly, $|B_{n,k-1}|/|B_{n,k}| \leq |B_{2n}|/|B_n|$ is bounded uniformly in~$n$ and~$k$. Regarding the term in the large brackets in \eqref{eq:itera12}, by the assumption on $\nabla f$ and \eqref{eq:itera3}, and since $ \Vert \bigdot{\zeta}_n/ \zeta_n\Vert_{L^{\infty}(\mathbb{R}_+)} = \Vert \bigdot{\zeta}/ \zeta\Vert_{L^{\infty}(\mathbb{R}_+)} /n^2$, we obtain, recalling also \eqref{E:itera1} and \eqref{eq:c31},
\begin{equation}
\Gamma_{n,k} \leq \frac{1}{n^2}\big(1+\Vert \bigdot{\zeta}/ \zeta\Vert_{L^{\infty}(\mathbb{R}_+)} + (\sigma_{k}-\sigma_{k+1})^{-1}  + (\sigma_{k}-\sigma_{k+1})^{-2}\big) \leq \frac{4\ciii2^{2(k+1)}}{n^2(\sigma-\sigma')^2}.
\end{equation}
Finally, \eqref{E:itera0} and \eqref{eq:itera4.3}
show $\Sigma_{n,k}  \subset [0, 2n^2]$ and so
\begin{equation}
\Bigl(\,\,\inf_{t \in \Sigma_{n,k}} \zeta_n(t)\Bigr)^{-1} \leq n^2 \Bigl(\,\,\inf_{t \in [0,2]} \zeta(t)\Bigr)^{-1} \leq \ciii n^2,
\end{equation}
whilst $|B_{n,k}|^{\frac2d}\leq |B_{2n}|^{\frac2d} \leq c n^2$. Recalling that $M_k \le c(\sigma-\sigma')^{-1}\e^{\rho k^2}$, cf. Lemma \ref{L:cutoff} and \eqref{eq:itera4.1}, and noting that there is a numerical constant~$c$ such that $2^{2k}\rho^{4k}\le c\e^{\rho k^2}$ holds for all $\rho\ge1$ and all $k\ge0$, we thus obtain
\begin{equation}
  \label{eq:itera13}
   A_{n,k} \leq \cii(\ciii)^2\civ\frac{W}{(\sigma-\sigma')^4}\e^{ 3\rho k^2},
\end{equation}
where $\civ=\civ(\alpha, \beta, q, d)  \geq 1$ collects the various numerical prefactors in the above estimates. 

Substituting \eqref{eq:itera13} into \eqref{eq:itera11} and using that~$A_{n,k}\ge1$ while $\gamma_{n,k}\le1$, the claim \eqref{eq:itera8} readily follows by induction over~$k$ (starting at $k=N+1$), noting also for the very last step that $B_{\sigma n}/B_{n,N} \leq c(d)$, which can be absorbed by adapting the constant $\civ$, and extending the arising sums over $\ell$ to start at $1$ (rather than $N+1$; the term in square brackets on the right-hand side of \eqref{eq:itera8} is greater or equal to $1$).
\end{proof}
  
Following up on Corollary~\ref{R:one-step}, one also has the following bound:
\begin{corollary}
\label{R:itera}  Under the setting and assumptions of Proposition~\ref{P:iter1}, for all $n \geq 1$, $k >N $, all $\rho \in [1, p\wedge q)$ and all weak solutions $u$ of \eqref{E:5.4}, with~$f$ on the right satisfying $\Vert\nabla f \Vert_{\ell^{\infty}(E)}\leq \frac1n$,
\begin{equation}
\label{eq:itera14}
\bvvert \,  \kappa_{n,k}^2  |u|^{\rho^k} \, \bvvert_{1 , \,  \infty\, ; \, \, B_{n,k},\,\zeta_n}^{1/\rho^k} \leq  \bigg[ \cv \frac{W \e^{ 3 \rho k^2}}{(\sigma-\sigma')^4}\bigg]^{1/\rho^k} \  \bvvert \, \kappa_{n,k-1}^{2/\rho^{k-1}} u \, \bvvert_{\rho^{k-1}p,\,\rho^{k-1}q\,;\,B_{n, k-1}, \,\zeta_n}^{\gamma_{n,k}}.
\end{equation}
\end{corollary}

\begin{proof}
We use the same setting as in the proof of Proposition~\ref{P:iter1} but invoke \eqref{E:onestep14} instead of~\eqref{E:onestep5}, and then apply \eqref{eq:itera13}.
\end{proof}

\subsection{Proof of maximal inequality}
Our next task is to ``upgrade'' the bound \eqref{eq:itera8} to an estimate on the maximum of the solution~$u$ over the space-time cylinder~$Q(n)$. First we state (in Lemma~\ref{P:itera15}) a rather immediate consequence of Proposition~\ref{P:iter1} which bounds the maximum of~$u$ in the space-time cylinder $B_{\sigma'n} \times [0, \sigma'n^2]$ in terms of the $(p',q')$-norm (for $p',q'$ as above) of~$u$ cut off outside of a slightly larger cylinder with spatial base $B_{\sigma n}$. Keeping all dependencies on~$\sigma,\sigma'$ explicit is crucial as these will be subsequently varied to replace the $(p',q')$-norm by the $(1,1)$-norm. 

\begin{lemma}
\label{P:itera15}
Suppose Assumption~\ref{ass1} and \eqref{E:2.2aa} hold. For all $d\ge2$ there is~$\cvi=\cvi(d, \rho)\in(0,\infty)$ such that for all $\alpha \in (2, 2\frac{d-1}{d-2})$, all $\beta \in (0,2)$, all $q > 1$ and $p$ as defined in \eqref{E:onestep4}, and for all integers $n \geq 1$, all $\rho \in (1, p\wedge q)$ and all weak solutions $u$ of~\eqref{E:5.4} with $f$ on the right-hand side satisfying $\Vert\nabla f \Vert_{\ell^{\infty}(E)}\leq \frac1n$ we have 
 \begin{equation}
 \label{eq:itera15}
\max_{(t,x) \in  [0, \sigma' n^2] \times B_{\sigma'n}} \bigl| \, u(t,x) \, \bigr|  
\le   \cvi\,\biggl[ \frac {\cv W}{(\sigma-\sigma')^4}\biggr]^{\frac{1}{\rho-1}} \, \, \bvvert \, \kappa_{n,0}^{2} \, u \,\bvvert_{ p' , \,  q' ; \, \, B_{\sigma n} ,\, \zeta_n}^{\bar\gamma\,' (\rho, n)},
 \end{equation}
 where $p' = \rho^Np$, $q'=\rho^Nq$, $\kappa_{n,0}$ is defined in \eqref{E:itera2.1}, $W$ is as in \eqref{eq:itera9}, $\cv $ is the constant from Proposition~\ref{P:iter1} and $\bar\gamma\,' (\rho, n) := \bar\gamma \, \big(n, (\lceil \log \log n / \log \rho \rceil)\vee (N+1)\big)$ for~$\bar\gamma \, (\cdot,\cdot)$ as defined in~ \eqref{eq:itera10} and with $N= N(\rho)$ given by \eqref{eq:iteraN}.
\end{lemma}

\begin{proof}
Let $k \geq N+2$, with $N=N(\rho)$ given by \eqref{eq:iteraN}. For any $ k $, the function $\kappa_{n,k}^{2/\rho^k}$ is equal to $1$ on $B_{n,k+1} \times [0, \tau_k n^2] \supset B_{\sigma'n} \times [0,\sigma' n^2]$, cf. \eqref{eq:itera3} and \eqref{eq:itera4.2}. Using that $\text{supp}(\zeta_n) \supset [0,\sigma' n^2]$ by \eqref{E:itera0} and \eqref{eq:kernel_conds}, and applying Corollary~\ref{R:itera} and Proposition~\ref{P:iter1} (the latter for index~$k-1 \, (>N)$), we thus get
\begin{equation}
\begin{aligned}
\label{eq:itera16}
\max_{(t,x) \in  [0,\sigma' n^2] \times B_{\sigma'n}}& \bigl| \, u(t,x) \, \bigr|  
\le \max_{(t,x) \in  [0,\sigma' n^2] \times B_{\sigma'n}} \bigl| \,  (\kappa_{n,k}^{2/\rho^k}u)(t,x) \, \bigr| 
\\
&\leq \max_{t \in [0,\sigma'n^2]} \biggl[\,\sum_{x\in B_{n,k}}\bigl|(\kappa_{n,k}^{2}\tilde u^{\rho^k})(t,x)\bigr| \biggr]^{1/{\rho^k}}
\\
&\!\!\!\!\stackrel{\eqref{eq:norm_3}}{\leq} |B_{n,k}|^{1/{\rho^k}} \bvvert \,  \kappa_{n,k}^{2}\tilde u^{\rho^k} \, \bvvert_{1,\infty; \, B_{n,k}}^{1/{\rho^k}} 
\\
&\!\!\!\!\stackrel{\eqref{eq:itera14}}{\leq} |B_{n,k}|^{1/{\rho^k}}  \biggl[ \frac{\cv W\e^{3\rho k^2}}{(\sigma-\sigma')^4}\biggr]^{1/{\rho^k}} \  \bvvert \, \kappa_{n,k-1}^{2/\rho^{k-1}} u \, \bvvert_{\rho^{k-1}p,\rho^{k-1}q ;B_{n, k-1}, \zeta_n}^{\gamma_{n,k}}\\
&\!\!\!\!\stackrel{\eqref{eq:itera8}}{\leq} |B_{2n}|^{1/{\rho^k}} \biggl[ \frac{\cv W}{(\sigma-\sigma')^4}\biggr]^{\sum_{\ell=1}^k \rho^{-\ell}} \, \e^{3\rho \sum_{\ell=1}^{k-1} {\ell^2}\rho^{-\ell}} \ \bvvert \, \kappa_{n,0}^{2} \, u \,\bvvert_{ p , \,  q ; \, \, B_{\sigma n} , \zeta_n}^{\bar\gamma (n,k)}.\end{aligned}
\end{equation}
Choosing $k := \lceil \log \log n / \log \rho \rceil \vee (N(\rho)+2)$ ensures that $|B_{2n}|^{1/{\rho^k}}\le \tilde{c}(\rho)$ uniformly in~$n$. The claim follows upon defining $ \cvi(\rho)= \tilde{c}(\rho) \exp (3\frac{\rho^2(\rho+1)}{(\rho-1)^3})$ by noting that $\sum_{\ell=1}^\infty \ell^2\rho^{1-\ell} = \rho^2(\rho+1)(\rho-1)^{-3}$ and $\sum_{\ell=1}^\infty\rho^{-\ell}=(\rho-1)^{-1}$ for all~$\rho>1$.
\end{proof}

The replacement of the $(p,q)$-norm by the $(1,1)$-norm is the subject of the following lemma,  which is more or less drawn from~\cite{ACDS16}. The proof of Proposition \ref{P:chi_Linfty} will then quickly follow, using also Lemma~\ref{L:nu_w} to bound $W$.

\begin{lemma}
\label{C:itera100}
For the setting of Lemma~\ref{P:itera15}, there are $\cvii= \cvii(\alpha, \beta, \rho, p,q,d, \ciii(\zeta))$, $\cviii=\cviii(\rho, p,q)$, and $\cix=\cix(\alpha, \beta, \rho, p,q,d, \ciii(\zeta))$,
 \begin{equation}
 \label{eq:itera17}
\max_{(t,x) \in  [0, n^2] \times B_{n}} \bigl| \, u(t,x) \, \bigr|  \le  \cvii  W^{\cviii}\  \bvvert \, 1_{[0,2n^2]} \, u \,\bvvert_{ 1 , \,  1 ; \, \, B_{2 n} , \zeta_n}^{\gamma_n(u)}.
 \end{equation}
where $1_{[0,2n^2]}$ abbreviates the indicator of $t\in[0,2n^2]$, and $\gamma_n(u)$ satisfies $1\leq \gamma_n(u) \leq \cix$ (and $\gamma_n(u)$ also implicitly depends on the same set of parameters as $\cix$). 
\end{lemma}

\begin{proof}
Define $\bar\sigma_i := 2-2^{-i}$ for $i \geq 0$, which is increasing in $i$ with $\bar\sigma_0=1$ and $\lim_{i\to\infty} \bar\sigma_i=2$. Abbreviate
\begin{equation}
\vvert f \vvert_{\infty, i}:= \max_{(t,x) \in  [0, \bar\sigma_i n^2] \times B_{\bar\sigma_i n}} \bigl| \, f(t,x) \, \bigr|.
\end{equation}
Our goal is to estimate $\vvert u\vvert_{\infty,0}$ by the right-hand side of \eqref{eq:itera17}.
We will apply \eqref{eq:itera15} repeatedly with $\sigma':= \bar\sigma_{i-1}$ and $\sigma:= \bar\sigma_i$. We will write $\kappa_{n,0}^{(i)}$ for the mollifier with these choices of~$\sigma'$ and~$\sigma$, and let $\bar\gamma^{\,\prime}_i := \bar\gamma^{\,\prime}_i(\rho, n)$ denote the quantity defined below \eqref{eq:itera15} for this pair, recalling the dependence of this quantity on $\sigma$ and $\sigma'$ via the cut-off function $\xi$ appearing in \eqref{eq:itera10}. (Since~$n$ will remain fixed, we will suppress it whenever possible.) Using \eqref{E:pq_monot} and \eqref{eq:intpol3} with $\theta := \theta(p,q,\rho)= 1- \frac{1}{p'\vee q'} \ (\in (0,1))$, where $p' =\rho^N p$ and $q'=\rho^N q$, we then have for each~$i\ge0$,
\begin{equation}
\label{eq:itera18}
\begin{aligned}
\bvvert \, (\kappa_{n,0}^{(i)})^{2} \, u \,\bvvert_{ p' , \,  q' ; \, \, B_{\bar\sigma_i n} , \zeta_n}
&\le \bvvert \, (\kappa_{n,0}^{(i)})^{2} \, u \,\bvvert_{ p'\vee q' , \,  p'\vee q' ; \, \, B_{\bar\sigma_i n} , \zeta_n}
\\
&\leq \bvvert \, (\kappa_{n,0}^{(i)})^{2} \, u \,\bvvert_{ 1 , \,  1 ; \, \, B_{\bar\sigma_i n} , \zeta_n}^{1-\theta} \bvvert \, (\kappa_{n,0}^{(i)})^{2} \, u \,\bvvert_{\infty,i}^\theta \\
&\leq c \bvvert \, 1_{[0,2n^2]} u \,\bvvert_{ 1 , \,  1 ; \, \, B_{2n} , \zeta_n}^{1-\theta}  \vvert  \, u \,\vvert_{\infty,i}^\theta
\end{aligned}
\end{equation}
for some $c=c(p,q,d)\in[1,\infty)$, where the second line follows from $\text{supp}(\kappa_{n,0}^{(i)}) \subset [0, \bar\sigma_i n^2] \times B_{\bar \sigma_i n} \subset [0, 2n^2] \times B_{2 n}$ and the fact that $|B_{2n}|/ |B_{\bar \sigma_i n}|\le |B_{2n}|/|B_n|$ is bounded uniformly in~$n$ and~$i$. Inserting \eqref{eq:itera18} into \eqref{eq:itera15} while noting that $\bar\sigma_i-\bar\sigma_{i-1} =2^{-i}$ yields, for all~$i\ge1$,
\begin{equation}
\label{eq:itera19}
\begin{split}
 \vvert  \, u \,\vvert_{\infty,i-1} 
 &= \max_{(t,x) \in  [0, \bar\sigma_{i-1} n^2] \times B_{\bar \sigma_{i-1}n}} \bigl| \, u(t,x) \, \bigr|
 \\
 & \le  c\bigl[ 2^{4i} W\bigr]^{\frac{1}{\rho-1}} \, \ \bvvert \, 1_{[0,2n^2]}  u \,\bvvert_{ 1 , \,  1 ; \, \, B_{2n} , \zeta_n}^{(1-\theta) \bar\gamma^{\,\prime}_i}\, \vvert  \, u \,\vvert_{\infty,i}^{\theta \bar\gamma^{\,\prime}_i}
 \end{split}
\end{equation}
for some $c\in[1,\infty)$ depending on the parameters $p$, $q$, and~$\rho$ but not on~$n$ or~$i$.
Iterating \eqref{eq:itera19}, we obtain, for all $m \geq 2$ and some constant $c\in[1,\infty)$ depending on the full set of parameters $\alpha, \beta, \rho, p,q,d, \ciii(\zeta)$,
\begin{equation}
\label{eq:itera20}
\begin{split}
\max_{(t,x) \in  [0,  n^2] \times B_{ n}} \bigl| \, u(t,x) \, \bigr| 
 &=  \vvert  \, u \,\vvert_{\infty,0}
 \\
 & \le  \bigl[ c W^{\frac{1}{\rho-1}} \bigr]^{1+\sum_{k=1}^m (\prod_{i=1}^k \bar\gamma^{\,\prime}_i) \theta^k} \bigl[ 2^{\frac{4}{\rho-1}} \bigr]^{1+\sum_{k=2}^m (\prod_{i=1}^{k-1} \gamma^{\,\prime}_i) k \theta^k} \\
 &\qquad \times \ \bvvert \, 1_{[0,2n^2]} u \,\bvvert_{ 1 , \,  1 ; \, \, B_{2 n} , \zeta_n}^{1+\sum_{k=1}^m (\prod_{i=1}^k \bar\gamma^{\,\prime}_i) (1-\theta)^k}  \vvert  \, u \,\vvert_{\infty,m}^{(\prod_{i=1}^m \bar\gamma^{\,\prime}_i) \theta^m}.
 \end{split}
\end{equation}
Now, since $\bar\gamma^{\,\prime}_i \leq 1$ for all $i\ge1$, see \eqref{eq:itera10} and below \eqref{eq:itera15}, and $\vvert  \, u \,\vvert_{\infty,m}$ is bounded uniformly in~$m$ (e.g., by the maximum of~$u$ over $[0, 2 n^2] \times B_{\bar 2n}$, which is finite by our assumptions on $u$) the last term on the right of \eqref{eq:itera20} tends to 1 as~$m\to\infty$. The claim \eqref{eq:itera17} follows from \eqref{eq:itera20} by letting $m \to \infty$ (the sums in the exponents all converge) and letting $\gamma_n(u) = 1+\sum_{k=2}^\infty (\prod_{i=1}^{k-1} \gamma^{\,\prime}_i) k \theta^k$.  
\end{proof}

We are now ready to prove the desired maximal inequality:

\begin{proof}[Proof of Proposition~\ref{P:chi_Linfty}]
The claim will follow by applying Lemma~\ref{C:itera100} for suitable choice of the parameters.
Fix $d\geq 2$ and $\vartheta > 4d$ as appearing in \eqref{E:q_ass} and any $r \in (2d, \frac{\vartheta}{2})$. Let $s:=r$ and let $\alpha$ and $\beta$ be defined by \eqref{E:1.4a} in terms of $r$ and $s$. Note in particular that $\beta \in (0,2)$ and $\alpha < 2\frac{d-1}{d-2}$, as follows plainly from \eqref{E:1.4a}. Moreover, since $ r> 2d$, 
$$
\frac1\alpha \stackrel{\eqref{E:1.4a}}{=} \Big( \frac12 + \frac1r -\frac1d\Big)\frac{d}{d-1}< \frac12\Big( 1-\frac1d \Big)\frac{d}{d-1} = \frac12,
$$
as required by Lemmas~\ref{P:itera15}--\ref{C:itera100}. Having selected $\alpha$ and $\beta$, the parameters $p$ and $q$ are defined by \eqref{E:onestep4} (and are both larger than $1$, as noted in Remark \ref{R:onestep001}), and we choose $\rho= \frac12(1+ (p\wedge q))$. The claim \eqref{eq:maxbd} is then an immediate consequence of \eqref{eq:itera17}. The (crucial!) fact that $W(r)< \infty$ can be arranged, cf. \eqref{eq:Wfinite}, follows from Lemma \ref{L:nu_w} by choosing $k_t:=2^{\mu} (1+t)^{-\mu}$ with any $\mu \in (4, 2\vartheta /r )$ (note that $ 2\vartheta /r > 4$ by our choice of $r$) and $\zeta(t)$ as in Lemma~\ref{lemma-2.2}, with $\nu := \mu/2$. 
\end{proof}

\section{Proof of one-step estimate}
\label{sec6a}\nopagebreak\noindent
Ouf final task is the proof of the one-step estimate in Proposition~\ref{P:one-step}. The proof hinges on three ingredients. The first one is the weighted Sobolev inequality proved in Lemma~\ref{thm-Sobolev} which bounds a suitable norm of~$f$ by the weighted Dirichlet form $\EE^w_t(f)$. The second ingredient is a comparison of the weighted Dirichlet form with its ``bare'' counterpart~$\EE^a_t(f)$. Lemma~\ref{L:MO_basic} provides such comparison when the argument is $u_t$, the solution to the Poisson equation \eqref{E:2.22}, inside a box; unfortunately, since we need to consider powers of the solution and invoke different (smoother) spatial and temporal truncations, we will have to prove the needed bound again. This is the content of (rather long) Lemma~\ref{L:DIRICH_CONVERT}. The final ingredient is a bound on the resulting ``bare'' Dirichlet energy in terms of a suitable norm of the solution. This is done in the second subsection; the proof of Proposition~\ref{P:chi_Linfty} is presented right afterwards. 

\subsection{Dirichlet energy comparison}
We begin by a comparison of the Dirichlet energies for powers of the solution of the inhomogeneous Poisson equation \eqref{E:2.22} mollified by spatial and temporal cut-off functions. While necessarily more involved, the mechanism behind the proofs is similar to that of Lemma~\ref{L:MO_basic}.

Let us introduce some more Dirichlet forms which will recurrently show up in what follows. Recall $ \EE^{\, w}_{t}(\cdot) $ and $\EE^{\, w, \zeta}(\cdot)$ from \eqref{E:2.9.0} and \eqref{E:2.9}, with weights $w$ as defined in \eqref{E:1.16}. For $f\colon E(\Z^d)\to \mathbb{R}$ and recalling our notation $x_e$ and~$y_e$ for (arbitrarily ordered) endpoints of edge~$e$, define
\begin{equation}
\label{eq:av}
\text{av}(f) (e):= \frac12\bigl(f(x_e) + f(y_e)\bigr), \quad e \in E(\Z^d).
\end{equation}
Using our earlier notation $\nabla f(e):=f(y_e)-f(x_e)$ for the gradient of~$f$, for all $g,h\colon \Z^d \to \mathbb{R}$, the discrete product rule reads
\begin{equation}
\label{eq:CR}
\nabla (gh) =  \text{av}(g) \nabla h +   \text{av}(h) \nabla g.
\end{equation}
Given~$\eta\colon \Z^d \to [0,1]$ with finite support and any $g\colon\Z^d\to\R$, let
\begin{equation}
\label{eq:Dirich9}
\EE^a_{t,\eta^2}(g):=  \sum_{e\in E(\Z^d)} \text{av}(\eta^2)(e) a_t(e) \bigl|\nabla g(e)\bigr|^2
 \end{equation}
and, similarly to \eqref{E:2.9}, for any $f\colon [0,\infty) \times \mathbb{Z}^d \to \mathbb{R}$ with compact (space-time) support, define
 \begin{equation}
 \label{eq:Dirich9.1}
 \EE^{\, a, \zeta}_{\eta^2}(f) := \int_0^\infty \textd t\,\zeta(t) \EE^a_{t,\eta^2}(f_t).
 \end{equation}
Recall the definition of the norms $\Vert \cdot \Vert_{p,q;B, \zeta}$ in \eqref{eq:norm_pp'}. We then have:

\begin{lemma}[Conversion of Dirichlet forms]
\label{L:DIRICH_CONVERT}
Suppose Assumption~\ref{ass1} and \eqref{E:2.2a} hold and let~$\ci$ be the constant from~\eqref{E:2.2a}. There is  $\cx=\cx(d)\in(0,\infty)$ such that the following holds: Let $u$ be a (weak) solution the equation \eqref{E:5.4} with $ \nabla f $ bounded uniformly on~$E:=E(\Z^d)$. Fix~$B\subset\Z^d$ finite and suppose $\eta\colon \Z^d \to [0,1]$ obeys $\text{supp}\, \eta \subset B$ and $\eta$ vanishes on the inner boundary of $B$ (i.e. the set $\{x \in \Z^d : \, x \in B, \exists y \in \Z^d \setminus B: \, x\sim y\}$). Let $\xi\colon[0,\infty)\to[0,1]$, with the value at~$t$ denoted by~$\xi_t$, be a $C^1$-function with compact support. Then for all~$\lambda\ge1$,
\begin{multline}
\quad
\label{E:Dirich10}
\EE^{w,\zeta}(\xi \eta \tilde u^\lambda) 
\le \ci\cx\lambda^2 \biggl[   \EE^{\,a,\zeta}_{\eta^2}(\xi \tilde u^\lambda) +\Vert\nabla f\Vert_{\ell^\infty(E)}^2 \bigl\Vert \xi^2 |u|^{2\lambda-2}\bigr\Vert_{1,1;B, \zeta} 
\\
+ \Vert \nabla \eta \Vert_{\ell^\infty(E)}^2 \bigl\Vert \xi^2 |u|^{2\lambda}\bigr\Vert_{1,1;B, \zeta} +  \bigl\Vert (\bigdot{\xi})^2 |u|^{2\lambda}\bigr\Vert_{1,1;B, \zeta} \biggr],
\quad
\end{multline}
where $\tilde u^{\lambda}:= \text{\rm sign}(u)|u|^{\lambda}$ and where $\bigdot{\xi}$ denotes the derivative of $\xi$.
\end{lemma}

\begin{remark}\label{R:dirich1}
The precise form of \eqref{E:Dirich10} is tailored to our future purposes, in the sense that the Dirichlet form  $\EE^{\,a,\zeta}(\xi \tilde u^\lambda)$ naturally comes out of a later energy estimate, see Lemma~\ref{lemma-energy} below. It is important that these two quantities be matched.
\end{remark}

In the proof we will need:

\begin{lemma}
\label{lemma_tilde}
For all $a,b \in \mathbb{R}$ and all $\lambda \geq 1$, with $\tilde a^{\lambda} := \text{\rm sign}(a)|a|^{\lambda}$, we have
\begin{equation}
\label{E:Dirich11}
\bigl(|a|^{2\lambda-2}+|b|^{2\lambda-2}\bigr)(b-a)^2\le 8(\tilde b^\lambda- \tilde a^\lambda)^2.
\end{equation}
\end{lemma}
\begin{proofsect}{Proof} Suppose first that~$a$ and~$b$ have the same sign. In this case, $(b-a)^2= (|b|-|a|)^2$ as well as $(\tilde b^\lambda-\tilde a^\lambda)^2 = (|b|^\lambda-|a|^\lambda)^2$ and so \eqref{E:Dirich11} can be recast as
\begin{equation}
\bigl(|a|^{2\lambda-2}+|b|^{2\lambda-2}\bigr)(|b|-|a|)^2\le 8(|b|^\lambda-  |a|^\lambda)^2.
\end{equation}
This is proved by setting $x:=|a|/|b|$ (assuming $|a|\le |b|$) and noting that $1-x^\lambda\ge 1-x$ for $x\in[0,1]$ and $\lambda\ge1$ (in fact the inequality even holds with~$2$ instead of~$8$ on the right-hand side).

Suppose now that $a$ and $b$ have opposite signs. By symmetry, it is enough to consider the case $a\geq 0$, $b < 0$, in which, using that $(a+ |b|)^2 \leq 2a^2 + 2|b|^2$,
\begin{equation}
\begin{split}
\bigl(|a|^{2\lambda-2}+|b|^{2\lambda-2}\bigr)(b-a)^2 &= \bigl(a^{2\lambda-2}+|b|^{2\lambda-2}\bigr)(a+ |b|)^2 \leq 8 (a \vee |b|)^{2\lambda}\\
& \leq 8 \bigl(a^{2\lambda} + |b|^{2\lambda}\bigr)\leq 8\bigl(a^{\lambda}+  |b|^{\lambda}\bigr)^2 = 8\bigl(\tilde a^{\lambda} - \tilde b^{\lambda}\bigr)^2.
\end{split}
\end{equation}
The claim follows.
\end{proofsect}

\begin{proofsect}{Proof of Lemma~\ref{L:DIRICH_CONVERT}}
We build on the argument from the proof of Lemma~\ref{L:MO_basic} which we hereby invite the reader to inspect first. To start, using the discrete product rule \eqref{eq:CR}, the definition of the weights $w$ in \eqref{E:1.16} along with the inequality $(a+b)^2\le 2a^2+2b^2$ and the bound $\text{av}(\eta)^2 \leq  \text{av}(\eta^2)$, and minding that $\xi$ is a function of $t$ alone, we obtain for all $t \geq 0$ that
\begin{equation}
\label{E:Dirich12}
\begin{aligned}
\EE^w_t(\xi_t \eta &\tilde{u}_t^\lambda) 
 = \xi_t^2 \int_t^\infty \textd s\,\,k_{s-t}\sum_e a_s(e) \bigl(\nabla (\eta \tilde u_t^\lambda)(e)\bigr)^2 
\\
&\leq  2   \int_t^\infty \textd s\,\, k_{s-t}\biggl[\,\sum_e a_s(e) \bigl(  \xi_t \, \text{av}(\eta)  \nabla \tilde u_t^\lambda \bigr)^2(e)  + \xi_t^2 \sum_e a_s(e) \bigr(\text{av}(\tilde u_t^\lambda) \nabla \eta \bigr)^2(e)\biggr] \\
&\leq 2 \EE^w_{t,\eta^2}(\xi_t \tilde u_t^\lambda) + 2d\Vert k\Vert_{L^1}\Vert \nabla \eta \Vert_{\ell^\infty(E)}^2 \xi_t^2 \bigl\Vert |u_t|^{2\lambda} \bigr\Vert_{\ell^1(B)}\,,
\end{aligned}
\end{equation}
where $\EE^w_{t,\eta^2}(\cdot)$ is the quantity from \eqref{eq:Dirich9} with~$w_t$ in place of~$a_t$ and where $\Vert\cdot\Vert_{L^1}$ abbreviates the $L^1$-norm on~$[0,\infty)$. Note that \eqref{E:2.2a} implies $\Vert k\Vert_{L^1}\le \ci$.
In view of \eqref{E:2.9} and \eqref{eq:Dirich9.1}, multiplying by $\zeta(t)$ on both sides of \eqref{E:Dirich12} and integrating over $t$, it thus suffices to show a bound of the form \eqref{E:Dirich10} for $\EE^{w, \zeta}_{\eta^2}( \xi \tilde u^\lambda) $ in place of $\EE^{w, \zeta}(\xi \eta \tilde u^\lambda)$. 

Using that $\text{av}(\eta^2) \leq 2 \text{av}(\eta)^2$, we write, for all $t \geq 0$,
\begin{equation}
\begin{aligned}
\label{E:Dirich13}
\EE^w_{t,\eta^2}( \xi_t \tilde u_t^\lambda) 
 &= \int_t^\infty \textd s\,\,k_{s-t}\sum_e \bigl(a_s\text{av}(\eta^2)\bigr) (e)  \bigl(\xi_t \nabla \tilde u_t^\lambda(e)\bigr)^2
\\
&\le
2\int_t^\infty \textd s\,\,k_{s-t}\sum_e\bigl(a_s\text{av}(\eta^2)\bigr)(e)\bigl( \xi_s \nabla  \tilde u_s^\lambda(e)\bigr)^2 
\qquad\qquad
\\
&\qquad\qquad+ 4\int_t^\infty \textd s\,\,k_{s-t}
\sum_e a_s(e)\Bigl(\xi_t\text{av}(\eta)  \nabla \tilde u_t^\lambda-\xi_s \text{av}(\eta)  \nabla \tilde u_s^\lambda\Bigr)(e)^2.
\end{aligned}
\end{equation}
Multiplying the first integral on the right by $\zeta(t)$ and integrating over~$t$, we get
\begin{multline}
\qquad
\int_0^\infty\textd t\,\zeta(t)\Bigl(\int_t^\infty \textd s\,\,k_{s-t}\sum_e \bigl(a_s\text{av}(\eta^2)\bigr)(e)\bigl(\xi_s \nabla \tilde u_s^\lambda(e)\bigr)^2 
\Bigr)
\\=\int_0^\infty \textd s\,\,\EE^a_{s, \eta^2}(\xi_s \tilde u_s^\lambda)\Bigl(\int_0^s\textd t\,\zeta(t)k_{s-t}\Bigr),
\qquad\end{multline}
which in light of the definition of~$K_t$ in \eqref{E:Kt} and \eqref{E:2.2a} is at most $\ci \EE^{\, a , \zeta}_{ \eta^2}( \xi \tilde u^\lambda)$. Hence, this term contributes directly to the first term on the right-hand side of \eqref{E:Dirich10}. Concerning the second integral on the right of \eqref{E:Dirich13}, the discrete product rule \eqref{eq:CR} implies
\begin{equation}
\begin{split}
&\xi_t\text{av}(\eta)  \nabla \tilde u_t^\lambda - \xi_s  \text{av}(\eta) \nabla \tilde u_s^\lambda\\
&\qquad = 
- \xi_t\text{av}( \tilde u_t^\lambda)  \nabla \eta
+\xi_s  \text{av}( \tilde u_s^\lambda) \nabla \eta
+\big(\xi_t  \nabla  (\eta \tilde u_t^\lambda) -\xi_s  \nabla (\eta \tilde u_s^\lambda)\big).\end{split}
\end{equation}
In conjunction with the inequality $(a+b+c)^2\le3a^2+3b^2+3c^2$, this yields
\begin{equation}
\label{E:Dirich14}
\begin{aligned}
&\int_t^\infty \textd s\,\,k_{s-t}
\sum_e a_s(e)\Bigl(\xi_t\text{av}(\eta)  \nabla \tilde u_t^\lambda-\xi_s \text{av}(\eta)  \nabla \tilde u_s^\lambda\Bigr)(e)^2 \leq 3\bigl[I^{(1)}_t + I^{(2)}_t + I^{(3)}_t\bigr], 
\end{aligned}
\end{equation}
where
\begin{equation}
\begin{aligned}
I^{(1)}_t &:= \int_t^\infty \textd s\,\,k_{s-t} \sum_e a_s(e) \xi_t^2 (\text{av}(\tilde u_t^\lambda) \nabla \eta ) (e)^2,\\
I^{(2)}_t &:= \int_t^\infty \textd s\,\,k_{s-t} \sum_e a_s(e) \xi_s^2 (\text{av}(\tilde u_s^\lambda) \nabla \eta ) (e)^2,\\
I^{(3)}_t &:=  \int_t^\infty \textd s\,\,k_{s-t}
\sum_e a_s(e)\Bigl( \xi_s \nabla (\eta\tilde u_s^\lambda)- \xi_t  \nabla (\eta\tilde u_t^\lambda)\Bigr)(e)^2 .
 \end{aligned}
\end{equation}
We will now show separately that, upon multiplication with $\zeta(t)$ and integration over~$t$, each of the three terms  $I^{(1)}_t$, $I^{(2)}_t$, $I^{(3)}_t$ in \eqref{E:Dirich14} is bounded by the right-hand side of \eqref{E:Dirich10}. 

Using that $a_s \leq 1$, we immediately get $I^{(1)}_t\leq2d \Vert k\Vert_{L^1}\Vert \nabla \eta \Vert_{\ell^\infty(E)}^2  \xi_t^2 \Vert |u_t|^{2\lambda}\Vert_{\ell^1(B)}$ for all $t \geq 0$. Since $\Vert k\Vert_{L^1}\le \ci$, this shows
\begin{equation}
\int_0^\infty\textd t\,\zeta(t)I^{(1)}_t
\le 2d\ci\Vert \nabla \eta \Vert_{\ell^\infty(E)}^2 \bigl\Vert \xi^2 |u|^{2\lambda}\bigr\Vert_{1,1;B, \zeta}\,.
\end{equation}
Some more care is needed to bound $\int_0^\infty \textd t\,\zeta(t) I^{(2)}_t$. Exchanging the order of integration  and using \eqref{E:2.2a} along with~$a_s(e)\le1$ again, we obtain
\begin{equation}
\begin{split}
\int_0^\infty \textd t\,\, \zeta(t) I^{(2)}_t&= \int_0^\infty \textd s\,\Bigl( \int_0^s \textd t\,\,\zeta(t) k_{s-t} \Bigr) \sum_e a_s(e) \xi_s^2 (\text{av}(\tilde u_s^\lambda) \nabla \eta )^2 (e) \\
& \leq 2d \ci  \Vert \nabla \eta \Vert_{\ell^\infty(E)}^2 \bigl\Vert \xi^2 |u|^{2\lambda} \bigr\Vert_{1,1;B, \zeta}.
\end{split}
\end{equation}
It remains to derive a suitable bound on $I^{(3)}_t$ which is considerably more involved. First, the assumption $a_s(e)\le1$ and elementary symmetrization arguments yield
\begin{equation}
\label{E:Dirich15}
I^{(3)}_t  \leq 4d \int_t^\infty \textd s\,\,k_{s-t}
\sum_x\big(\xi_s \tilde u_s^\lambda(x)- \xi_t \tilde u_t^\lambda(x)\bigr)^2\eta(x)^2
\end{equation}
with the summation effectively only over a finite set since~$\eta$ has finite support. We now use that~$u$ solves \eqref{E:5.4} along with the fact that $\partial_t \tilde u_t^{\lambda} = \lambda |u_t|^{\lambda-1} \partial_t u_t$ to get
\begin{equation}
\begin{aligned}
\label{E:Dirich16}
\xi_s \tilde u_s^\lambda(x)- \xi_t\tilde u_t^\lambda(x) 
&= \int_t^s  \textd r\, \frac{\textd }{\textd r} (\xi_r \tilde u_r^\lambda(x))
\\
&= \int_t^s  \textd r\, \bigdot{\xi}_r \tilde u_r^\lambda(x)  +\int_t^s \textd r\,\, \xi_r\lambda |u_r(x)|^{\lambda-1}(L_r f)(x) 
\\
&\qquad\qquad\qquad-\int_t^s \textd r\,\, \xi_r \lambda  |u_r(x)|^{\lambda-1}(L_r u_r)(x).
\end{aligned}
\end{equation}
Substituting \eqref{E:Dirich16} into \eqref{E:Dirich15}, using the Cauchy-Schwarz inequality and the standard inequality $(a+b+c)^2\le3a^2+3b^2+3c^2$, we thus get
\begin{equation}
\label{E:Dirich17}
I^{(3)}_t 
\le 12d\bigl[\,\hat A_t+\hat B_t + \hat C_t\bigr],
\end{equation}
where
\begin{equation}
\begin{split}
\label{E:Dirich18}
&\hat A_t:= \int_t^\infty \textd s\,\,k_{s-t}(s-t)\sum_x \Bigl(\int_t^s \textd r\,\,\eta(x)^2 (\bigdot{\xi}_r)^2 |u_r(x)|^{2\lambda}\Bigr),\\
&\hat B_t:= \int_t^\infty \textd s\,\,k_{s-t}(s-t)\sum_x \Bigl(\int_t^s\textd r\,\,\lambda^2 \eta(x)^2 \xi_r^2 |u_r(x)|^{2\lambda-2}(L_r f)(x)^2\Bigr) \\
&\hat C_t:= \int_t^\infty \textd s\,\,k_{s-t}(s-t)\sum_x \Bigl(\int_t^s\textd r\,\,\lambda^2 \eta(x)^2\xi_r^2 |u_r(x)|^{2\lambda-2}(L_r u_r)(x)^2\Bigr).
\end{split}
\end{equation}
The following consequence of our basic assumptions on~$\zeta$ and~$k$ will be useful for bounding all three quantities in \eqref{E:Dirich18}: For any measurable $g\colon\R\to[0,\infty)$, the definition of $K_t$ in \eqref{E:Kt} and condition \eqref{E:2.2a} imply
\begin{equation}
\label{E:Dirich19}
\begin{aligned}
\int_0^\infty\textd t\,\,\zeta(t)\biggl(\int_t^\infty &\textd s\,\,k_{s-t}(s-t)\Bigl(\,\int_t^s \textd r\,\,g(r)\Bigr)\biggr)
\\
&=\int_0^\infty \textd r\,\,g(r)\biggl(\int_0^r\textd t\,\,\zeta(t)\Bigl(\int_{r-t}^\infty \textd u\,\,k_{u}\,u\,\Bigr)\biggr)
\\
&\le \int_0^\infty \textd r\,\,g(r)\Bigl(\,\int_0^r\textd t\,\,\zeta(t) K_{r-t}\Bigr)
\le \ci\int_0^\infty \textd r\,\,g(r)\zeta(r).
\end{aligned}
\end{equation}
Indeed, applying this with $g(r):= (\bigdot{\xi}_r)^2 |u_r(x)|^{2\lambda}$ (which is indeed non-negative) yields
\begin{equation}
\label{E:Dirich20}
\int_0^\infty\zeta(t)\,  \hat A_t \, \textd t\leq \, \ci \bigl\Vert (\bigdot{\xi})^2 |u|^{2\lambda} \bigr\Vert_{1,1;B, \zeta},
\end{equation}
which in light of~$\lambda\ge1$ is bounded by a corresponding term on the right-hand side of \eqref{E:Dirich10}. For the term $\hat B_t$ we use $a_r(e)\le1$ to bound $(L_rf)^2\le 4d^2\Vert\nabla f\Vert_{\ell^\infty(E)}^2$. Then \eqref{E:Dirich19} shows
\begin{equation}
\label{E:Dirich22}
\int_0^\infty\textd t\,\,\zeta(t)\,  \hat B_t \leq 4d^2 \ci \lambda^2 \, \Vert\nabla f\Vert_\infty^2  \,\bigl\Vert {\xi}^2 |u|^{2\lambda-2} \bigr\Vert_{1,1;B, \zeta}\,.
\end{equation}
In order to bound $\hat C_t$, we first use the Cauchy-Schwarz inequality, $a_t \leq 1$, $\eta(x)^2 \leq \text{av}(\eta^2)(e)$ and Lemma~\ref{lemma_tilde} to get
\begin{equation}
\label{E:Dirich21}
\begin{aligned}
\sum_x \eta(x)^2 &|u_r(x)|^{2\lambda-2}(L_r u_r)(x)^2\\
&=\sum_x|u_r(x)|^{2\lambda-2} \eta(x)^2\Bigl( \sum_{e=(x,y)} a_r(e)\bigl[u_r(y)-u_r(x)\bigr]\Bigr)^2
\\
&\leq 2d \sum_x|u_r(x)|^{2\lambda-2} \sum_{e=(x,y)} \bigl(\text{av}(\eta^2)a_r\bigr)(e)\bigl[u_r(y)-u_r(x)\bigr]^2
\\
&\le
{2d}\sum_{e=(x,y)}\bigl(\text{av}(\eta^2)a_r\bigr)(e)\bigl[|u_r(y)|^{2\lambda-2}+|u_r(x)|^{2\lambda-2}\bigr] (\nabla u_r)^2(e).
\\
&\le 
 16d \sum_{e}\bigl(\text{av}(\eta^2)a_r\bigr)(e)(\nabla \tilde u_r^{\lambda})(e)^2= 16d \EE^a_{t,\eta^2}(\tilde u_r^{\lambda}).
\end{aligned}
\end{equation}
Plugging this in \eqref{E:Dirich18} and invoking \eqref{E:Dirich19} then yields
\begin{equation}
\begin{aligned}
\label{E:Dirich23}
\int_0^\infty\textd t\,\,\zeta(t)\,  \hat C_t \,  
&\le 16d\lambda^2 \int_0^\infty\textd t\,\,\zeta(t) \biggl(\int_t^\infty \textd s\,\,k_{s-t}(s-t)\Bigl(\int_t^s\textd r\,\,\EE^a_{r,\eta^2}(\xi_r \tilde u_r^\lambda)\Bigr)\biggr) 
\\
&\leq 16d\ci\lambda^2\int_0^\infty\textd r\,\,\EE^a_{r,\eta^2}(\xi_r \tilde u_r^\lambda)\zeta(r)
=16d \ci\lambda^2\EE^{\,a, \zeta}_{\eta^2}(\xi \tilde u^\lambda).
\end{aligned}
\end{equation}
It follows from \eqref{E:Dirich17}, \eqref{E:Dirich20},  \eqref{E:Dirich22} and~\eqref{E:Dirich23} that $\int_0^\infty\zeta(t) I_t^{(3)} \textd t$ admits the desired bound. The proof of \eqref{E:Dirich10} is complete.
\end{proofsect}

\subsection{Energy estimate}
Our next step is the so-called energy estimate which bounds the Dirichlet energy of powers of solution to the inhomogeneous Poisson equation (under truncation with respect to space and time) by a suitable norm thereof. The same calculation also produces a pointwise estimate (in time) of the $\ell^1$-norm of the (power of) solution weighted by~$\zeta$. The precise statement is as follows:

\begin{lemma}[Energy estimate]
\label{lemma-energy} Suppose Assumption~\ref{ass1} and \eqref{E:2.2a} hold. There is a numerical constant $\cxi\in(0,\infty)$ such that for all~$\lambda\ge1$ and for any solution~$u$ of \eqref{E:5.4}, we have 
\begin{equation}
\label{E:energy1}
\begin{split}
&\max \Big\{ \sup_{t\geq 0} \Big[ \zeta(t)\Vert ( \xi_t \eta \tilde{u}_t^{\lambda})^2\Vert_{\ell^1(B)} \Big], \, \, \EE^{\,a, \zeta}_{\eta^2}(\xi \tilde u^\lambda) \Big\}\\
&\qquad\begin{array}{rl}
\leq \cxi \lambda^2 &\hspace{-1ex}\Bigg[\Vert\nabla f\Vert_{\ell^\infty(E)}^2 \bigl\Vert \xi^2 |u|^{2\lambda-2}\bigr\Vert_{1,1;B, \zeta} + \bigl\Vert(\nabla f)( \nabla \eta) \bigr\Vert_{\ell^\infty(E)} \bigl\Vert \xi^2 |u|^{2\lambda-1}\bigr\Vert_{1,1;B, \zeta}\\[1em]
&\hspace{-1ex}+ \big( \Vert \nabla \eta \Vert_{\ell^\infty(E)}^2  + \Vert \bigdot{\zeta}/ \zeta\Vert_{L^{\infty}(\mathbb{R}_+)} \big) \bigl\Vert \xi^2 |u|^{2\lambda}\bigr\Vert_{1,1;B, \zeta} + \Bigl\Vert \frac{\textd \xi^2}{\textd t} \, |u|^{2\lambda}\Bigr\Vert_{1,1;B, \zeta} \Bigg].
\end{array}
\end{split}
\end{equation}
\end{lemma}

\begin{proof}
Repeating the argument leading to (5.18) of \cite{ACDS16} and using that $a_t \leq 1$ yields an absolute  constant~$c\in(0,\infty)$ such that the bound
\begin{multline}
\qquad
\label{E:energy2}
-\partial_t\bigl\Vert (\eta \tilde u_t^{\lambda})^2\bigr\Vert_{\ell^1(B)} + \EE^{\,a}_{t,\eta^2}( \tilde u_t^\lambda)  
\\
\leq c \lambda^2 \Bigl[ 
\Vert \nabla \eta \Vert_{\ell^\infty(E)}^2  \bigl\Vert  |u_t|^{2\lambda}\bigr\Vert_{\ell^1(B)} +  \Vert\nabla f\Vert_{\ell^\infty(E)}^2 \bigl\Vert  |u_t|^{2\lambda-2}\bigr\Vert_{\ell^1(B)} 
\\
+\bigl\Vert(\nabla f)( \nabla \eta) \bigr\Vert_{\ell^\infty(E)}  \bigl\Vert  |u_t|^{2\lambda-1}\bigr\Vert_{\ell^1(B)} \Bigr]
\qquad
\end{multline}
holds for all $t \geq 0$ and all $\lambda \geq 1$.
Next we observe that, for all $s \geq 0$,
\begin{equation}
\label{E:energy3}
\begin{aligned}
\int_s^{\infty} &\textd t \, \zeta(t) \xi_t^2 \big(-\partial_t \Vert (\eta \tilde u_t^{\lambda})^2\Vert_{\ell^1(B)} \big)\\
&\qquad = \zeta(s)  \xi_s^2 \bigl\Vert (\eta \tilde u_s^{\lambda})^2\bigr\Vert_{\ell^1(B)}+ \int_s^{\infty} \textd t \,  \bigl\Vert (\eta \tilde u_t^{\lambda})^2\bigr\Vert_{\ell^1(B)}\,\partial_t \bigl(\zeta(t) \xi_t^2 \bigr)\\
& \qquad \geq  \zeta(s)  \xi_s^2 \bigl\Vert (\eta \tilde u_s^{\lambda})^2\bigr\Vert_{\ell^1(B)}-  \Vert \bigdot{\zeta}/ \zeta\Vert_{L^{\infty}(\mathbb{R}_+)}\bigl \Vert \xi^2 |u|^{2\lambda} \bigr\Vert_{1,1; B, \zeta} - \Bigl\Vert \frac{\textd \xi^2}{\textd t} \, |u|^{2\lambda} \Bigr\Vert_{1,1; B, \zeta}\,.
\end{aligned}
\end{equation}
Multiplying both sides of \eqref{E:energy2} by $\zeta(t)\xi_t^2$, integrating over $t$ from $0$ to infinity, invoking \eqref{E:energy3} with $s=0$ and foregoing the term $\zeta(0)  \xi_0^2 \Vert (\eta \tilde u_0^{\lambda})^2\Vert_{\ell^1(B)}$, we find that $\EE^{\,a, \zeta}_{\eta^2}(\xi \tilde u^\lambda)$ is bounded by the right-hand side of \eqref{E:energy1}. Repeating the argument, but this time neglecting the term $\EE^{\,a}_{t,\eta^2}( \tilde u_t^\lambda)$ in \eqref{E:energy2}, and integrating from $s$ to infinity, we infer that $\zeta(s) \Vert ( \xi_s \eta \tilde u_s^{\lambda})^2\Vert_{\ell^1(B)}$ admits the same bound, for all $s \geq 0$. Hereby \eqref{E:energy1} follows.
\end{proof}

\subsection{Proof of one-step estimate}
The proof of Proposition~\ref{P:one-step}, which we are about to give, combines the Sobolev inequality of Corollary~\ref{C:Sobolev-tailored'} with Lemmas~\ref{L:DIRICH_CONVERT} and~\ref{lemma-energy}. The conversion lemma (Lemma~\ref{L:DIRICH_CONVERT}) will play a pivotal role in recovering the Dirichlet form that the energy estimate gives us information about; namely, the one naturally associated to the Poisson equation \eqref{E:5.4}, cf. Remark~\ref{R:dirich1}(1).

\begin{proof}[Proof of Proposition~\ref{P:one-step}]
Abbreviate $\lambda:=\lambda_1/2$ and note that $\lambda \geq 1$, as will be desired for applications of the previous two lemmas. In view of \eqref{E:onestep4} and the interpolation inequality \eqref{eq:intpol3}, we have
\begin{equation}
\begin{split}
\label{E:onestep7}
\vvert \,  \kappa_1^{2/\lambda_1} u \, \vvert_{ \lambda_1 p , \, \lambda_1 q ; \, \, B_1, \zeta}
&= \vvert \,  \kappa_1^{1/\lambda}  u \, \vvert_{ 2 \lambda p , \, 2\lambda q ; \, \, B_1, \zeta} 
= \bvvert \, ( \kappa_1  \tilde u^{\lambda})^2 \, \bvvert_{ p , \,  q ; \, \, B_1, \zeta
}^{\frac{1}{2\lambda}}\\
& \leq  \bvvert \, ( \kappa_1  \tilde u^{\lambda})^2 \, \bvvert_{ \hat p(\alpha) , \,  \hat q(\beta) ; \, \, B_1, \zeta}^{\frac{\theta}{2\lambda}} \,\,\bvvert \, ( \kappa_1  \tilde u^{\lambda})^2 \, \bvvert_{1 , \,  \infty ; \, \, B_1,\zeta }^{\frac{1-\theta}{2\lambda}}. 
\end{split}
\end{equation}
We will now estimate each of the arising norms individually. 

We begin with the second norm on the right of \eqref{E:onestep7} as its control is easier. The energy estimate \eqref{E:energy1} from Lemma~\ref{lemma-energy} along with $\supp(\zeta)\subseteq[0,\infty)$ readily yield
\begin{equation}
\begin{aligned}
\label{E:onestep8}
\bvvert \, ( \kappa_1  \tilde u^{\lambda})^2& \, \bvvert_{1 , \,  \infty ; \, \, B_1,\zeta} 
\le\, \sup_{t \geq 0} |B_1|^{-1}\bigl\Vert ( \xi_1(t) \eta_1 \tilde{u}_t^{\lambda})^2\bigr\Vert_{\ell^1(B_1)}
\\
&\le \Bigl[\,\inf_{t \in \text{supp}(\xi_1)} \zeta(t)\Bigr]^{-1}
|B_1|^{-1}\sup_{t \geq 0} \Bigl\{\zeta(t)\bigl\Vert ( \xi_1(t) \eta_1 \tilde{u}_t^{\lambda})^2\bigr\Vert_{\ell^1(B_1)}\Bigr\}
\\
&\leq \cxi \lambda^2 \Bigl[\,\inf_{t \in \text{supp}(\xi_1)} \zeta(t)\Bigr]^{-1} |B_1|^{-1} 
 \\
&\qquad\times\bigg[\, \Gamma \, \Big\Vert \xi_1^2 \big( |u|^{2\lambda-2} + |u|^{2\lambda-1} + |u|^{2\lambda}\big) \Big\Vert_{1,1;B_1, \zeta} + \ \Big\Vert \bigl| {\textstyle\frac{\textd \xi_1^2}{\textd t}} \bigr| \, |u|^{2\lambda}\Big\Vert_{1,1;B_1, \zeta} \bigg]\,,
\end{aligned}
\end{equation}
where $\Gamma$ is as defined in \eqref{E:onestep6}. Since the weights $\kappa_1,\kappa_2$ were assumed to be $(B_1,B_2)$-adapted with parameters $(\delta,\rho,M)$, \eqref{E:onestep3} shows that, for all $t \geq 0$,
\begin{equation}
\label{E:onestep9}
\Big| \frac{\textd \xi_1^2(t)}{\textd t}\Big| = 2 \xi_1(t) |\bigdot{\xi_1}(t)| \le 2 \delta M^2 \xi_2(t)^{2\rho}.
\end{equation}
With the help of Jensen's inequality we in turn get that, for $k =0,1,2$,
\begin{equation}
\begin{aligned}
\label{E:onestep10}
|B_1|^{-1} \bigl\Vert \xi_1^2 |u|^{2\lambda-k} \bigr\Vert_{1,1;B_1, \zeta} 
&= \Vert \zeta \Vert_{L^1} \bvvert \xi_1^2 |u|^{2\lambda-k} \bvvert_{1,1;B_1, \zeta} 
\\
&\leq  \Vert \zeta \Vert_{L^1} \bvvert \xi_1^{2\frac{2\lambda}{2\lambda-k}} |u|^{2\lambda} \bvvert_{1,1;B_1, \zeta}^{1-\frac{k}{2\lambda}} 
\leq \Vert \zeta \Vert_{L^1} M^2 \bvvert \xi_2^{2\rho} |u|^{2\lambda} \bvvert_{1,1;B_1, \zeta}^{1-\frac{k}{2\lambda}},
\end{aligned}
\end{equation}
where in the last step we used that $\xi_1^{2\frac{2\lambda}{2\lambda-k}} \leq \xi_1^2 \leq M^2 \xi_2^{2\rho}$ thanks to  $\xi_1\in[0,1]$ and~\eqref{E:onestep3}. Substituting \eqref{E:onestep9} and \eqref{E:onestep10} into \eqref{E:onestep8}, we find
\begin{equation}
\label{E:onestep11}
 \bvvert \, ( \kappa_1  \tilde u^{\lambda})^2 \, \bvvert_{1 , \,  \infty ; \, \, B_1,\zeta} \leq 
 3\cxi (\lambda M)^2  \Vert \zeta \Vert_{L^1} \frac{\Gamma + \delta}{  \inf_{\Sigma_1}\zeta} \ \bvvert \, \xi_2^{2\rho} |u|^{2\lambda} \, \bvvert_{1,1;B_1, \zeta}^{\gamma},
\end{equation}
where we also invoked the definition of $\gamma$ from \eqref{E:onestep5.1}.

We now turn to the first norm in the second line of \eqref{E:onestep7}. Using the Sobolev inequality from Corollary~\ref{C:Sobolev-tailored'}, whose conditions are met for the allowed range of $\alpha$ and $\beta$, cf. above \eqref{E:onestep4}, and subsequently applying the energy-conversion Lemma~\ref{L:DIRICH_CONVERT} yields
\begin{equation}
\label{E:onestep12}
\begin{aligned}
\bvvert \,& ( \kappa_1  \tilde u^{\lambda})^2 \, \bvvert_{ \hat p(\alpha) , \,  \hat q(\beta) ; \, \, B_1, \zeta}
\stackrel{\eqref{E:1.5abcd'}}{\leq} 2d \czero^2 |B_1|^{\frac2d} \vvert w^{-1}\vvert_{\frac r2,\frac s2,E(B_1),\zeta} \,\,\frac{\EE^{w,\zeta}(\xi_1 \eta_1  \tilde u^{\lambda})}{|B_1|}\\
&\qquad\stackrel{\eqref{E:Dirich10}}{\leq} 2d \czero^2\ci\cx \lambda^2 |B_1|^{\frac2d} \vvert w^{-1}\vvert_{\frac r2,\frac s2,E(B_1),\zeta} \Biggl[\frac{\EE^{a,\zeta}_{\eta_1^2}(\xi_1  \tilde u^{\lambda})}{|B_1|}  
\\
&\qquad\qquad\qquad+\Vert \zeta \Vert_{L^1}\Big( \Gamma \ \bvvert \, \xi_1^{2}\big( |u|^{2\lambda} + |u|^{2\lambda-2}\big) \bvvert_{1,1;B_1, \zeta} +  \bvvert (\bigdot{\xi}_1)^2 |u|^{2\lambda}\bvvert_{1,1;B_1, \zeta}\Big)\Biggr].
\end{aligned}
\end{equation}
The ``bare'' Dirichlet energy on the right is now bounded using Lemma~\ref{lemma-energy} exactly as above with the result
\begin{equation}
\frac{\EE^{a,\zeta}_{\eta_1^2}(\xi_1  \tilde u^{\lambda})}{|B_1|} \leq 3\cxi (\lambda M)^2 (\Gamma + \delta) \Vert \zeta \Vert_{L^1} \bvvert \, \xi_2^{2\rho} |u|^{2\lambda} \, \bvvert_{1,1;B_1, \zeta}^{\gamma}\,.
\end{equation}
The remaining terms are estimated directly with the help of \eqref{E:onestep10} and the bounds on the mollifiers in \eqref{E:onestep3}. This yields
\begin{multline}
\label{E:onestep12a}
\qquad
\bvvert \, ( \kappa_1  \tilde u^{\lambda})^2 \, \bvvert_{ \hat p(\alpha) , \,  \hat q(\beta) ; \, \, B_1, \zeta}
\\
\le 10d \czero^2\ci\cx \cxi (\lambda^2 M)^2 \vvert w^{-1}\vvert_{\frac r2,\frac s2,E(B_1),\zeta}\,\, \Vert \zeta \Vert_{L^1} \Big[ |B_1|^{\frac2d} (\Gamma + \delta) \Big] \  \bvvert \, \xi_2^{2\rho} |u|^{2\lambda} \, \bvvert_{1,1;B_1, \zeta}^{\gamma}
\qquad
\end{multline}
In order to covert the last norm to the desired form, we observe that, since $\eta_2 =1$ on $B_1$ by assumption, cf. \eqref{E:onestep2}, and minding that $p/\rho > 1$ and $q/\rho >1$, we have
\begin{multline}
\label{E:onestep13}
\qquad
\frac{|B_1|}{|B_2|}\, \bvvert \, \xi_2^{2\rho} |u|^{2\lambda} \, \bvvert_{1,1;B_1, \zeta} \stackrel{\eqref{E:onestep1}}{\leq}   \bvvert \, \kappa_2^{2\rho} |u|^{2\lambda} \, \bvvert_{1,1;B_2, \zeta} \\
\stackrel{\eqref{E:pq_monot}}{\leq}   \bvvert \, \kappa_2^{2\rho} |u|^{2\lambda} \, \bvvert_{\frac p\rho,\frac q\rho;B_2, \zeta} =  \vvert \, \kappa_2^{2/\lambda_2} u \, \vvert_{\lambda_2p,\lambda_2q ;B_2, \zeta}^{2\lambda},
\qquad
\end{multline}
where we also recalled that $\lambda_2 := \lambda_1/\rho = 2\lambda/\rho$.
Substituting this into \eqref{E:onestep11} and \eqref{E:onestep12a}, and then these back into \eqref{E:onestep7}, the claim follows by noting that~$\gamma\le1$.
\end{proof}

We also need to finish:

\begin{proof}[Proof of Corollary~\ref{R:one-step}]
This is due to \eqref{E:onestep11} (recalling that $2\lambda = \lambda_1$) and \eqref{E:onestep13}.
\end{proof}


\section*{Appendix}
\renewcommand{\thesection}{A}
\setcounter{subsection}{0}
\setcounter{theorem}{0}
\setcounter{equation}{0}
\renewcommand{\thesubsection}{A.\arabic{subsection}}
\renewcommand{\theequation}{\thesection.\arabic{equation}}
\noindent
This short section collects various calculations that were relegated here from the main text of the paper. Specifically, we give proofs of Lemma~\ref{lemma-moments} and Lemma~\ref{lemma-approx}.

\subsection{Moment comparisons}
We begin by a comparison of the ranges of parameters for negative moments of $a_0(e)$ with the positive moments of~$T_e$:

\begin{proofsect}{Proof of Lemma~\ref{lemma-moments}}
Let~$q>0$ be such that $\E(a_0(e)^{-q})<\infty$. (Otherwise there is nothing to prove.) The assumption of separate ergodicity and the Pointwise Ergodic Theorem then imply
\begin{equation}
\frac1t\int_0^t\textd s\,\, a_s(e)^{-q}\,\,\underset{t\to\infty}\longrightarrow\,\,\E\bigl(a_0(e)^{-q}\bigr).
\end{equation}
Next fix~$M>0$ large. Renewal considerations show
\begin{equation}
\frac1t\int_0^t\textd s\,\, a_s(e)^{-q}\,\,\underset{t\to\infty}\longrightarrow\,\,\frac1{\E(T_e\wedge M)}\E\biggl(\,\int_0^{T_e\wedge M}\textd t a_t(e)^{-q}\biggr).
\end{equation}
It follows that
\begin{equation}
\label{E:1.21ui}
\E\biggl(\,\int_0^{T_e\wedge M}\textd t\, a_t(e)^{-q}\biggr) = \E(T_e\wedge M)\E\bigl(a_0(e)^{-q}\bigr).
\end{equation}
Next we note that H\"oler's inequality shows, for any $r>1$,
\begin{equation}
T_e\wedge M=\int_0^{T_e\wedge M}\textd t \,\,1
\le\Bigl(\int_0^{T_e\wedge M}\textd t\, a_t(e)\textd t\Bigr)^{1/r}\Bigl(\int_0^{T_e\wedge M}\textd t\,a_t(e)^{-\frac1{r-1}}\Bigr)^{\frac{r-1}r}
\end{equation}
The definition of~$T_e$ ensures that the first term on the right is at most~$1$. Raising both sides of the resulting bound to power $\frac{r}{r-1}$ and setting $r:=1+1/q$, which is equivalent to $\frac1{r-1}=q$ and $\frac{r}{r-1}=q+1$, then gives
\begin{equation}
\E\bigl((T_e\wedge M)^{q+1}\bigr)
\le \E\biggl(\,\int_0^{T_e\wedge M}\textd t \, a_t(e)^{-q}\biggr).
\end{equation}
Plugging in \eqref{E:1.21ui} and bounding $\E(T_e\wedge M)$ by the $\frac{1}{q+1}$-power of $\E((T_e\wedge M)^{q+1})$ shows
\begin{equation}
\E\bigl((T_e\wedge M)^{q+1}\bigr)
\le\Bigl[\E\bigl(a_0(e)^{-q}\bigr)\Bigr]^{\frac{q+1}q}.
\end{equation}
Taking $M\to\infty$ and invoking the Monotone Convergence Theorem, the claim follows.
\end{proofsect}

\subsection{Approximating corrector by gradients}
Our next task is to complete the proof of Lemma~\ref{lemma-approx} showing that the corrector lies in the closed subspace generated by gradients of~$L^p$-functions. In order to avoid dealing with complicated summation formulas, we will cast the proof in functional-analytic notation and language. 

Fix~$p\ge1$ such that the integrability conditions in \eqref{E:3.4} apply. For each~$k=1,\dots,d$, consider the linear operator $\hat T_k\colon L^p(\BbbP)\to L^p(\BbbP)$ defined by
\begin{equation}
\hat T_k f:=f\circ\tau_{0,\hat{e}_k},
\end{equation}
 with $\hat{e}_k$ the $k$-th unit vector in $\Z^d$. We also set
\begin{equation}
\hat T_{d+1}f:=\int_0^1\textd t f\circ\tau_{t,0}
\end{equation}
for the corresponding time-shift.
The operators $\hat T_1,\dots,\hat T_{d+1}$ commute and they are all contractions (by Assumption \ref{ass1}). For any~$\epsilon>0$ and $k=1,\dots, d+1$, the operator $(1+\epsilon-\hat T_k)^{-1}$ is well defined and can be expressed as $\sum_{n\ge0}(1+\epsilon)^{-n-1}\hat T_k^n$. Let $\hat P_k\colon L^p(\BbbP)\to L^p(\BbbP)$ be defined by the $L^p$-limit
\begin{equation}
\label{E:4.8ia}
\hat P_k f:=\lim_{n\to\infty}\frac1n\sum_{\ell=0}^{n-1}\hat T_k^\ell f
\end{equation}
which exists by the Pointwise Ergodic Theorem, see \cite[p.9, Thm. 2.3]{K85}; the fact that the convergence is in $L^p$ follows in standard fashion by uniform integrability. 
Rewriting $\hat T_k^n= A_k^{n+1}-A_k^n$ with $A_k^n f:= \sum_{\ell=0}^{n-1}\hat T_k^\ell f$, simple resummation shows
\begin{equation}
\epsilon(1+\epsilon-\hat T_k)^{-1} f = \sum_{n \geq 1}\frac{n\epsilon^2}{(1+\epsilon)^{n+1}} \frac{1}{n}A_k^nf\,.
\end{equation}
From \eqref{E:4.8ia} and $ \sum_{n \geq 1}\frac{n\epsilon^2}{(1+\epsilon)^{n+1}} =1$,
we thus have
\begin{equation}
\label{eq:erg1}
\epsilon(1+\epsilon-\hat T_k)^{-1} f\, \overset{L^p}{\underset{\epsilon\downarrow0}\longrightarrow}\,\hat P_k f,\qquad f\in L^p(\BbbP).
\end{equation}
for each~$k=1,\dots,d+1$. 

Next, consider the (vector) valued functions $u_1,\dots,u_{d+1}$ defined by
\begin{equation}
u_k:=\chi(0,\hat{e}_k,\cdot),\qquad k=1,\dots,d,
\end{equation}
and
\begin{equation}
u_{d+1}:=\int_0^1\textd t\,\chi(t,0,\cdot).
\end{equation}
The cocycle condition then reads
\begin{equation}
\label{E:4.12ui}
(1-\hat T_j)u_k = (1-\hat T_k)u_j,\qquad j,k=1,\dots,d+1,
\end{equation}
By the cocycle property and \eqref{E:4.8ia} we also have
\begin{equation}
\hat P_ku_k = \lim_{n\to\infty}\frac{\chi(0,n\hat{e}_k,\cdot)}n, \quad k=1,\dots,d+1.
\end{equation}
The cocycle property then also gives, for each~$j\ne k$ and each~$t\in\R$,
\begin{equation}
\chi(0,n\hat{e}_k,\cdot)\circ\tau_{t,e_j} = \chi(0,n\hat{e}_k,\cdot) + \chi(t,e_j,\cdot)\circ\tau_{0,n\hat{e}_k}-\chi(t,e_j,\cdot).
\end{equation}
Upon division by~$n$, the last two terms on the right tend to zero in~$L^p(\BbbP)$ and so~$\hat P_ku_k$ is invariant under space-time shifts. A completely analogous argument applies to $u_{d+1}$; in light of the joint ergodicity of~$\BbbP$ with respect to the space-time shifts we thus get
\begin{equation}
\label{E:4.18ui}
\hat P_k u_k = 0,\qquad k=1,\dots,d+1.
\end{equation}
We are now ready to give:

\begin{proofsect}{Proof of Lemma~\ref{lemma-approx}}
 Define $h_\epsilon$ by
\begin{equation}
\label{E:A.18}
h_\epsilon:=\sum_{k=1}^{d+1}\epsilon^{k-1}\prod_{j=1}^k(1+\epsilon-\hat T_j)^{-1} u_k.
\end{equation}
Pick $\ell=1,\dots,d+1$ and use \eqref{E:4.12ui} along with the fact that $\hat T_1,\dots,\hat T_{d+1}$ commute, minding also the rewrite $1-\hat T_\ell = (1+\epsilon-\hat T_\ell)-\epsilon$, to get
\begin{equation}
\begin{aligned}
(1-\hat T_\ell)h_\epsilon
&=\sum_{k=1}^{d+1}\epsilon^{k-1}(1-\hat T_\ell)\prod_{j=1}^k(1+\epsilon-\hat T_j)^{-1} u_k
\\
&=\sum_{k=1}^{d+1}\epsilon^{k-1}(1-\hat T_k)\prod_{j=1}^k(1+\epsilon-\hat T_j)^{-1} u_\ell
\\
&=\sum_{k=1}^{d+1}\biggl[\,\epsilon^{k-1}\prod_{j=1}^{k-1}(1+\epsilon-\hat T_j)^{-1}-
\epsilon^{k}\prod_{j=1}^{k}(1+\epsilon-\hat T_j)^{-1}\biggr]\,u_\ell
\\
&=u_\ell-\epsilon^{d+1}\prod_{j=1}^{d+1}(1+\epsilon-\hat T_j)^{-1} u_\ell,
\end{aligned}
\end{equation}
where the last line follows by noting that the expression on the line before is a telescopic sum. Since $\Vert\epsilon(1+\epsilon-\hat T_j)\Vert_{L^p\to L^p}\le1$ for each~$j=1,\dots,d+1$, the norm of second term on the last line is at most that of $\epsilon(1+\epsilon-\hat T_\ell)^{-1}u_\ell$. But this term converges to $\hat P_\ell u_\ell$ by \eqref{eq:erg1} which vanishes thanks to \eqref{E:4.18ui}. This implies that, for all~$\ell=1,\dots,d+1$,
\begin{equation}
(1-\hat T_\ell)h_\epsilon\,\underset{\epsilon\downarrow0}\longrightarrow\,u_\ell,\qquad \text{in }L^p(\BbbP),
\end{equation}
which is now easily checked to give the desired claim.
\end{proofsect}

\begin{remark}
Under the assumption of separate ergodicity --- i.e., triviality of~$\BbbP$ on events~$A$ such that, for at least one~$k=1,\dots,d+1$, we have $T_k1_A=1_A$ --- we have $P_ku_\ell=0$ for all~$k,\ell=1,\dots,d+1$. It then suffices to take $h_\epsilon:=(1+\epsilon-\hat T_1)u_1$; cf Biskup and Spohn~\cite{BS11}. However, unlike erroneously concluded in~\cite{BS11}, this does not suffice for~$\BbbP$ that are only jointly ergodic where one has to use \eqref{E:A.18} instead.
\end{remark}

\section*{Acknowledgments}
\noindent
This research has been partially supported by NSF grant DMS-1407558 and GA\v CR project P201/16-15238S. We wish to thank Jean-Dominique Deuschel for useful suggestions at various stages of this work and an anonymous referee for his careful reading of the manuscript. P-F.R.\ wishes to thank Center for Theoretical Study in Prague for hospitality that made this project take off the ground.

\bibliographystyle{abbrv}

\end{document}